\documentclass{article}

\usepackage[margin=1in]{geometry}
\usepackage{amsmath,mathtools,amssymb,amsthm}
\usepackage[all]{xy}
\usepackage{enumerate,paralist}
\usepackage{color}
 \usepackage{setspace}

\theoremstyle{definition}
\newtheorem*{theorem*}{Theorem}
\newtheorem{theorem}{Theorem}
\newtheorem{proposition}[theorem]{Proposition}
\newtheorem{lemma}[theorem]{Lemma}
\newtheorem{corollary}[theorem]{Corollary}
\newtheorem{definition}[theorem]{Definition}
\newtheorem{remark}[theorem]{Remark}
\newtheorem{conjecture}[theorem]{Conjecture}
\numberwithin{theorem}{section} 					% subsectionと連動、定理3.1.1など
\numberwithin{equation}{section}		 			% 等式はsectionと連動、式(3.5)

\newcommand*{\ilim}[1]{\mathop{\varinjlim}\limits_{#1}}		%帰納極限
\newcommand*{\plim}[1]{\mathop{\varprojlim}\limits_{#1}}		%射影極限

\newcommand*{\Ker}[0]{\text{Ker}}							%Ker
\newcommand*{\pro}[0]{\text{pro-}}							%pro
\newcommand*{\ab}[0]{\text{ab}}							%ab
\newcommand*{\sep}[0]{\text{sep}}							%sep
\newcommand*{\op}[0]{\text{op}}							%op
\newcommand*{\cl}[0]{\text{cl}}								%cl
\newcommand*{\pr}[0]{\text{pr}}							%pr
\newcommand*{\Hom}[0]{\text{Hom}}							%Hom
\newcommand*{\Isom}[0]{\text{Isom}}							%Isom
\newcommand*{\Aut}[0]{\text{Aut}}							%Aut
\newcommand*{\Out}[0]{\text{Out}}							%Out
\newcommand*{\Sect}[0]{\text{Sect}}							%Sect
\newcommand*{\GSect}[0]{\Sect^{\text{geom}}}					%GSect
\newcommand*{\QSect}[0]{\text{QSect}}						%QSect

\begin{document}

\title{The geometrically  $m$-step solvable Grothendieck conjecture\\ for affine hyperbolic curves over finitely generated fields}
\author{Naganori Yamaguchi\thanks{Affiliation: Research Institute for Mathematical Sciences, Kyoto University, e-mail: naganori@kurims.kyoto-u.ac.jp}}
\date{}
\maketitle

\begin{abstract}
In this paper, we present some new results on the geometrically  $m$-step solvable Grothendieck conjecture in
anabelian geometry. Specifically, we show the (weak bi-anabelian and strong bi-anabelian) geometrically  $m$-step solvable Grothendieck conjecture(s) for affine hyperbolic curves over fields finitely generated over the prime field. First of all, we show the conjecture over  finite fields. Next, we  show the geometrically  $m$-step solvable version of the Oda-Tamagawa good reduction criterion for hyperbolic curves.   Finally, by using these two results, we show the conjecture over fields finitely generated over the prime field.
\end{abstract}

\tableofcontents

%%%%%%%%%%%%%%%%%%%%%%%%%%%%%%%%%%%%%%%%%%%%%%%%%%%%%%%%%%%%%%%%%%%%%%%%%
%%%%%%%%%%%%%%%%%%%%%%%%%%%%%%%%%%%%%%%%%%%%%%%%%%%%%%%%%%%%%%%%%%%%%%%%%
%%%%%%%%%%%%%%%%%%%%%%%%%%%%%%%%%%%%%%%%%%%%%%%%%%%%%%%%%%%%%%%%%%%%%%%%%
%%%%%%%%%%%%%%%%%%%%%%%%%%%%%%%%%%%%%%%%%%%%%%%%%%%%%%%%%%%%%%%%%%%%%%%%%
%%%%%%%%%%%%%%%%%%%%%%%%%%%%%%%%%%%%%%%%%%%%%%%%%%%%%%%%%%%%%%%%%%%%%%%%%
%%%%%%%%%%%%%%%%%%%%%%%%%%%%%%%%%%%%%%%%%%%%%%%%%%%%%%%%%%%%%%%%%%%%%%%%%
%%%%%%%%%%%%%%%%%%%%%%%%%%%%%%%%%%%%%%%%%%%%%%%%%%%%%%%%%%%%%%%%%%%%%%%%%
%%%%%%%%%%%%%%%%%%%%%%%%%%%%%%%%%%%%%%%%%%%%%%%%%%%%%%%%%%%%%%%%%%%%%%%%%

\addcontentsline{toc}{section}{Introduction}
\section*{Introduction}
\hspace{\parindent}
In this introduction, we use  the following notation. Let $m$ be an integer greater than or equal to $1$. Let $k$ be a field. Set $p:=\text{ch}(k)(\geq 0)$. Let  $i=1$, $2$. Let  $X_{i}$ be a  proper, smooth, geometrically connected scheme of  relative dimension one over $k$ (we call such a scheme a  proper, smooth curve over $k$)  and $E_{i}$ a   closed subscheme of $X_{i}$ which is finite, \'{e}tale over $k$. Let  $g_{i}$ be the genus of $X_{i}$ and   $r_{i}$   the degree of $E_{i}$ over $k$. Set $U_{i}:=X_{i}-E_{i}$. We say that $U_{i}$  is  $hyperbolic$ if  $2-2g_{i}-r_{i}<0$. For a scheme $S$ (satisfying suitable conditions), we write $\pi_{1}(S,\overline{s})$ for the \'{e}tale fundamental group of $S$ and write $\pi_{1}^{\text{tame}}(S,\overline{s})$ for the tame fundamental group of $S$, where $\overline{s}$ stands for a geometric point of $S$. In the rest of this introduction,  we always fix a geometric point $\overline{s}$ and  write  $\pi_{1}(S)$, $\pi_{1}^{\text{tame}}(S)$ instead of   $\pi_{1}(S,\overline{s})$, $\pi_{1}^{\text{tame}}(S,\overline{s})$, respectively.\par
When $k$ is a field finitely generated over $\mathbb{Q}$ and  $U_{1}$ is hyperbolic,  we have a fundamental conjecture called  (the relative, weak bi-anabelian form of) the Grothendieck conjecture, which predicts:  If a $G_{k}$-isomorphism $\pi_{1}(U_{1})\xrightarrow{\sim}\pi_{1}(U_{2})$  exists, then  a $k$-isomorphism $U_{1}\xrightarrow{\sim}U_{2}$ exists. This conjecture was completely solved by \cite{Na1990-411}, \cite{Ta1997}, and \cite{Mo1999}. \par
Next, we consider variants of the Grothendieck conjecture by replacing  $\pi_{1}(U_{1})$ and $\pi_{1}(U_{2})$ with quotients of these profinite groups. In this paper, we  mainly consider various (geometrically) $m$-step solvable quotients.   Let $\ell$ be a prime.  For any profinite group $G$, we write $G^{[m]}$ for  the $m$-step derived subgroup of $G$ and $G^{\text{pro-}\ell'}$ for the maximal pro-prime-to-$\ell$ quotient of $G$. (We also write $G^{\text{pro-}0'}:=G$.) We define  $G^{m}:=G/G^{[m]}$, 
\begin{eqnarray*}
\pi_{1}^{\text{tame}}(U_{i})^{(m)}&:=&\pi_{1}^{\text{tame}}(U_{i})/\pi_{1}^{\text{tame}}(U_{i,k^{\text{sep}}})^{[m]}, \text{ and}\\
\pi_{1}^{\text{tame}}(U_{i})^{(m,\text{pro-}\ell')}&:=&\pi_{1}^{\text{tame}}(U_{i})/\text{Ker}(\pi_{1}^{\text{tame}}(U_{i,k^{\text{sep}}}) \rightarrow \pi_{1}^{\text{tame}}(U_{i,k^{\text{sep}}})^{m,\text{pro-}\ell'}).
\end{eqnarray*}
 For any  integer $n\in\mathbb{Z}_{\geq 0}$ satisfying  $m> n$, we write   $\Isom^{(m)}_{G_{k}}(\pi_{1}^{\text{tame}}(U_{1})^{(m-n)},\pi_{1}^{\text{tame}}(U_{2})^{(m-n)})$  for the image of the natural map \[\Isom_{G_{k}}(\pi_{1}^{\text{tame}}(U_{1})^{(m)},\pi_{1}^{\text{tame}}(U_{2})^{(m)})\rightarrow \Isom_{G_{k}}(\pi_{1}^{\text{tame}}(U_{1})^{(m-n)},\pi_{1}^{\text{tame}}(U_{2})^{(m-n)}).\] 
We also define $\pi_{1}(U_{i})^{(m)}$, $\pi_{1}(U_{i})^{(m,\text{pro-}\ell')}$, and $\Isom^{(m)}_{G_{k}}(\pi_{1}(U_{1})^{(m-n)},\pi_{1}(U_{2})^{(m-n)})$ by replacing $\pi_{1}^{\text{tame}}(U_{i})$ with $\pi_{1}(U_{i})$ in the above. \par
Let $\text{Sch}_{k}^{\text{geo.red.}}$ be the category of all geometrically reduced schemes over $k$.  When $p>0$, we  define $\text{Sch}_{k,\textbf{Fr}^{-1}}^{\text{geo.red.}}$ as the category  obtained by localizing  $\text{Sch}_{k}^{\text{geo.red.}}$ with respect to all relative Frobenius morphisms of geometrically reduced schemes over $k$.  We write $\mathfrak{S}_{k}$ for the category $\text{Sch}_{k}^{\text{geo.red.}}$ (resp. the category  $\text{Sch}_{k,\textbf{Fr}^{-1}}^{\text{geo.red.}}$) when $p=0$ (resp. when $p>0$). Note that the following equivalence  holds.
\begin{equation*}
 U_{1}\xrightarrow{\sim} U_{2}\text{ in }\mathfrak{S}_{k}\Longleftrightarrow\begin{cases}  U_{1}\xrightarrow[k]{\sim}U_{2}& (p=0) \\
U_{1}(n_{1})\xrightarrow[k]{\sim}U_{2}(n_{2})\text{ for some }n_{1},n_{2}\in \mathbb{Z}_{\geq 0}& (p>0)
\end{cases}
\end{equation*}
Here, $U_{i}(n_{i})$  stands for  the $n_{i}$-th Frobenius twist of $U_{i}$ over $k$. 
\par
In this paper, we consider  the following variants of the Grothendieck conjecture. 
\begin{conjecture}[The (relative, geometrically)  $m$-step solvable Grothendieck  conjecture]\label{conj1} Assume that $m\geq 2$,  that $k$ is a field finitely generated over the prime field, and that  $U_{1}$ is hyperbolic.  
\begin{enumerate}[(1)]
\item ($\text{W}_{m,U_{1},U_{2}}$: Weak bi-anabelian form) 
\begin{equation*}
\pi_{1}^{\text{tame}}(U_{1})^{(m)}\xrightarrow[G_k]{\sim} \pi_{1}^{\text{tame}}(U_{2})^{(m)}\ \Longleftrightarrow U_{1}\xrightarrow{\sim} U_{2}\text{ in }\mathfrak{S}_{k}
\end{equation*}
\item ($\text{S}_{m,n,U_{1},U_{2}}$: Strong bi-anabelian form)
 Assume that  $U_{1,\overline{k}}$ does not descend to a curve over $\overline{\mathbb{F}}_{p}$ when $p>0$.  Let $n\in\mathbb{Z}_{\geq 0}$ be an integer satisfying $m>n$. Then the following natural map  is bijective.
\begin{equation*}
\text{Isom}_{\mathfrak{S}_{k}}(U_1,U_2)\rightarrow \Isom^{(m)}_{G_{k}}(\pi_{1}^{\text{tame}}(U_{1})^{(m-n)},\pi_{1}^{\text{tame}}(U_{2})^{(m-n)})/\text{Inn}(\pi_{1}^{\text{tame}}(U_{2,k^{\text{sep}}})^{m-n})
\end{equation*}
\end{enumerate}
\end{conjecture}\noindent

\begin{remark}\label{remstr2}
Let $m'\in\mathbb{Z}_{\geq 2}$ be an integer satisfying $m'\geq m$. Then $\text{W}_{m,U_{1},U_{2}}$ implies $\text{W}_{m',U_{1},U_{2}}$. Hence we want to prove $\text{W}_{m,U_{1},U_{2}}$  for as small $m$ as possible.  The best expected result is for  $m = 2$. As for  $\text{S}_{m,n,U_{1},U_{2}}$,  the best expected result is for  $(m,n) = (2,0)$. 
\end{remark}

The following three theorems are  all the previous results  that the author knows about the weak  bi-anabelian and  strong bi-anabelian form of the  $m$-step solvable Grothendieck conjectures for hyperbolic curves. 

\begin{theorem}[cf. \cite{Na1990-405} Theorem A]\label{thmnakamura}
Assume that  $m\geq 2$ and  that  $k$  is  an  algebraic number field  satisfying one of the following conditions (a)-(b).
\begin{enumerate}[(a)]
\item $k$ is a quadratic field $\neq  \mathbb{Q}(\sqrt{2})$.
\item There exists a prime ideal $\mathfrak{p}$ of $O_{k}$ unramified in $k/\mathbb{Q}$ such that  $|O_k/\mathfrak{p}|=2$ (e.g., $k=\mathbb{Q}$).
\end{enumerate}
Let  $\lambda_{i}$ be an element of $k-\{0,1\}$ and  set $\Lambda_{i}:=\{0,1,\infty,\lambda_{i}\}$ for each $i=1,2$.  Then the following holds.
\begin{equation*}
\pi_{1}(\mathbb{P}^{1}_{k}-\Lambda_{1})^{(m)}\xrightarrow[G_{k}]{\sim} \pi_{1}(\mathbb{P}^{1}_{k}-\Lambda_{2})^{(m)}\iff  \mathbb{P}^{1}_{k}-\Lambda_{1}\xrightarrow[k]{\sim}\mathbb{P}^{1}_{k}-\Lambda_{2}
\end{equation*}
\end{theorem}

\begin{theorem}[cf. \cite{Mo1999} Theorem A${'}$]\label{thmmochizuki}
 Assume that $m\geq 5$, that $k$ is a field finitely generated over the prime field, and that   $U_{1}$ is  hyperbolic. Let $n\in \mathbb{Z}_{\geq 3}$ be an integer satisfying $m> n$.     Then the following natural map is bijective.
\begin{equation*}
\Isom_{k}(U_1,U_2)\rightarrow \Isom^{(m)}_{G_{k}}(\pi_{1}(U_{1})^{(m-n)},\pi_{1}(U_{2})^{(m-n)})/\text{Inn}(\pi_{1}(U_{2,k^{\text{sep}}})^{m-n})
\end{equation*}
 In particular, the following holds. 
 \begin{equation*}
\pi_{1}(U_{1})^{(m)}\xrightarrow[G_k]{\sim} \pi_{1}(U_{2})^{(m)}\ \Longleftrightarrow U_{1}\xrightarrow[k]{\sim} U_{2}
\end{equation*}
\end{theorem}

\begin{remark} More generally,   in \cite{Mo1999} Theorem A${'}$, Mochizuki proved a  certain Hom-version of the strong bi-anabelian form of the $m$-step solvable Grothendieck conjecture for hyperbolic curves over sub-$\ell$ adic fields (i.e., subfields of a finitely generated extension  field of $\mathbb{Q}_{\ell}$) for any prime $\ell$. 
\end{remark}

\begin{theorem}[cf. \cite{Ya2020} Theorem 2.4.1]\label{thmnyamaguchi}
Assume that $m\geq 3$, that $k$ is a field finitely generated over the prime field, that   $U_{1}$ is hyperbolic, and that $g_{1}=0$.  When  $p> 0$, assume that  the curve $X_{1,\overline{k}}-S$ does not  descend to a curve over  $\overline{\mathbb{F}}_{p}$  for each $S\subset E_{1, \overline{k}}\ \text{with}\ |S|=4$. Then the following holds.
\[
\pi_{1}(U_{1})^{(m,\text{pro-}p')}\xrightarrow[G_k]{\sim} \pi_{1}(U_{2})^{(m,\text{pro-}p')}\iff U_{1}\xrightarrow{\sim}U_{2}\text{ in }\mathfrak{S}_{k}
\]
\end{theorem}

 In this paper, we  give some new results on the weak  bi-anabelian and strong  bi-anabelian form of the  $m$-step solvable Grothendieck conjectures for hyperbolic curves over fields finitely generated over the prime field, by referring to the methods of \cite{Ta1997}  and \cite{St2002A} (and \cite{St2002P} in part).  First, we consider   the case that the base field is  finite (see section 2).

\begin{theorem}[Theorem \ref{finGCweak}, Corollary \ref{finGCcorinn}]\label{ddd}
For $i=1,$ $2$,   let $k_{i}$ be a finite field. Let $X'_{i}$ be a  proper, smooth curve over $k_{i}$, $E'_{i}$ a   closed subscheme of $X_{i}'$ which is finite, \'{e}tale over $k_{i}$. Let   $g'_{i}$ be the genus of $X'_{i}$ and   $r'_{i}$   the degree of $E'_{i}$ over $k_{i}$. Set $U'_{i}:=X'_{i}-E'_{i}$. Assume  that $U'_{1}$ is affine hyperbolic. 
\begin{enumerate}[(1)]
\item Assume that  $m$ satisfies
\begin{equation*}
\begin{cases}
\ \ m\geq 2 & \ \ (\text{if }r'_{1}\geq  3 \text{ and } (g'_{1},r'_{1})\neq (0,3), (0,4))\\
\ \ m\geq 3 & \ \ (\text{if }r'_{1}<3 \text{ or } (g'_{1},r'_{1})= (0,3), (0,4))
\end{cases}
\end{equation*} Then  the following holds.
\begin{equation*}
\pi_{1}^{\text{tame}}(U'_{1})^{(m)}\xrightarrow{\sim} \pi_{1}^{\text{tame}}(U'_{2})^{(m)}\iff U'_{1}\xrightarrow[\text{scheme}]{\sim}U'_{2}
\end{equation*}
\item  Assume that $m\geq 3$. Let $n\in\mathbb{Z}_{\geq 2}$ be an integer satisfying $m> n$. Then the  following natural map is bijective.
\begin{equation*}\label{mapddd2}
\Isom(U'_1,U'_2)\rightarrow \Isom^{(m)}(\pi_{1}^{\text{tame}}(U'_{1})^{(m-n)},\pi_{1}^{\text{tame}}(U'_{2})^{(m-n)})/\text{Inn}(\pi_{1}^{\text{tame}}(U'_{2})^{(m-n)})
\end{equation*}
Here,  $\Isom^{(m)}(\pi_{1}^{\text{tame}}(U'_{1})^{(m-n)},\pi_{1}^{\text{tame}}(U'_{2})^{(m-n)})$ stands for  the image of the map \\%はみ出し防止のため改行
$\Isom(\pi_{1}^{\text{tame}}(U'_{1})^{(m)},\pi_{1}^{\text{tame}}(U'_{2})^{(m)})\rightarrow \Isom(\pi_{1}^{\text{tame}}(U'_{1})^{(m-n)},\pi_{1}^{\text{tame}}(U'_{2})^{(m-n)})$, see Definition \ref{m+nletter}. 
 \end{enumerate}
\end{theorem}
\noindent Theorem \ref{ddd} is a completely new result and even the first result on the $m$-step solvable Grothendieck conjecture for hyperbolic curves over finite fields. Next, we consider  the case that $k$ is a field finitely generated over the prime field  (see section 4). 

\begin{theorem}[Theorem \ref{fingeneGCweak}, Corollary \ref{fingeneGCcorinn2}]\label{ccc}
Assume that $k$ is a field finitely generated over the prime field  and  that $U_{1}$ is  affine hyperbolic.  Assume that $U_{1,\overline{k}}$ does not  descend to a curve over  $\overline{\mathbb{F}}_{p}$  when $p>0$.  
\begin{enumerate}[(1)]
\item Assume that  $m$ satisfies
\begin{equation*}
\begin{cases}
\ \ m\geq 4 & \ \ (\text{if }r_{1}\geq  3 \text{ and } (g_{1},r_{1})\neq (0,3), (0,4))\\
\ \ m\geq 5 & \ \ (\text{if }r_{1}<3 \text{ or } (g_{1},r_{1})= (0,3), (0,4)).
\end{cases}
\end{equation*}
 Then  the following holds.
\begin{equation*}
\pi_{1}^{\text{tame}}(U_{1})^{(m)}\xrightarrow[G_{k}]{\sim} \pi_{1}^{\text{tame}}(U_{2})^{(m)}\Longleftrightarrow U_{1}\xrightarrow{\sim}U_{2}\text{ in }\mathfrak{S}_{k}
\end{equation*}
\item Assume that $m\geq 5$. Let $n\in\mathbb{Z}_{\geq 4}$ be an integer satisfying $m>n$. Then the following natural  map is bijective.
\begin{equation*}\label{map1234}
\Isom_{\mathfrak{S}_{k}}(U_1,U_2)\rightarrow \Isom^{(m)}_{G_{k}}(\pi_{1}^{\text{tame}}(U_{1})^{(m-n)},\pi_{1}^{\text{tame}}(U_{2})^{(m-n)})/\text{Inn}(\pi_{1}^{\text{tame}}(U_{2,k^{\text{sep}}})^{m-n})
\end{equation*}
  \end{enumerate}
\end{theorem}

\noindent
The following is a summary of the new results contained in Theorem \ref{ccc} that are not covered by the previous results Theorem \ref{thmnakamura}, Theorem \ref{thmmochizuki}, and Theorem \ref{thmnyamaguchi}.

\begin{theorem*}[Summary of new results contained  in Theorem \ref{ccc}]\label{theoremtt}
Assume that  $k$ is finitely generated over the prime field and that   $U_{1}$ is   affine  hyperbolic. Assume that $U_{1,\overline{k}}$ does not  descend to a curve over  $\overline{\mathbb{F}}_{p}$ when $p>0$.
\begin{enumerate}[(1)]
\item We assume one of the following (a)-(d).
\begin{enumerate}[(a)]
\item $p=0$, $r_{1}\geq 3$, $g_{1}\geq 1$, and $m= 4$.
\item $p>0$, $r_{1}\geq 3$, $g_{1}\geq 1$, and $m\geq  4$.
\item $p>0$, $r_{1}< 3$, and $m\geq 5$.
\item $p>0$, $g_{1}=0$, $r_{1}\geq 5$, $m\geq 4$, and the curve  $X_{1,\overline{k}}-S$    descends to a curve over  $\overline{\mathbb{F}}_{p}$ for some $S\subset E_{1, \overline{k}}\ \text{with}\ |S|=4$.
\end{enumerate}
 Then  the following holds.
\begin{equation*}
\pi_{1}^{\text{tame}}(U_{1})^{(m)}\xrightarrow[G_{k}]{\sim} \pi_{1}^{\text{tame}}(U_{2})^{(m)}\Longleftrightarrow U_{1}\xrightarrow{\sim}U_{2}\text{ in }\mathfrak{S}_{k}
\end{equation*}
\item Assume that $p>0$ and that $m\geq 5$. Let $n\in\mathbb{Z}_{\geq 4}$ be an integer satisfying $m>n$. Then the following map is bijective.
\begin{equation*}
\Isom_{\mathfrak{S}_{k}}(U_1,U_2)\rightarrow \Isom^{(m)}_{G_{k}}(\pi_{1}^{\text{tame}}(U_{1})^{(m-n)},\pi_{1}^{\text{tame}}(U_{2})^{(m-n)})/\text{Inn}(\pi_{1}^{\text{tame}}(U_{2,k^{\text{sep}}})^{m-n})
\end{equation*}
  \end{enumerate}
\end{theorem*}

Let us sketch the proofs of Theorem \ref{ddd} and Theorem \ref{ccc}.  For simplicity, we also write  $U_{i}$  for $U'_{i}$ (in Theorem \ref{ddd}).  
Roughly speaking, the proof of Theorem \ref{ddd} (resp.  Theorem \ref{ccc}) is based on \cite{Ta1997} sections 2, 4 (resp. \cite{Ta1997} sections 5, 6, and \cite{St2002A}). However, our proofs  differ from those in \cite{Ta1997} and  \cite{St2002A} in the following point, among other things.

\begin{itemize}
\item[(P)] We  need to replace various arguments in \cite{Ta1997} and \cite{St2002A} (that involve the full (tame) fundamental group $\pi_{1}^{\text{tame}}(U_{i})$) with new arguments that only involve the (geometrically) $m$-step solvable quotient $\pi_{1}^{\text{tame}}(U_{i})^{(m)}$. Further, we also need to have these new arguments for as small $m$ as possible. (See Remark \ref{remstr2}.)\par
\end{itemize}
\noindent
Let us divide the proofs into seven steps. In all steps, we need to treat carefully  the difficulties that come from (P).
\\
\ \\
\noindent\underline{The Sketch of Proof of Theorem \ref{ddd}}
\begin{itemize}
\item[(Step 1: contents in subsections 2.1 and 2.2)] We reconstruct   the $\pi_{1}^{\text{tame}}(U_{i} )^{(m-1)}$-set $\text{Dec }(\pi_{1}^{\text{tame}}(U_{i} )^{(m-1)})$  from $\pi_{1}^{\text{tame}}(U_{i})^{(m)}$  (see Proposition \ref{bidecoabel}). In this step,  we always face the difficulty that comes from (P) (for example,  when proving the separatedness of  decomposition groups of $\pi_{1}^{\text{tame}}(U_{i} )^{(m)}$ (see Lemma \ref{seplemmacase10} and Proposition \ref{sepprop}) and  when discussing how to get the result for the reconstruction of the $\pi_{1}^{\text{tame}}(U_{i} )^{(m-1)}$-set $\tilde{U_{i} }^{m-1 ,\text{cl}}$, where  $\tilde{U}^{m-1}_{i}$ is the maximal unramified covering of $U_{i}$ which  is tamely ramified outside of $U_{i}$ and a (geometrically)    $(m-1)$-step  solvable covering of $U_{i}$   (see Lemma \ref{reconstructiondecoabel})). 

\item[(Step 2: contents in subsection 2.3)] We reconstruct the curve $U_{i} $ from  $\pi_{1}^{\text{tame}}(U_{i} )^{(m)}$ and the $\pi_{1}^{\text{tame}}(U_{i} )^{(m-1)}$-set $\text{Dec }(\pi_{1}^{\text{tame}}(U_{i} )^{(m-1)})$. The basic plan is to reconstruct the multiplicative group and the addition of the function field $K(U_{i} )$.   For the first reconstruction,  we use   class field theory, and for the second reconstruction,  we use  Lemma \ref{Talem} (\cite{Ta1997} Lemma 4.7).  Thus, by using Step 1 and Step 2,  Theorem \ref{ddd}(1) follows. 

\item[(Step 3: contents in subsection 2.4)] In this step, we prove  Theorem \ref{ddd}(2). To prove the  injectivity,  we use Lemma \ref{stxlemma1} (\cite{St2002P} Theorem 1.2.1).  To prove the surjectivity, we use the results obtained in Step 1 and Step 2.  \hspace{0pt plus 1 filll}$\Box$
\end{itemize}
\ \\\ \\
\noindent\underline{The Sketch of Proof of Theorem \ref{ccc}}
\begin{itemize}
\item[(Step 4: contents in section 3)]  Let $R$ be a regular local ring, $s$ the closed point of $\text{Spec}(R)$, and $(X,E)$ a  hyperbolic curve over the function field $K:=K(R)$.  Set $U:=X-E$. Let $I$ be an inertia group of $G_{K}$ at $s$. To show Theorem \ref{ccc}, we need    Theorem \ref{ddd}(2) and the following   results on the $m$-step solvable version of the Oda-Tamagawa good reduction criterion (\cite{Ta1997} Theorem (5.3)). 
\end{itemize}
\begin{theorem}[Theorem \ref{2goodreductiontheorem}]\label{edx}
 Assume that $R$ is a discrete valuation ring and that $m\geq2$. Then  $(X,E)$ has good reduction at $s$ if and only if  the image of $I $ in $\Out(\pi_{1}^{\text{tame}}({U}_{K^{\text{sep}}})^{m,\pro \text{ch}(\kappa(s))'})$ is trivial.
\end{theorem}

\begin{corollary}[Corollary \ref{goodcor}]\label{caseq}
Assume that  $R$ is a  henselian regular local ring.    Let $(\mathfrak{X},\mathfrak{E})$ be a smooth model of $(X,E)$ over $\text{Spec}(R)$. Set $\mathfrak{U}:=\mathfrak{X}-\mathfrak{E}$.    Then  $\pi_{1}^{\text{tame}}(\mathfrak{U}_{s})^{(m-2)}\xleftarrow{\sim}\plim{H}\pi_{1}^{\text{tame}}(U)^{(m)}/H$ holds, where $H$ runs over  all open normal subgroups of $\pi_{1}^{\text{tame}}(U)^{(m)}$  satisfying (i) $ \pi_{1}^{\text{tame}}(U_{K^{\text{sep}}})^{[m-2]}/\pi_{1}^{\text{tame}}(U_{K^{\text{sep}}})^{[m]}\subset H$,   (ii) the image of $H$ in $G_{K}$ contains $I $, and  (iii) the image of $I $ in $\Out((H\cap \pi_{1}^\text{tame}(U_{K^{\text{sep}}}))^{2,\pro \text{ch}(\kappa(s))'})$ is trivial. 
\end{corollary}

\begin{itemize}
\item[(Step 5: contents in subsection \ref{subsection4.1})] We investigate the properties of the category $\text{Sch}_{k,\textbf{Fr}^{-1}}^{\text{geo.red.}}$. In \cite{St2002A} and \cite{St2002P},  to extend the arguments in \cite{Ta1997}  sections 5, 6   to the  positive characteristic case,  the category obtained by localizing the category of varieties over $k$ with respect to all relative Frobenius morphisms of varieties over $k$ was introduced.  In this step, we need to consider not only varieties over $k$ but also arbitrary geometrically reduced $k$-schemes to apply the argument to $\tilde{U}^{m}_{i}$.

\item[(Step 6: contents in subsection \ref{subsectionfingeneweak})] Fix an isomorphism $\alpha: \pi_{1}^{\text{tame}}(U_{1})^{(m)}\xrightarrow{\sim}\pi_{1}^{\text{tame}}(U_{2})^{(m)}$ ($m\geq 5$).  By Galois descent theory, we only need to consider the case that the Jacobian variety of  $X_{1}$ has a level $N$ structure and $E_{i}$ consists of  $k$-rational points.   Let  $S$  be an integral regular scheme of finite type over $\text{Spec}(\mathbb{Z})$ with function field $k$. 
 By replacing $S$ with a suitable open subscheme  if necessary, we may assume that  there exists a smooth  curve $(\mathcal{X}_{i},\mathcal{E}_{i})$ over $S$ whose generic fiber is isomorphic to (and identified with) $(X_{i},E_{i})$.  Let  $\zeta_{i}: S\rightarrow \mathcal{M}_{g,r}[N]$ be  the morphism  classifying   $(\mathcal{X}_{i},\mathcal{E}_{i})$ (with a suitable ordering of $\mathcal{E}_{i}$ and  a suitable level $N$ structure).  First, we show the claim: $\zeta_{1}$ and $\zeta_{2}$ coincide (up to composition with a power of the absolute Frobenius of $S$ when $p>0$) (see Lemma \ref{stxlemma2fin}).  By Lemma \ref{stxlemma1} (\cite{St2002P} Theorem 1.2.1),  it is sufficient to show that   $\zeta_{1}$ and $\zeta_{2}$  coincide set-theoretically. This is shown by the contents of Step 3 and  Step 4.  Hence the claim  follows. By using the claim, Theorem \ref{ccc}(1) follows.

\item[(Step 7: contents in subsection \ref{subsectionfingenestrong})]In this step, we prove Theorem \ref{ccc}(2).  To prove  the injectivity,  we use the center-freeness of $\pi_{1}(U_{i})^{(m)}$ (Proposition \ref{center}).  To prove the surjectivity, we use the result proved in Step 6.  In this proof,  we must  be careful about the  number of Frobenius twists (see the proof of  Lemma \ref{lemmaaaaa}).\hspace{0pt plus 1 filll}$\Box$
\end{itemize}

%%%%%%%%%%%%%%%%%%%%%%%%%%%%%%%%%%%%%%%%%%%%%%%%%%%%%%%%%%%%%%%%%%%%%%%%%
%%%%%%%%%%%%%%%%%%%%%%%%%%%%%%%%%%%%%%%%%%%%%%%%%%%%%%%%%%%%%%%%%%%%%%%%%
%%%%%%%%%%%%%%%%%%%%%%%%%%%%%%%%%%%%%%%%%%%%%%%%%%%%%%%%%%%%%%%%%%%%%%%%%
%%%%%%%%%%%%%%%%%%%%%%%%%%%%%%%%%%%%%%%%%%%%%%%%%%%%%%%%%%%%%%%%%%%%%%%%%
%%%%%%%%%%%%%%%%%%%%%%%%%%%%%%%%%%%%%%%%%%%%%%%%%%%%%%%%%%%%%%%%%%%%%%%%%
%%%%%%%%%%%%%%%%%%%%%%%%%%%%%%%%%%%%%%%%%%%%%%%%%%%%%%%%%%%%%%%%%%%%%%%%%
%%%%%%%%%%%%%%%%%%%%%%%%%%%%%%%%%%%%%%%%%%%%%%%%%%%%%%%%%%%%%%%%%%%%%%%%%
%%%%%%%%%%%%%%%%%%%%%%%%%%%%%%%%%%%%%%%%%%%%%%%%%%%%%%%%%%%%%%%%%%%%%%%%%

\addcontentsline{toc}{section}{Notation}
\section*{Notation}
In the rest of  this paper, we use the following notation.

\begin{enumerate}[(a)]

	\item \label{mdef}
We fix  an integer $m\in\mathbb{Z}_{\geq1}$. Remark that $m$ is {\bf always greater than or equal to ${\bf1}$} by definition.  

	\item \label{centralizer}
Let  $G$ be a profinite group.  Then we write $H\overset{\text{op}}\subset  G$ (resp. $H\overset{\text{cl}}\subset G$)  if   $H$ is an open (resp. a closed) subgroup of $G$. We define  $Z(G)$ as the center of $G$ and define  $Z_{G}(H)$ as the centralizer of $H$ in $G$ for any $H\overset{\text{cl}}\subset G$.

	\item \label{mstepsolvabledef}
Let  $G$ be a profinite group. Let  $w\in\mathbb{Z}_{\geq0}$ be an integer.   Then we write  $\overline{[G,G]} $ for the closed subgroup of $G$ which is (topologically) generated by the commutator subgroup of $G$. We set $G^{[0]}:=G$ and $G^{[w]}:=\overline{[G^{[w-1]},G^{[w-1]}]}$ ($w\geq 1$). The group  $G^{w}:=G/G^{[w]}$ is called the maximal $w$-step solvable quotient of $G$.  Let    $\Sigma$ be a set of primes.  We write  $G^{\Sigma}$ for the maximal pro-$\Sigma$ quotient of $G$.  We set  $G^{w,\Sigma}:=(G^{w})^{\Sigma}$. For   a prime $\ell$, we write ``$\pro\ell$" (resp.``$\pro\ell'$") instead of  ``$\Sigma$"  when $\Sigma=\{\ell\}$ (resp. $\Sigma$ is the set of all primes  different from $\ell$). 
	\item 
Let $S$ be a  scheme. We denote by $S^{\cl}$ the set of all closed points of $S$.
	
	\item 
Let  $k$ be a field.  Then we write  $\overline{k}$ for an algebraic  closure of $k$ and   $k^{\text{sep}}$ for the  maximal
separable extension of $k$ contained in $\overline{k}$. We set $G_{k}:=\text{Gal}(k^{\text{sep}}/k)$. When $k$ is a finite field, we write $\text{Fr}_{k}\in G_{k}$ for the Frobenius element of $k$.

	\item \label{notationcorvesovert}
Let $S$ be a scheme, $\mathcal{X}$ a scheme over $S$, $\mathcal{E}$ a (possibly empty) closed subscheme of $\mathcal{X}$, and  $(g,r)$ a pair of non-negative integers.  Then we say that the pair $(\mathcal{X},\mathcal{E})$ is a $smooth$ $curve$ ($of$ $type$ $(g,r)$) over $S$ if the following conditions hold.
\begin{itemize}
\item  $\mathcal{X}$ is  smooth, proper, and of  relative dimension one over $S$.
\item  For any geometric point $\overline{s}$   of  $S$,  the  geometric fiber $\mathcal{X}_{\overline{s}}$  at $\overline{s}$ is connected and  satisfies $\text{dim}(H^{1}(\mathcal{X}_{\overline{s}},\mathcal{O}_{\mathcal{X}_{\overline{s}}}))=g$.
\item  The composite of $\mathcal{E}\hookrightarrow  \mathcal{X}\rightarrow S$ is finite, \'{e}tale and of degree $r$. 
\end{itemize}
If there is no risk of confusion, we also call  the complement $\mathcal{U}:=\mathcal{X}-\mathcal{E}$ a  smooth curve over $S$ (of type $(g,r)$).  We  write $g(\mathcal{U})$ and $r(\mathcal{U})$ for  $g$ and $r$, respectively.  We say that a smooth curve $\mathcal{U}$ of type $(g,r)$ is  $hyperbolic$ if  $2-2g-r<0$ (in other words, $(g,r)\neq (0,0),(0,1),(0,2),(1,0)$).

\end{enumerate}

\noindent In the following  (\ref{notaionoverlineH})-(\ref{notationtate}), let $(X,E)$ be a smooth curve  over  a field $k$, $U:=X-E$,  $K(U_{k^{\sep}})$ the function field of $U_{k^{\sep}}$, $\Omega$ an algebraically closed field  containing $K(U_{k^{\sep}})$, and   $\overline{\eta}:\text{Spec}(\Omega)\rightarrow U_{k^{\sep}}(\rightarrow U)$ the corresponding  geometric point.   Let $\Sigma$ be a set of primes.

\begin{enumerate}[(a)]
\setcounter{enumi}{6}
	\item \label{notaionoverlineH}
We set
\begin{equation*}
\Pi_{U}:=\pi_1^{\text{tame}}(U,\overline{\eta})\ \ \ \text{and}\  \ \ \overline{\Pi}_{U}:=\pi_1^{\text{tame}}(U_{k^{\sep}},\overline{\eta}).
\end{equation*}
Let $G$ be a quotient of $\Pi_{U}$,  defined  by a surjection   $\rho:\Pi_{U}\twoheadrightarrow G$. Let   $H$ be an closed subgroup of $G$.    Let    $w\in\mathbb{Z}_{\geq0}$ be an integer.  Then we set 
\[\overline{H}:= H\cap \rho(\overline{\Pi}_{U}),\ 
\ H^{(\Sigma)}:=H/\Ker (\overline{H}\twoheadrightarrow\overline{H}^{\Sigma}),\ \ H^{(w)}:=H/\overline{H}^{[w]}, \ \ \text{and}\ \ H^{(w,\Sigma)}:=H/\Ker (\overline{H}\twoheadrightarrow\overline{H}^{w,\Sigma}).
\]
 For   a prime $\ell$, we write ``$\pro\ell$" (resp.``$\pro\ell'$") instead of  ``$\Sigma$"  when $\Sigma=\{\ell\}$ (resp. $\Sigma$ is the set of all primes different from $\ell$). 

	\item \label{surfacegroup}
Let  $(g,r)$ be a pair of non-negative integers. We write $\Pi_{g,r}$ for the group

\begin{equation}\label{equation2.1}
\left\langle \alpha_1,\cdots,\alpha_g,\beta_1,\cdots,\beta_g,\sigma_1,\cdots,\sigma_r\middle|\prod_{i=1}^{g}[\alpha_i,\beta_i]\prod_ {j=1}^{r}\sigma_j=1\right\rangle, 
\end{equation}
and  $\hat{\Pi}_{g,r}$ for the profinite completion of  $\Pi_{g,r}$.  Assume that   $\Sigma$   contains  a prime different from $\text{ch}(k)$. Set  $\Sigma^{\dag}:=\Sigma-\{\text{ch}(k)\}$. Then the  existence of surjections $ \hat{\Pi}^{\Sigma}_{g,r}\twoheadrightarrow \overline{\Pi}_{U}^{\Sigma}\twoheadrightarrow \hat{\Pi}^{\Sigma^{\dag}}_{g,r}$ (see \cite{SGA1}) implies   the following equivalences (see \cite{Ya2020}).
\begin{equation}\label{trivialcong2}
\overline{\Pi}^{m,\Sigma}_{U}\text{ is not trivial}\Leftrightarrow  (g,r)\neq(0,0),(0,1)
\end{equation}
\begin{equation}\label{trivialcong}
\overline{\Pi}^{m,\Sigma}_{U}\text{ is not abelian}\Leftrightarrow  m\geq 2 \text{ and }(g,r)\neq(0,0),(0,1),(1,0),(0,2).
\end{equation}

	\item \label{univcovers}
We define  $\tilde{\mathcal{K}}(U)\subset\Omega$ (resp. $\tilde{\mathcal{K}}(U)^{\Sigma}\subset\Omega$)  as the maximal tamely ramified Galois (resp. pro-$\Sigma$ Galois) extension  of  $K(U_{k^{\sep}})$ in $\Omega$ unramified on $U$. We write $\tilde{U}\ (=\tilde{U}^{U})$ and $\tilde{X}\ (=\tilde{X}^{U})$ (resp. $\tilde{U}^{\Sigma}\ (=\tilde{U}^{U,\Sigma})$ and $\tilde{X}^{\Sigma}\ (=\tilde{X}^{U,\Sigma})$) for  the integral closures of $U$ and $X$ in $\tilde{\mathcal{K}}(U)$ (resp. $\tilde{\mathcal{K}}(U)^{\Sigma}$), respectively. We denote $\tilde{X}-\tilde{U}$ (resp. $\tilde{X}^{\Sigma}-\tilde{E}^{\Sigma}$)  by $\tilde{E}\ (=\tilde{E}^{U})$ (resp. $\tilde{E}^{\Sigma}\ (=\tilde{E}^{U,\Sigma})$). Let $G$ be a quotient of $\Pi_{U}$,  defined  by a surjection   $\rho:\Pi_{U}\twoheadrightarrow G$.  Let $H$  be a closed subgroup  of  $G$. We write $U_{H}:=\rho^{-1}(H)\backslash\tilde{U}$,   $X_{H}:=\rho^{-1}(H)\backslash\tilde{X}$ and $E_{H}:=\rho^{-1}(H)\backslash\tilde{E}$. For   a prime $\ell$, we write ``$\pro\ell$" (resp.``$\pro\ell'$") instead of  ``$\Sigma$"  when $\Sigma=\{\ell\}$ (resp. $\Sigma$ is the set of all primes different from $\ell$). 

	\item \label{univcoversms}
  Let  $w\in\mathbb{Z}_{\geq0}$ be an integer.  Then we define $\tilde{\mathcal{K}}^{w}(U)$ (resp. $\tilde{\mathcal{K}}^{w,\Sigma}(U)$)  as the maximal tamely ramified $w$-step   (resp.  pro-$\Sigma$ $w$-step)  solvable Galois extension of $K(U_{k^{\sep}})$ in $\tilde{\mathcal{K}}(U)$.  We write $\tilde{U}^{w}$  and  $\tilde{X}^{w}$ (resp.  $\tilde{U}^{w,\Sigma}$   and  $\tilde{X}^{w,\Sigma}$) for the integral closures of $U$ and $X$ in $\tilde{\mathcal{K}}^{w}(U)$ (resp.   $\tilde{\mathcal{K}}^{w,\Sigma}(U)$). We denote  $\tilde{X}^{w}- \tilde{U}^{w}$ (resp.  $\tilde{X}^{w,\Sigma} -\tilde{U}^{w,\Sigma}$) by $\tilde{E}^{w}$ (resp.  $\tilde{E}^{w,\Sigma}$). For   a prime $\ell$, we write ``$\pro\ell$" (resp.``$\pro\ell'$") instead of  ``$\Sigma$"  when $\Sigma=\{\ell\}$ (resp. $\Sigma$ is the set of all primes different from $\ell$).      \par

	\item \label{notationfundgroup} 
Let $Z$ be a normal integral scheme, $K(Z)$ the function field of $Z$, and $L$ a Galois  extension of $K(Z)$. Then we write $\tilde{Z}^{L}$  for the integral closure of $Z$ in $L$. Let $\tilde{v}\in (\tilde{Z}^{L})^{\text{cl}}$ be a closed point. Then  we define $D_{\tilde{v}}:=D_{\tilde{v},\text{GaL}(L/K(Z))}$  (resp. $I_{\tilde{v}}:=I_{\tilde{v},\text{GaL}(L/K(Z))}$) as  the  subgroup $\{\gamma\in \text{Gal}(L/K(Z))\mid \gamma(\tilde{v})=\tilde{v}\}$ (resp. $\{\gamma\in \text{Gal}(L/K(Z))\mid \gamma(\tilde{v})=\tilde{v},  \gamma \text{ acts trivially on }\kappa(\tilde{v})\}$) of Gal$(L/K(Z))$.  We call it the decomposition group (resp. inertia group) at $\tilde{v}$.  We define $\text{Dec}(\text{Gal}(L/K(Z)))$ (resp.  $\text{Iner}(\text{Gal}(L/K(Z)))$) as the $\text{Gal}(L/K(Z))$-set of all decomposition  groups (resp. inertia groups) of $\text{Gal}(L/K(Z))$. We write $I_{ \text{Gal}(L/K(Z))}$  for  the subgroup of $\text{Gal}(L/K(Z))$  (topologically) generated by all inertia groups.    For  $w\in\mathbb{Z}_{\geq0}$, we define  $\tilde{X}^{X,w}:=X_{I_{\overline{\Pi}_{U}^{w}}}$.

	\item \label{notationtate}
 Let $A$ be a semi-abelian variety over $k$.   Then we  write  $T_{\Sigma}(A)$ for the  pro-$\Sigma$ Tate module of $A$.  We write $T(A)$ instead of $T_{\Sigma}(A)$ when  $\Sigma$ is the set of all primes.  For a prime $\ell$, we write $T_{\ell}(A)$ (resp. $T_{\ell'}(A)$) instead of $T_{\Sigma}(A)$ when  $\Sigma=\{\ell\}$  (resp. $\Sigma$ is the set of all primes different from $\ell$).    We write $J_{X}$ for the Jacobian variety of $X$.

	\item \label{eee}
Let  $S_i$ be a  scheme and    $T_{i}$  a scheme over $S_{i}$ for $i=1$, $ 2$.   Then we define  $\text{Isom}(T_{1}/S_{1},T_{2}/S_{2})$ as the set
\[
\left\{(\tilde{F},F)\in \Isom(T_{1},T_{2})\times  \Isom(S_{1},S_{2})\middle|\vcenter{
\xymatrix{
T_{1}\ar[r]^{\tilde{F}}\ar[d]& T_{2}\ar[d]\\
S_{1}\ar[r]^{F} & S_{2}
} }\text{ is commutative.}
\right\}.
\]
	\item \label{homnotation2}
 Let  $k$ be a field and $L$  an extension of $k$. Let  $S_{i}$   be  a scheme over $k$,  $T_{i}$   a scheme over $L$, and   $T_{i}\rightarrow S_{i}$ a morphism over $k$ for $i=1$, $ 2$.   Then we define  $\text{Isom}_{L/k}(T_{1}/S_{1},T_{2}/S_{2})$ as the set
\[
\left\{(\tilde{F},F)\in \Isom_{L}(T_{1},T_{2})\times  \Isom_{k}(S_{1},S_{2})\middle|\vcenter{
\xymatrix{
T_{1}\ar[r]^{\tilde{F}}\ar[d]& T_{2}\ar[d]\\
S_{1}\ar[r]^{F} & S_{2}
} }\text{ is commutative.}
\right\}.
\]

\item\label{Frobeniusdef}
For  a scheme $S$ over $\mathbb{F}_p$, we write $\text{Fr}_{S}:S\rightarrow S$ for the morphism with the identity map on the underlying topological space and the $p$-th power endomorphism on the structure sheaf  and   call it the  $absolute$ $Frobenius$  $morphism$ $of$ $S$.   For  a scheme $T$ over $S$, we consider the following commutative diagram.
\begin{equation*}
\vcenter{
\xymatrix@C=46pt{
T\ar@/^10pt/[rrd]^{\text{Fr}_{T}}\ar@/_10pt/[rdd]	\ar@{.>}[rd]|{\text{Fr}_{T/S}}&&\\
&	T(1) \ar@{.>}[r] \ar@{.>}[d]&T\ar[d]\\
&	S\ar[r]^{\text{Fr}_{S}}&S
}
}
\end{equation*}
Here,  $T(1):=T\times_{S,\text{Fr}_{S}} S$. Let $n\in\mathbb{Z}$ be an non-negative integer. We set $T(0):=T$ and $T(n):=T(n-1)(1)$ for $n\geq 1$.   We call  $T(n)$  the ($n\text{-}th$) $Frobenius$ $twist$ $of$ $T$ over $S$.  The morphism $\text{Fr}^{n}_{T/S}:T\rightarrow T(n)$ induced by the universality of the fiber product  is called the $ (n\text{-}th)\ relative\ Frobenius\ morphism$ $of$ $T$ $over$ $S$. 
\end{enumerate}
\section*{Acknowledgments}
The author would like to  acknowledge my advisor Professor Akio Tamagawa     for his great help and  suggestions.  The author  would like to thank  Ippei Nagamachi.   Without his result  Lemma \ref{lemmanag}, we would have had difficulty getting  Theorem \ref{2goodreductiontheorem} when $\kappa(s)$ is imperfect.

%%%%%%%%%%%%%%%%%%%%%%%%%%%%%%%%%%%%%%%%%%%%%%%%%%%%%%%%%%%%%%%%%%%%%%%%%%%%%%%%%%%%%%%%%%%%%
%%%%%%%%%%%%%%%%%%%%%%%%%%%%%%%%%%%%%           %%%%%%%%%%%%%%%%%%%%%%%%%%%%%%%%%%%%%%%%%%%
%%%%%%%%%%%%%%%%%%%%%%%%%%%%%%%%%%%       %%    %%%%%%%%%%%%%%%%%%%%%%%%%%%%%%%%%%%%%%%%%%
%%%%%%%%%%%%%%%%%%%%%%%%%%%%%%%%%%%%%%%%%    %%%%%%%%%%%%%%%%%%%%%%%%%%%%%%%%%%%%%%%%%%%%%%
%%%%%%%%%%%%%%%%%%%%%%%%%%%%%%%%%%%%%%%%%    %%%%%%%%%%%%%%%%%%%%%%%%%%%%%%%%%%%%%%%%%%%%%%
%%%%%%%%%%%%%%%%%%%%%%%%%%%%%%%%%%%%%%%%%    %%%%%%%%%%%%%%%%%%%%%%%%%%%%%%%%%%%%%%%%%%%%%%%
%%%%%%%%%%%%%%%%%%%%%%%%%%%%%%%%%%%%%%%%%    %%%%%%%%%%%%%%%%%%%%%%%%%%%%%%%%%%%%%%%%%%%%%%
%%%%%%%%%%%%%%%%%%%%%%%%%%%%%%%%%%%%%%%%%    %%%%%%%%%%%%%%%%%%%%%%%%%%%%%%%%%%%%%%%%%%%%%
%%%%%%%%%%%%%%%%%%%%%%%%%%%%%%%%%%%%%%%%%    %%%%%%%%%%%%%%%%%%%%%%%%%%%%%%%%%%%%%%%%%%%
%%%%%%%%%%%%%%%%%%%%%%%%%%%%%%%%%%%%                      %%%%%%%%%%%%%%%%%%%%%%%%%%%%%%%
%%%%%%%%%%%%%%%%%%%%%%%%%%%%%%%%%%%%%%%%%%%%%%%%%%%%%%%%%%%%%%%%%%%%%%%%%%%%%%%%%%%%%%%%%%%%%

\numberwithin{theorem}{section}

\section{Basic results on $\Pi_{U}^{(m)}$}\label{sectionbasicresults}

\hspace{\parindent}
In this section, we introduce some basic results on $\Pi_{U}^{(m)}$.
In subsection \ref{subsectionwf}, by using  the weight filtration of $\overline{\Pi}^{1,\Sigma}_{U}$, we   show  the center-freeness of $\Pi^{(m,\Sigma)}_{U}$.
In subsection \ref{subsectionseverallemma},  we introduce several known group-theoretical reconstructions and  show  several useful lemmas, which  are used many times in this paper. 
In subsection \ref{subsectioninertia}, we  show the group-theoretical reconstruction of  inertia groups of $\overline{\Pi}_{U}^{m-1,\Sigma}$ from $\Pi_{U}^{(m,\Sigma)}$. \\\ \\
\noindent{\bf Notation of section \ref{sectionbasicresults} } In this section, we use  the following notation in addition to  Notation (in the introduction).  \begin{itemize}
\item Let   $k$ be a   field finitely generated over the prime field. Let  $p$ $(\geq 0)$ be the characteristic of $k$. 
\item  Let  $(X,E)$ be a smooth curve of type $(g,r)$ over   $k$  and set  $U:=X-E$.
\item Let $\Sigma$ be a set of primes  containing  a prime different from $p$. Set  $\Sigma^{\dag}:=\Sigma-\{p\}$. 
\end{itemize}

%%%%%%%%%%%%%%%%%%%%%%%%%%%%%%%%%%%%%%%%%%%%%%%%%%%%%%%%%%%%%%%%%%%%%%%%%%%%%%%%%%%%%%%%%%%%%
%%%%%%%%%%%%%%%%%%%%%%%%%%%%%%%%%%%%%           %%%%%%%%%%%%%%%%%%%%%%%%%%           %%%%%%%%%%%%%%%%%
%%%%%%%%%%%%%%%%%%%%%%%%%%%%%%%%%%%       %%    %%%%%%%%%%%%%%%%%%%%%%%%      %%     %%%%%%%%%%%%%%%%%%
%%%%%%%%%%%%%%%%%%%%%%%%%%%%%%%%%%%%%%%%%    %%%%%%%%%%%%%%%%%%%%%%%%%%%%%     %%%%%%%%%%%%%%%%%%
%%%%%%%%%%%%%%%%%%%%%%%%%%%%%%%%%%%%%%%%%    %%%%%%%%%%%%%%%%%%%%%%%%%%%%%     %%%%%%%%%%%%%%%%%%
%%%%%%%%%%%%%%%%%%%%%%%%%%%%%%%%%%%%%%%%%    %%%%%%%%%%%%%%%%%%%%%%%%%%%%%     %%%%%%%%%%%%%%%%%%
%%%%%%%%%%%%%%%%%%%%%%%%%%%%%%%%%%%%%%%%%    %%%%%%%%%%%%%%%%%%%%%%%%%%%%%     %%%%%%%%%%%%%%%%%%
%%%%%%%%%%%%%%%%%%%%%%%%%%%%%%%%%%%%%%%%%    %%%%%%%%%%%%     %%%%%%%%%%%%%%     %%%%%%%%%%%%%%%%%%%
%%%%%%%%%%%%%%%%%%%%%%%%%%%%%%%%%%%%%%%%%    %%%%%%%%%%%%     %%%%%%%%%%%%%%     %%%%%%%%%%%%%%%%%%
%%%%%%%%%%%%%%%%%%%%%%%%%%%%%%%%%%%%                      %%%%%%%     %%%%%%%%%%                   %%%%%%%%%%%%%%
%%%%%%%%%%%%%%%%%%%%%%%%%%%%%%%%%%%%%%%%%%%%%%%%%%%%%%%%%%%%%%%%%%%%%%%%%%%%%%%%%%%%%%%%%%%%%

\subsection{The center-freeness of  $\Pi_{U}^{(m,\Sigma)}$.}\label{subsectionwf}

\hspace{\parindent}
In this subsection, by using  the weight filtration of $\overline{\Pi}_{U}^{1,\Sigma}$, we show the  center-freeness of  $\Pi_{U}^{(m,\Sigma)}$.  \par

\begin{lemma}\label{quotientlemma1.1}
 Let $n$ be  an integer that satisfies  $m\geq n\geq 0$.
\begin{enumerate}[(1)]
\item Let $G$ be a profinite group. Let  $H$ be an open subgroup of $G^{m}$  containing  $G^{[m-n]}/G^{[m]}$. Let $\tilde{H}$ be  the inverse image of $H$ in $G$ by the natural surjection $G\twoheadrightarrow G^{m}$. Then  the natural surjection   $\tilde{H}^{n}\twoheadrightarrow H^{n}$ is  an isomorphism.
\item Let $H$ be an open subgroup of $\Pi_{U}^{(m,\Sigma)}$ containing  $(\overline{\Pi}_{U}^{\Sigma})^{[m-n]}/(\overline{\Pi}_{U}^{\Sigma})^{[m]}$. Let  $\tilde{H}$ be the inverse image of $H$ in $\Pi_{U}$  by the natural surjection  $\Pi_{U}\twoheadrightarrow \Pi_{U}^{(m,\Sigma)}$. Then  the natural surjection   $\tilde{H}^{(n,\Sigma)}\twoheadrightarrow H^{(n,\Sigma)}(=H^{(n)})$ is an  isomorphism.
\end{enumerate}
\end{lemma}

\begin{proof}
(1)  We have the following commutative diagram.
\begin{equation*}\label{diagramH}
\vcenter{\xymatrix@R=20pt@C=60pt{
1\ar[r] 	&\tilde{H}^{[n]}\ar[r]\ar@{->>}[d]		&\tilde{H}\ar[r]\ar@{->>}[d]&\tilde{H}^{n}\ar[r]\ar[d]	&1\\
1\ar[r]	& H^{[n]}\ar[r]							&H\ar[r]					&H^{n}\ar[r]								&1
}}\end{equation*}
The kernel of the middle vertical arrow coincides with $G^{[m]}=G^{[m]}\cap \tilde{H}$. The kernel of the  left-hand vertical arrow also coincides with $G^{[m]}=(G^{[m-n]})^{[n]}$ as  $G^{[m-n]}\subset \tilde{H}$.  Hence  the right-hand vertical arrow is an isomorphism by the snake lemma. \\
(2) Let $H_{1}$ be the inverse image of  $H$ in $\Pi_{U}^{(\Sigma)}$  by the natural surjection $\Pi_{U}^{(\Sigma)}\twoheadrightarrow \Pi_{U}^{(m,\Sigma)}$. By applying (1) to the case where   $G=\overline{\Pi}_{U}^{\Sigma}$, we get   $\overline{H}_{1}^{n}\xrightarrow{\sim} \overline{H}^{n}$. Moreover, we have  $\overline{\tilde{H}}^{\Sigma}\xrightarrow{\sim}\overline{H}_{1}^{\Sigma}(=\overline{H}_{1})$. These imply  $\overline{\tilde{H}}^{n,\Sigma}\xrightarrow{\sim}\overline{H}^{n,\Sigma}(=\overline{H}^{n})$. Hence we obtain that  $\tilde{H}^{(n,\Sigma)}\xrightarrow{\sim} H^{(n,\Sigma)}(=H^{(n)})$ by the snake lemma.
\end{proof}

We  define  an outer Galois representation $G_{k}\rightarrow \Out(\overline{\Pi}^{m,\Sigma}_{U})$ by the following diagram. 
\begin{equation}\label{outeraction}
\vcenter{\xymatrix@R=20pt@C=40pt{
1\ar[r] 	&\overline{\Pi}^{m,\Sigma}_{U}\ar[r]\ar@{->>}[d]		&\Pi^{(m,\Sigma)}_{U}\ar[r]\ar[d]&G_{k}\ar[r] \ar@{.>}[d]	&1\\
1\ar[r]	& \text{Inn}(\overline{\Pi}^{m,\Sigma}_{U})\ar[r]							&\Aut(\overline{\Pi}^{m,\Sigma}_{U})\ar[r]					&\Out(\overline{\Pi}^{m,\Sigma}_{U})\ar[r]								&1
}}
\end{equation}
Here, the middle vertical arrow in (\ref{outeraction}) is the homomorphism determined from the conjugate action. 

\begin{lemma}\label{wflemma}
The following isomorphism and the  exact sequence of  $G_k$-modules exist.
\begin{equation}\label{wfseq}
\begin{cases}
\overline{\Pi}_{U}^{1 ,\Sigma}\xrightarrow{\sim} T_{\Sigma}(J_{X})&\ \ \ \ \ (r=0)\\
0\rightarrow \hat{\mathbb{Z}}^{\Sigma^{\dag}}(1)\rightarrow \mathbb{Z}[E(k^{\sep})]\bigotimes_{\mathbb{Z}} \hat{\mathbb{Z}}^{\Sigma^{\dag}}(1)\xrightarrow{f} \overline{\Pi}_{U}^{1 ,\Sigma}\rightarrow T_{\Sigma}(J_{X})\rightarrow 0&\ \ \ \ \ (r\neq 0).\\
\end{cases}
\end{equation}
Here, $ \mathbb{Z}[E(k^{\sep})]$ is the  free $\mathbb{Z}$-module with the basis $E(k^{\sep})$ and is regarded  as a  $G_k$-module via the natural   $G_k$-action on   $E(k^{\sep})$, and $f$ satisfies that $f(v\otimes1)$  is  a (topological) generator of the inertia group of $\overline{\Pi}_{U}^{1 ,\Sigma}$ at  $v\in E(k^{\sep})$. Further, the $G_k$-representations on $\mathbb{Z}[E(k^{\sep})]\bigotimes_{\mathbb{Z}} \hat{\mathbb{Z}}^{\Sigma^{\dag}}(1)$ and $T_{\Sigma}(J_{X})$ have (Frobenius) weights $-2$ and $-1$, respectively. 
\end{lemma}

\begin{proof}
For ``the first assertion'' and ``the second assertion when $p\notin \Sigma$'', see  \cite{Na1990-405} section 2, \cite{Ta1997} Remark (1.3), \cite{Ya2020} subsection 1.3. Thus, it is sufficient  to show that  the $G_k$-representation on $T_{p}(J_{X})$ has  weight  $-1$  when $p>0$.  Let $(\mathcal{X},\mathcal{E})$ be an affine hyperbolic curve of type $(g,r)$ over $S$ whose generic fiber is isomorphic to (and identified with) $(X,E)$, where  $S$ is an  integral regular scheme  of finite type over $\text{Spec}(\mathbb{Z})$ with function field $k$. By shrinking $S$ if necessary,  we have that $G_{k}\rightarrow \text{Aut}(T_{p}(J_{X}))$ factors through  $G_{k}\twoheadrightarrow \pi_{1}(S)$. Let $s\in S^{\text{cl}}$. Let  $\text{Fr}_{\kappa(s)}\in G_{\kappa(s)}$ be the Frobenius element of $\kappa(s)$.  Let $P(t)\in \mathbb{Z}[t]$ be the characteristic polynomial of $\text{Fr}_{\kappa(s)}$ on $T_{\ell}(J_{\mathcal{X}_{s}})$, where $\ell\in\Sigma^{\dag}$. Then $P(\text{Fr}_{\kappa(s)})\mid_{J_{\mathcal{X}_{s}}[\ell^{\infty}]}=0$. The $\text{Fr}_{\kappa(s)}$-action  and the restriction of the Frobenius endomorphism of $J_{\mathcal{X}_{s}}$ coincide on $J_{\mathcal{X}_{s}}[\ell^{\infty}]$.  As  $J_{\mathcal{X}_{s}}[\ell^{\infty}]$ is dense in $J_{\mathcal{X}_{s}}$,  we also obtain that $P(\text{Fr}_{\kappa(s)})\mid_{J_{\mathcal{X}_{s}}[p^{\infty}]}=0$, and the eigenvalues of the action of $\text{Fr}_{\kappa(s)}$ on  $T_{p}(J_{\mathcal{X}_{s}})$ are roots of  $P(t)$. Thus, the $G_{\kappa(s)}$-representation on $T_{p}(J_{\mathcal{X}_{s}})$ has weight  $-1$. Therefore,  the $G_k$-representation on $T_{p}(J_{X})$ has weight  $-1$.
\end{proof}

\noindent   We write  $W_{-2}(\overline{\Pi}_{U}^{1,\Sigma})$ for the maximal weight $-2$ submodule of $\overline{\Pi}_{U}^{1,\Sigma}$, which can be regarded as a part of the weight filtration of $\overline{\Pi}_{U}^{1,\Sigma}$.  We have  that $W_{-2}(\overline{\Pi}_{U}^{1,\Sigma})=I_{\overline{\Pi}^{1,\Sigma}_{U}}(=I_{\overline{\Pi}^{1,\Sigma^{\dag}}_{U}})$ by  Lemma \ref{wflemma}.  \par 
Next, we  show  the center-freeness of $\Pi^{(m,\Sigma)}_{U}$.  

\begin{proposition}\label{center}
\begin{enumerate}[(1)]
\item  $Z(\Pi^{(m,\Sigma)}_{U})\cap \overline{\Pi}^{m,\Sigma}_{U}=\{1\}$.
\item  Assume that the homomorphism $G_{k}\rightarrow \text{Aut}(\overline{\Pi}^{1,\Sigma}_{U})$ is injective when $k$ is a finite field.  Then  $\Pi^{(m,\Sigma)}_{U}$ is center-free.
\end{enumerate}
\end{proposition}
\begin{proof}
(1) Let us show the assertion by using induction on $m$. First, we consider the case that $m=1$.  We have  that $(\overline{\Pi}_{U}^{1,\Sigma})^{G_{k}}=\left\{1\right\}$, since  the action  $G_{k}\curvearrowright \overline{\Pi}_{U}^{1,\Sigma}$ has weights  $-1$ and $-2$ by Lemma \ref{wflemma}. Hence  $ Z(\Pi^{(1,\Sigma)}_{U})\cap \overline{\Pi}^{1,\Sigma}_{U}=(\overline{\Pi}_{U}^{1,\Sigma})^{\Pi^{(1,\Sigma)}_{U}}=\left\{1\right\}$ follows, where $\Pi^{(1,\Sigma)}_{U}$ acts on $\overline{\Pi}_{U}^{1,\Sigma}$ by conjugation.  Next, we consider the general case. By the assumption of induction on $m$, we get $Z(\Pi^{(m,\Sigma)}_{U})\cap \overline{\Pi}^{m,\Sigma}_{U}\subset (\overline{\Pi}_{U}^{\Sigma})^{[m-1]}/(\overline{\Pi}_{U}^{\Sigma})^{[m]}$. Hence  it is sufficient to show that $Z(\Pi^{(m,\Sigma)}_{U})\cap(\overline{\Pi}_{U}^{\Sigma})^{[m-1]}/(\overline{\Pi}_{U}^{\Sigma})^{[m]}=\{1\}$.  Set  $\mathcal{Q}:=\{H\overset{\text{op}}\subset \Pi^{(m,\Sigma)}_{U}\mid (\overline{\Pi}_{U}^{\Sigma})^{[m-1]}/(\overline{\Pi}_{U}^{\Sigma})^{[m]}\subset H\}$. Let $H$ be an element of $\mathcal{Q}$. By the case that $m=1$, we have that  $ Z(H^{(1)})\cap \overline{H}^{1}=\{1\}$, and hence $Z(\Pi_{U}^{(m,\Sigma)})\cap \overline{H}\subset  \overline{H}^{[1]}$. Considering all $H\in\mathcal{Q}$,   we obtain that  $Z(\Pi_{U}^{(m,\Sigma)})\cap ((\overline{\Pi}_{U}^{\Sigma})^{[m-1]}/(\overline{\Pi}_{U}^{\Sigma})^{[m]}) \subset \underset{H\in\mathcal{Q}}\cap \overline{H}^{[1]}=((\overline{\Pi}_{U}^{\Sigma})^{[m-1]}/(\overline{\Pi}_{U}^{\Sigma})^{[m]})^{[1]}=\{1\}$.  Thus,  the assertion follows.\\
(2)  When $k$ is not  finite, we know that  $G_{k}$  is center-free by \cite{Fa2008} section 16, and   hence   $ Z(\Pi^{(m,\Sigma)}_{U})\subset \overline{\Pi}^{m,\Sigma}_{U}$  follows. Thus, $Z(\Pi^{(m,\Sigma)}_{U})=\{1\}$ follows by (1). Next, we consider the case that  $k$ is  finite.   The injectivity of $G_{k}\rightarrow \text{Aut}(\overline{\Pi}^{1,\Sigma}_{U})$   implies that  $ Z(\Pi^{(1,\Sigma)}_{U})\subset Z_{\Pi^{(1,\Sigma)}_{U}} ( \overline{\Pi}^{1,\Sigma}_{U})\subset \overline{\Pi}^{1,\Sigma}_{U}$.   This implies that  $ Z(\Pi^{(m,\Sigma)}_{U})$ is mapped to $\{1\}$ by the   homomorphism $\Pi_{U}^{(m,\Sigma)}(\rightarrow \Pi_{U}^{(1,\Sigma)})\rightarrow G_{k}$. Therefore, by (1),  $Z(\Pi^{(m,\Sigma)}_{U})=Z(\Pi^{(m,\Sigma)}_{U})\cap \overline{\Pi}^{m,\Sigma}_{U}=\{1\}$ follows.
\end{proof}

\noindent The representation $G_{k}\rightarrow \text{Aut}(\overline{\Pi}^{1,\Sigma}_{U})$  is not always injective when $k$ is a finite field.  Consider a  character $\rho^{\dag}_{U/k}: G_{k}\rightarrow (\hat{\mathbb{Z}}^{\pro p'})^{\times}$ obtained as the composite of the following homomorphisms.
\begin{equation}\label{rhodet}
\rho^{\dag}_{U/k}: G_{k}\rightarrow \Aut(\overline{\Pi}^{1}_{U}) \rightarrow \Aut(\overline{\Pi}^{1,\pro p'}_{U}) \xrightarrow{\text{det}}   \text{Aut}\left(\bigwedge_{\hat{\mathbb{Z}}^{\pro p'}}^{\text{max}}\overline{\Pi}_{U}^{1,\pro p' }\right)=(\hat{\mathbb{Z}}^{\pro p'})^{\times} 
\end{equation}

\begin{lemma}\label{rhoinj}
Assume that  $(g,r)\neq (0,0)$, $(0,1)$,  and  that  $k$ is a finite field. Then the character $\rho^{\dag}_{U/k}$  is injective.  In particular, the representations $G_{k}\rightarrow \text{Aut}(\overline{\Pi}^{1,\pro p'}_{U})$ and $G_{k}\rightarrow \text{Aut}(\overline{\Pi}^{1}_{U})$ are  injective. 
\end{lemma}
\begin{proof}
We consider  the action $G_{k}\curvearrowright \mathbb{Z}[E(k^{\text{sep}})]$.    Let   $v\in E$. Let  $\rho:E(k^{\text{sep}})\rightarrow E$ be the natural surjection. We have that the action $\text{Fr}_{k}\curvearrowright \rho^{-1}(v)$ is a cyclic permutation, hence the determinant of  $\text{Fr}_{k}\curvearrowright \mathbb{Z}[\rho^{-1}(v)]$ is $(-1)^{|\rho^{-1}(v)|-1}$. Hence we obtain that the determinant of $\text{Fr}_{k}\curvearrowright\mathbb{Z}[E(k^{\text{sep}})]$ is $(-1)^{|E(k^{\text{sep}})|-|E|}$.  Let  $\chi:G_{k}\rightarrow (\hat{\mathbb{Z}}^{\pro p'})^{\times}$ be  the  cyclotomic character and  set  $\lambda: G_{k}\rightarrow G_{k}/G_{k}^{2}\cong \mathbb{Z}/2\mathbb{Z}\cong \{\pm1\}\hookrightarrow (\hat{\mathbb{Z}}^{\text{pro-}p'})^{\times}$.    Then    we obtain that 
\begin{equation}\label{rhocalu}
\rho^{\dag}_{U/k}= \lambda^{|E(k^{\text{sep}})|-|E|}\chi^{g+r-\epsilon},
\end{equation}
by Lemma \ref{wflemma}, where  $\epsilon$ stands for $1$ (resp. $0$) when $r\geq 1$ (resp. when  $r=0$).  Since  the cyclotomic character $\chi$ and the map $\hat{\mathbb{Z}}\rightarrow \hat{\mathbb{Z}}$ of multiplication by $n$ $(n\in \mathbb{Z}_{\geq 1})$ are injective,   the character  $\chi^{n}$ is also injective.   Hence we get $(\rho^{\dag}_{U/k})^{2}$ is injective.  Thus,  $\rho^{\dag}_{U/k}$ is also  injective. The second assertion is clear because $\rho^{\dag}_{U/k}$ factors through  $G_{k}\rightarrow \text{Aut}(\overline{\Pi}^{1,\pro p'}_{U})$ and $G_{k}\rightarrow \text{Aut}(\overline{\Pi}^{1}_{U})$.
\end{proof}

%%%%%%%%%%%%%%%%%%%%%%%%%%%%%%%%%%%%%%%%%%%%%%%%%%%%%%%%%%%%%%%%%%%%%%%%%%%%%%%%%%%%%%%%%%%%%
%%%%%%%%%%%%%%%%%%%%%%%%%%%%%%%%%%%%%           %%%%%%%%%%%%%%%%%%%%%%%%%%%%%            %%%%%%%%%%%%%%%
%%%%%%%%%%%%%%%%%%%%%%%%%%%%%%%%%%%       %%    %%%%%%%%%%%%%%%%%%%%%%%%%%      %%%        %%%%%%%%%%%%%%
%%%%%%%%%%%%%%%%%%%%%%%%%%%%%%%%%%%%%%%%%    %%%%%%%%%%%%%%%%%%%%%%%%%%      %%%%%    %%%%%%%%%%%%%%
%%%%%%%%%%%%%%%%%%%%%%%%%%%%%%%%%%%%%%%%%    %%%%%%%%%%%%%%%%%%%%%%%%%%%%%%%%%    %%%%%%%%%%%%%%%
%%%%%%%%%%%%%%%%%%%%%%%%%%%%%%%%%%%%%%%%%    %%%%%%%%%%%%%%%%%%%%%%%%%%%%%%%%    %%%%%%%%%%%%%%%
%%%%%%%%%%%%%%%%%%%%%%%%%%%%%%%%%%%%%%%%%    %%%%%%%%%%%%%%%%%%%%%%%%%%%%%%     %%%%%%%%%%%%%%%%%
%%%%%%%%%%%%%%%%%%%%%%%%%%%%%%%%%%%%%%%%%    %%%%%%%%%%%%%%%%%%%%%%%%%%%%     %%%%%%%%%%%%%%%%%%%
%%%%%%%%%%%%%%%%%%%%%%%%%%%%%%%%%%%%%%%%%    %%%%%%%%%%%%%    %%%%%%%%%%%     %%%%%%%%%%%%%%%%%%%%%
%%%%%%%%%%%%%%%%%%%%%%%%%%%%%%%%%%%%                      %%%%%%%%     %%%%%%%%%                         %%%%%%%%%%%
%%%%%%%%%%%%%%%%%%%%%%%%%%%%%%%%%%%%%%%%%%%%%%%%%%%%%%%%%%%%%%%%%%%%%%%%%%%%%%%%%%%%%%%%%%%%%

\subsection{The group-theoretical reconstruction of  various invariants of  $\Pi_{U}^{(m,\Sigma)}$.}\label{subsectionseverallemma}

\hspace{\parindent}In this subsection, we show the group-theoretical reconstruction of  the invariants  $\overline{\Pi}_{U}^{m,\Sigma}$, $g$, and  $r$ (resp. the invariant $|k|$)  from  $\Pi_{U}^{(m,\Sigma)}$ (resp. $\Pi_{U}^{(m)}$).  The results of this subsection are essentially shown in \cite{Ta1997} section 3  if we discuss from  $\Pi_{U}^{(\Sigma)}$ instead of $\Pi_{U}^{(m,\Sigma)}$.

\begin{lemma}\label{gandr}
Assume that $U$ is hyperbolic   (i.e., $(g,r)\neq (0,0),(0,1),(0,2),(1,0)$). Let $\ell$ be a prime different from $p$. Let    $g_{0},r_{0}\in \mathbb{Z}_{\geq 0}$ be integers. 
\begin{enumerate}[(1)]
\item If $r=0$, then there exists an open characteristic  subgroup $H$ of $ \overline{\Pi}_{U}^{1,\text{pro-}\ell}$ such that  $g(U_{H})\geq g_{0}$.
\item   If  $r>0$ and  $(g,r)\neq  (1,1)$,  then there exists an open  characteristic subgroup $H$ of $ \overline{\Pi}_{U}^{1,\text{pro-}\ell}$  such that  $g(U_H)\geq g_{0}$ and  $r(U_H)\geq r_{0}$.
\item If  $(g,r)=(1,1)$,  then there exists an open characteristic subgroup $H$ of $ \overline{\Pi}_{U}^{1,\text{pro-}\ell}$ (resp. $ \overline{\Pi}_{U}^{2,\text{pro-}\ell}$) such that  $r(U_H)\geq r_{0}$ (resp.  $g(U_H)\geq g_{0}$ and  $r(U_H)\geq r_{0}$).
\end{enumerate}
\end{lemma}

\begin{proof}
Let $a\in\mathbb{Z}_{\geq1}$, and set $N:=\ell^{a}$. We define  $\epsilon$ as $1$ (resp. $0$) when $r\geq 1$ (resp. when $r=0$).   We set $H:=\text{Ker}(\overline{\Pi}_{U}^{1,\text{pro-}\ell}\twoheadrightarrow (\overline{\Pi}_{U}^{1,\pro\ell}/(\overline{\Pi}_{U}^{1,\pro\ell})^{N})\cong(\mathbb{Z}/N\mathbb{Z})^{2g+r-\epsilon})$. We set $\alpha:=r-2$ and  $\beta:=1$ (resp. $\alpha:=0$ and  $\beta:=0$) when $r\geq 2$ (resp. $r<2$). Note that $2g+r-\epsilon=2g+\alpha+\beta$. We have the following equalities.
\begin{eqnarray*}
2g(U_{H})-2&=&(2g-2)[\overline{\Pi}^{1,\pro\ell}_{U}:H]+\underset{x\in X_{H}}\Sigma(e_{x}-1)\ \ (\text{the Riemann-Hurwitz formula})\\
&=& (2g-2)N^{2g+r-\epsilon}+rN^{2g+\alpha}(N^{\beta}-1)\\
&=&(2g-2+r)N^{2g+r-\epsilon}-rN^{2g+\alpha},
\end{eqnarray*}
where  $e_{x}$ is the ramification index of $x$ in $X_{H}\rightarrow X$.   We have that $2g+r-\epsilon\geq 2g-2+r>0$ by the hyperbolicity of $U$, and $2g+r-\epsilon=2g+\alpha$ if and only if $r<2$. Thus, when ``$g\geq 2$ or $r\geq 2$'' ($\Leftrightarrow (g,r)\neq (1,1)$), we can take $g(U_{H})$  large enough (by taking $N$ large enough). We also  have $r(U_{H})=rN^{2g+\alpha}$. Therefore, when $(g,r)\neq  (1,1)$  and  $r>0$, we can take $g(U_{H})$ and $r(U_{H})$   large enough (by taking $N$ large enough).
Thus, the assertions   (1) and (2) hold. The assertion (3) holds from (1), (2), and  $r(U_{H})=rN^{2g+\alpha}$.
\end{proof}

\begin{lemma}\label{grreco}
 Assume that $(g,r)\neq (0,0)$, $(0,1)$. Let $\ell$ be a prime different from $p$.
\begin{enumerate}[(1)]
\item\label{greco}  $g=\frac{1}{2}\text{rank}_{\mathbb{Z}_{\ell}}(\overline{\Pi}_{U}^{1,\text{pro-}\ell}/W_{-2}(\overline{\Pi}_{U}^{1,\text{pro-}\ell}))$.
\item\label{rreco-1}  If $r\geq 1$, then $r=\text{rank}_{\mathbb{Z}_{\ell}}(W_{-2}(\overline{\Pi}_{U}^{1,\text{pro-}\ell}))+1$
\item\label{rreco-2} $r\leq 1$ if and only if $W_{-2}(\overline{\Pi}_{U}^{1,\text{pro-}\ell})=\{0\}$.  Moreover,  if  $m\geq 2$, then  $r= 0$ if and only if $W_{-2}(\overline{H}^{1})=\{0\}$ for every open subgroup $H$ of $\Pi_{U}^{(m,\text{pro-}\ell)}$ that contains $(\overline{\Pi}_{U}^{\text{pro-}\ell})^{[m-1]}/(\overline{\Pi}_{U}^{\text{pro-}\ell})^{[m]}$. 
\end{enumerate}
\end{lemma}
\begin{proof}
The assertions (1)(2) and the first assertion of (3) follow from  Lemma \ref{wflemma}.  When $U$ is hyperbolic (i.e., $\overline{\Pi}_{U}^{m}$ is not abelian) the second assertion of (3) follows from  the first assertion of (3) and Lemma \ref{gandr}. When $(g,r)=(0,2)$ or $(1,0)$,   the second assertion of (3) follows from  the first assertion of (3). Thus, the assertions follow.
\end{proof}

\begin{proposition}\label{severalinvrecoprop}
 Let $i=1$, $2$. Let    $g_{i},r_{i}\in \mathbb{Z}_{\geq 0}$ be integers. Let $(X_{i},E_{i})$ be a smooth curve of type $(g_{i},r_{i})$ over   $k$ and set  $U_{i}:=X_{i}-E_{i}$.  Assume that $(g_{1},r_{1})\neq (0,0)$, $(0,1)$. Let  $\Phi: \Pi^{(m,\Sigma)}_{U_{1}}\xrightarrow[G_{k}]{\sim}\Pi^{(m,\Sigma)}_{U_{2}}$ be  a $G_{k}$-isomorphism.  
\begin{enumerate}[(1)]
\item  $g_{1}=g_{2}$.
\item  If, moreover,   either ``$r_{1}\geq 2$", ``$r_{1}\geq 1$ and $r_{2}\geq 1$" or  ``$m\geq 2$" holds, then   $r_{1}=r_{2}$ holds.
\end{enumerate}
\end{proposition}
\begin{proof}
By  (\ref{trivialcong2}),  we have that   $(g_{1},r_{1})\in \{(0,0), (0,1)\}$ if and only if $(g_{2},r_{2})\in \{(0,0), (0,1)\}$. Hence the assertions follow  from  Lemma \ref{grreco}(\ref{greco})(\ref{rreco-1})(\ref{rreco-2}).
\end{proof}

In section 2, we have to consider isomoprhisms $\Pi^{(m,\Sigma)}_{U_{1}}\xrightarrow{\sim}\Pi^{(m,\Sigma)}_{U_{2}}$  for smooth curves $U_{1}$, $U_{2}$ over finite fields $k_{1}$, $k_{2}$, respectively. Hence  we have to show that Proposition \ref{severalinvrecoprop} is also true in the case that $k$ is finite and $\Phi$ is an arbitrary isomoprhism  (which may not be a  $G_{k}$-isomorphism).  

\begin{lemma}\label{Frobeniusreco}
Assume that   $k$ is  a finite field.
\begin{enumerate}[(1)]
\item  $\overline{\Pi}^{m,\Sigma}_{U}$ coincides with the kernel of the morphism
\begin{equation*}\label{geompi1}
\Pi_{U}^{(m,\Sigma)}\twoheadrightarrow (\Pi_{U}^{(m,\Sigma)})^{\text{ab}}/(\Pi_{U}^{(m,\Sigma)})^{\text{ab}}_{\text{tor}}.
\end{equation*}
\item Assume that  $(g,r)\neq (0,0)$, $(0,1)$. Then  $p$ is the unique prime number such that $\overline{\Pi}_{U}^{1,\text{pro-}p'}$ is free as a $\hat{\mathbb{Z}}^{\text{pro-}p'}$-module.
\item  Assume that  $(g,r)\neq (0,0)$, $(0,1)$. Then the  $|k|$-th power Frobenius element $\text{Fr}_{k}\in G_{k}$ is a unique element of $G_{k}$ that satisfies the following conditions. 
\begin{enumerate}[(a)]
\item $G_{k}$ is topologically generated by $\text{Fr}_{k}$. 
\item $\rho^{\dag}_{U/k}(\text{Fr}_{k}) $ (see (\ref{rhodet}))  is contained in $\pm p^{\mathbb{Z}_{\geq 0}}$.
\end{enumerate}
\item Assume that  $(g,r)\neq (0,0)$, $(0,1)$.  Let $\ell$ be a prime different from $p$ and   $\mathcal{A}$  the set of  absolute values of all eigenvalues of the Frobenius action $\text{Fr}_{k}\curvearrowright \overline{\Pi}_{U}^{1,\pro\ell}$.  Then $\mathcal{A}=\{|k|^{\frac{1}{2}},\ |k|\}$ (resp.  $\mathcal{A}=\{|k|^{\frac{1}{2}}\}$, resp. $\mathcal{A}=\{|k|\}$) when  $r\geq 2$ and   $g\geq 1$ (resp. $r<2$, resp.  $g=0$). 
\end{enumerate}
\end{lemma}
\begin{proof}
(1) Similar to   \cite{Ta1997} Proposition (3.3)(ii).\\
(2) Similar to   \cite{Ta1997} Proposition (3.1).\\
(3) Similar to  \cite{Ta1997} Proposition (3.4)(i)(ii).\\
(4) Similar to  \cite{Ta1997} Proposition  (3.4)(iii).
\end{proof}

\begin{proposition}\label{Gkisomorphic}
 Let $i=1$, $2$. Let $k_{i}$ be a finite field of characteristic $p_{i}$.   Let    $g_{i},r_{i}\in \mathbb{Z}_{\geq 0}$ be integers. Let $(X_{i},E_{i})$ be a smooth curve of type $(g_{i},r_{i})$ over   $k_{i}$ and set  $U_{i}:=X_{i}-E_{i}$.  Let $\Phi: \Pi^{(m,\Sigma)}_{U_{1}}\xrightarrow{\sim}\Pi^{(m,\Sigma)}_{U_{2}}$ be  an isomorphism.   
 \begin{enumerate}[(1)]
 \item  For any integer $n\in\mathbb{Z}_{\geq 0}$ satisfying $m\geq n$, $\Phi$ induces a unique  isomorphism $\Phi^{m-n}:\Pi^{(m-n,\Sigma)}_{U_{1}}\xrightarrow{\sim}\Pi^{(m-n,\Sigma)}_{U_{2}}$  such that the following diagram is commutative.
\[
\xymatrix@R=20pt{
\Pi^{(m,\Sigma)}_{U_{1}}\ar[r]^-{\Phi}\ar[d] & \Pi^{(m,\Sigma)}_{U_{2}}\ar[d]\\
\Pi^{(m-n,\Sigma)}_{U_{1}}\ar[r]^-{\Phi^{m-n}} & \Pi^{(m-n,\Sigma)}_{U_{2}}
}\] 

  \item    Assume that  $\Sigma$ contains all primes but $p_{1}$ and   that $(g_{1},r_{1})\neq (0,0)$, $(0,1)$. Then $p_{1}=p_{2}$ and   $\Phi^{0}(\text{Fr}_{k_{1}})=\text{Fr}_{k_{2}}$ hold. 
  \item  Assume that  $\Sigma$ contains all primes but $p_{1}$, that $(g_{1},r_{1})\neq (0,0)$, $(0,1)$, and that $m\geq2$. Then  $|k_{1}|=|k_{2}|$ holds.
    \item  Assume that $\Sigma$ contains all primes but $p_{1}$ and that $m\geq 2$. Then  $\Phi^{1}|_{\overline{\Pi}_{U_{1}}^{1,\Sigma}}:\overline{\Pi}_{U_{1}}^{1,\Sigma}\xrightarrow{\sim}\overline{\Pi}_{U_{2}}^{1,\Sigma}$ induces $W_{-2}(\overline{\Pi}_{U_{1}}^{1, \Sigma})\xrightarrow{\sim}W_{-2}(\overline{\Pi}_{U_{2}}^{1, \Sigma})$.
 \end{enumerate}
\end{proposition}
\begin{proof}
(1) The assertion follows from Lemma \ref{Frobeniusreco}(1).\\ 
(2) By (1), $\Phi$ induces an isomorphism $\Phi^{1}|_{\overline{\Pi}_{U_{1}}^{1,\Sigma}}:\overline{\Pi}_{U_{1}}^{1,\Sigma}\xrightarrow{\sim}\overline{\Pi}_{U_{2}}^{1,\Sigma}$. Thus, the first and second assertions follow from Lemma \ref{Frobeniusreco}(2)(3), respectively.\\
(3) By (\ref{trivialcong2}) and (\ref{trivialcong}), we have that $(g_{1},r_{1})= (1,0)$, $(0,2)$ if and only if $(g_{2},r_{2})=  (1,0)$, $(0,2)$. 
If $(g_{i},r_{i})= (1,0)$ (resp.  $(0,2)$), then $\text{rank}_{\hat{\mathbb{Z}}^{\Sigma^{\dag}}}(\overline{\Pi}_{U_{i}}^{1,\Sigma^{\dag}})=2$ (resp. $1$). 
Hence $(g_{1},r_{1})=(1,0)$  (resp.  $(0,2)$) if and only if $(g_{2},r_{2})=(1,0)$  (resp.  $(0,2)$).  
Thus,  the assertion follows from Lemma \ref{Frobeniusreco}(4) when $(g_{1},r_{1})=(1,0)$, $(0,2)$. 
Hence we may  assume that $U_{1}$ is hyperbolic. 
Let $s: G_{k}\rightarrow \Pi_{U_{1}}^{(m,\Sigma)}$ be a section of the projection $\Pi_{U_{1}}^{(m,\Sigma)}\twoheadrightarrow G_{k}$. By Lemma \ref{gandr}(1)(2), there exists an open characteristic subgroup $H'$ of $\overline{\Pi}_{U_{1}}^{m,\Sigma}$ containing $(\overline{\Pi}_{U_{1}}^{\Sigma})^{[1]}/(\overline{\Pi}_{U_{1}}^{\Sigma})^{[m]}$ such that  $g(U_{1,H'})\geq 1$. We set $H:=s(G_{k})\cdot H'$.   Since $H^{(1)}\xrightarrow{\sim}\Phi(H)^{(1)}$ and  $g(U_{1,H})=g(U_{2,\Phi(H)})\geq 1$, we obtain that $|k_{1}|=|k_{2}|$ by Lemma \ref{Frobeniusreco}(4). Thus, the assertion follows.\\
(4) By (\ref{trivialcong2}), we have that $(g_{1},r_{1})= (0,0)$, $(0,1)$ if and only if $(g_{2},r_{2})= (0,0)$, $(0,1)$.  When  $(g_{1},r_{1})= (0,0)$, $(0,1)$,  the assertion is clearly true, since $\overline{\Pi}^{m, \Sigma}_{U_{1}}$ is trivial by  (\ref{trivialcong2}). When $(g_{1},r_{1})\neq  (0,0)$, $(0,1)$,  the assertion follows from (1)(2)(3).
\end{proof}

%%%%%%%%%%%%%%%%%%%%%%%%%%%%%%%%%%%%%%%%%%%%%%%%%%%%%%%%%%%%%%%%%%%%%%%%%%%%%%%%%%%%%%%%%%%%%
%%%%%%%%%%%%%%%%%%%%%%%%%%%%%%%%%%%%%           %%%%%%%%%%%%%%%%%%%%%%%%%%%%%            %%%%%%%%%%%%%%
%%%%%%%%%%%%%%%%%%%%%%%%%%%%%%%%%%%       %%    %%%%%%%%%%%%%%%%%%%%%%%%%%      %%%        %%%%%%%%%%%%%
%%%%%%%%%%%%%%%%%%%%%%%%%%%%%%%%%%%%%%%%%    %%%%%%%%%%%%%%%%%%%%%%%%%%      %%%%%    %%%%%%%%%%%%%
%%%%%%%%%%%%%%%%%%%%%%%%%%%%%%%%%%%%%%%%%    %%%%%%%%%%%%%%%%%%%%%%%%%%%%%%%%%     %%%%%%%%%%%%%%
%%%%%%%%%%%%%%%%%%%%%%%%%%%%%%%%%%%%%%%%%    %%%%%%%%%%%%%%%%%%%%%%%%%%%%%%%%     %%%%%%%%%%%%%%%
%%%%%%%%%%%%%%%%%%%%%%%%%%%%%%%%%%%%%%%%%    %%%%%%%%%%%%%%%%%%%%%%%%%%%%%%%%%     %%%%%%%%%%%%%%
%%%%%%%%%%%%%%%%%%%%%%%%%%%%%%%%%%%%%%%%%    %%%%%%%%%%%%%%%%%%%%%%%%%%%       %%%%     %%%%%%%%%%%%
%%%%%%%%%%%%%%%%%%%%%%%%%%%%%%%%%%%%%%%%%    %%%%%%%%%%%%%    %%%%%%%%%%%%       %%%      %%%%%%%%%%%%
%%%%%%%%%%%%%%%%%%%%%%%%%%%%%%%%%%%%                      %%%%%%%%     %%%%%%%%%%%%%%         %%%%%%%%%%%%%%%
%%%%%%%%%%%%%%%%%%%%%%%%%%%%%%%%%%%%%%%%%%%%%%%%%%%%%%%%%%%%%%%%%%%%%%%%%%%%%%%%%%%%%%%%%%%%%

\subsection{Inertia groups of  $\overline{\Pi}^{m,\Sigma}_{U}$}\label{subsectioninertia}
\hspace{\parindent} In this subsection,  we show the group-theoretical reconstruction of  inertia groups of $\overline{\Pi}_{U}^{m,\Sigma}$.  First, we consider the relationship between inertia groups of $\Pi_{U}^{\Sigma}$ and  $\Pi_{U}^{(m,\Sigma)}$.

\begin{lemma}\label{decisom2}
Assume that    $(m,r)\neq (1,1)$. Let $\tilde{v}$ be an element of  $\tilde{E}^{\Sigma,\text{cl}}$ and $\tilde{v}^{m}$ the image of $\tilde{v}$ in $(\tilde{E}^{m,\Sigma})^{\text{cl}}$. Then  the natural surjection  $I_{\tilde{v},\overline{\Pi}^{\Sigma}_{U}}\twoheadrightarrow I_{\tilde{v}^{m},\overline{\Pi}^{m,\Sigma}_{U}}$ is an isomorphism. 
\end{lemma}
\begin{proof}
If $\overline{\Pi}^{\Sigma}_{U}$ is abelian, then the assertion clearly holds. Hence we may assume that $(g,r)\neq(0,0)$, $(0,1)$, $(0,2)$, $(1,0)$ by the equivalence (\ref{trivialcong}). Moreover, we may  assume that $r\geq 1$. Since $\text{Ker}(I_{\tilde{v},\overline{\Pi}^{\Sigma}_{U}}\rightarrow \overline{\Pi}^{\Sigma,m}_{U})=I_{\tilde{v},\overline{\Pi}_{U}^{\Sigma}}\cap (\overline{\Pi}_{U}^{\Sigma})^{[m]}$, it is sufficient to show that  $I_{\tilde{v},\overline{\Pi}^{\Sigma}_{U}}\cap  (\overline{\Pi}_{U}^{\Sigma})^{[m]}=\{1\}$. First, we consider the case that $r\geq 2$. Let $x$ be a generator of the inertia group $I_{\tilde{v},\overline{\Pi}^{\Sigma}_{U}}$.  The surjection $\overline{\Pi}^{\Sigma}_{U}\twoheadrightarrow  \hat{\Pi}^{\Sigma^{\dag}}_{g,r}$ in Notation (\ref{surfacegroup}) maps $x$ to (a conjugate of) $\sigma_{i}$ for some $i$, and induces an isomorphism $\overline{\Pi}^{\Sigma^{\dag}}_{U}\xrightarrow{\sim}  \hat{\Pi}^{\Sigma^{\dag}}_{g,r}$. We have that  $\hat{\Pi}^{\Sigma^{\dag}}_{g,r}$ is a free pro-$\Sigma^{\dag}$ group of  rank $2g+r-1\ (>1)$,  and $\sigma_{i}$ is an element of  a set of  free generators. (Here, we use the assumption $r\geq 2$.) Hence   $\overline{\langle \sigma_{i}\rangle}\cap (\hat{\Pi}^{\Sigma^{\dag}}_{g,r})^{[1]}=\{1\}$ follows.   Thus, $I_{\tilde{v}^{\dag},\overline{\Pi}^{\Sigma^{\dag}}_{U}}\cap  (\overline{\Pi}_{U}^{\Sigma^{\dag}})^{[m]}=\{1\}$ follows, where $\tilde{v}^{\dag}$ is the image of $\tilde{v}$ in $(\tilde{E}^{m,\Sigma^{\dag}})^{\text{cl}}$.   Since the natural surjection $\overline{\Pi}_{U}^{\Sigma}\twoheadrightarrow \overline{\Pi}_{U}^{\Sigma^{\dag}}$  induces an isomorphism $I_{\tilde{v},\overline{\Pi}^{\Sigma}_{U}}\xrightarrow{\sim}I_{\tilde{v}^{\dag},\overline{\Pi}^{\Sigma^{\dag}}_{U}}$, we obtain that  $I_{\tilde{v},\overline{\Pi}^{\Sigma}_{U}}\cap  (\overline{\Pi}_{U}^{\Sigma})^{[m]}=\{1\}$. Thus, the assertion follows when $r\geq 2$.  Finally, we consider the case that $r=1$. (In particular, $m\geq 2$ by assumption.)  By  Lemma \ref{gandr}(2)(3), there exists   an open subgroup $H$ of $\Pi_{U}^{(\Sigma)}$ which contains $(\overline{\Pi}_{U}^{\Sigma})^{[1]}$ and satisfies $r(U_{H})\geq 2$.  Hence, by the case that $r\geq 2$, we obtain that $I_{\tilde{v},\overline{\Pi}_{U}^{\Sigma}}\cap \overline{H}^{[1]}=\{1\}$. Since $(\overline{\Pi}_{U}^{\Sigma})^{[1]}\subset H$, we get $(\overline{\Pi}_{U}^{\Sigma})^{[m]}\subset (\overline{\Pi}_{U}^{\Sigma})^{[2]}\subset\overline{H}^{[1]}$. Thus, $I_{\tilde{v},\overline{\Pi}_{U}^{\Sigma}}\cap (\overline{\Pi}_{U}^{\Sigma})^{[m]}=\{1\}$ follows. Therefore, the assertion follows.  
\end{proof}
 
In  \cite{Ya2020} subsection 1.2,  we obtained the separatedness of inertia groups of  $\overline{\Pi}_{U}^{m,\Sigma^{\dag}}$ (\cite{Ya2020} Lemma 1.2.1 and Lemma 1.2.2).  In the following lemma, we show a slightly stronger result.

\begin{lemma}\label{inerrev}
\begin{enumerate}[(1)]
\item Assume that $r\neq 2$. Let   $\tilde{v},\tilde{v}'$ be elements of $\tilde{E}^{1,\Sigma^{\dag}} $ and $\rho: \tilde{E}^{1,\Sigma^{\dag}}\rightarrow \tilde{E}^{0,\Sigma^{\dag}}$ $(=\tilde{E}^{0})$ the natural surjection. Then the following conditions (a)-(c) are equivalent.
\begin{enumerate}[(a)]
\item  $\rho(\tilde{v})=\rho(\tilde{v}')$.
\item $I_{\tilde{v},\Pi_{U}^{(1,\Sigma^{\dag})}}=I_{\tilde{v}',\Pi_{U}^{(1,\Sigma^{\dag})}}$.
\item $I_{\tilde{v},\Pi_{U}^{(1,\Sigma^{\dag})}}$ and $I_{\tilde{v}',\Pi_{U}^{(1,\Sigma^{\dag})}}$ are commensurable.
\end{enumerate} 
\item Assume that $(g,r)\neq (0,0),(0,1),(0,2)$ and that $(m,r)\neq (1,2)$.   Let   $\tilde{v},\tilde{v}'$ be elements of $\tilde{E}^{m,\Sigma} $ and $\rho_{m}: \tilde{E}^{m,\Sigma}\rightarrow \tilde{E}^{m-1,\Sigma}$ the natural surjection. Consider the following conditions (a)-(d).
\begin{enumerate}[(a)]
\item  $\tilde{v}=\tilde{v}'$.
\item $I_{\tilde{v},\Pi_{U}^{(m,\Sigma)}}=I_{\tilde{v}',\Pi_{U}^{(m,\Sigma)}}$.
\item $I_{\tilde{v},\Pi_{U}^{(m,\Sigma)}}$ and $I_{\tilde{v}',\Pi_{U}^{(m,\Sigma)}}$ are commensurable.
\item  $\rho_{m}(\tilde{v})=\rho_{m}(\tilde{v}')$.
\end{enumerate} 
Then (a)$\Rightarrow$(b)$\Rightarrow$(c)$\Rightarrow$(d) holds. 
\item Assume that $(g,r)\neq (0,0),(0,1),(0,2)$ and  that either  ``$m\geq 3$'' or ``$m\geq 2$ and $r\geq 2$''.    Let   $\tilde{v},\tilde{v}'$ be elements of  $\tilde{E}^{m,\Sigma^{\dag}} $. Then  the following conditions (a)-(c) are equivalent.
\begin{enumerate}[(a)]
\item  $\tilde{v}=\tilde{v}'$.
\item $I_{\tilde{v},\Pi_{U}^{(m,\Sigma^{\dag})}}=I_{\tilde{v}',\Pi_{U}^{(m,\Sigma^{\dag})}}$.
\item $I_{\tilde{v},\Pi_{U}^{(m,\Sigma^{\dag})}}$ and $I_{\tilde{v}',\Pi_{U}^{(m,\Sigma^{\dag})}}$ are commensurable.
\end{enumerate} 
In particular, $D_{\tilde{v},\Pi_{U}^{(m,\Sigma^{\dag})}}$ coincides with the nomalizer of $I_{\tilde{v},\Pi_{U}^{(m,\Sigma^{\dag})}}$ in $\Pi_{U}^{(m,\Sigma^{\dag})}$.
\end{enumerate}
\end{lemma}
\begin{proof}
(1) The assertion follows from \cite{Ya2020} Lemma 1.2.1.\\
(2) The implications (a)$\Rightarrow$(b)$\Rightarrow$(c) are clear. We show the implication (c)$\Rightarrow$(d). If $m=1$, then the assertion follows from (1). Hence we assume that $m\geq 2$. We may also assume that $r\geq 1$. Set  $\mathcal{Q}_{1}:=\{H\overset{\text{op}}\subset \Pi_{U}^{(m,\Sigma)}\mid(\overline{\Pi}^{\Sigma}_{U})^{[m-1]}/(\overline{\Pi}^{\Sigma}_{U})^{[m]}\subset H,\  r(U_{1,H})\geq 3\}$ and let  $H$ be an element of $\mathcal{Q}_{1}$.  Let   $v_{H},v_{H}'\in(\tilde{X}_{H}^{0,\Sigma})^{\cl}=\tilde{X}_{H}^{0,\text{cl}}$ be   the  images of $\tilde{v},\tilde{v}'\in\tilde{E}^{m,\Sigma}\subset(\tilde{X}^{m,\Sigma})^{\cl}$, respectively. Then (c) implies that  the images of $I_{\tilde{v}}\cap H$ and $I_{\tilde{v}'}\cap H$ by the map $H\twoheadrightarrow H^{(1,\text{pro-}\Sigma^{\dag})}$ are commensurable, and  hence  we get $v_{H}=v_{H}'$ by (1).  By  Lemma \ref{gandr}(2)(3), $\mathcal{Q}_{1}$  is  cofinal in the set of  open subgroups of $\Pi_{U}^{(m,\Sigma)}$ containing $(\overline{\Pi}_{U}^{\Sigma})^{[m-1]}/(\overline{\Pi}_{U}^{\Sigma})^{[m]}$.  Hence we obtain  that $(\overline{\Pi}^{\Sigma}_{U})^{[m-1]}/(\overline{\Pi}^{\Sigma}_{U})^{[m]}\xrightarrow{\sim}\plim{H\in\mathcal{Q}_{1}}\overline{H}$ and $(\tilde{X}^{m-1,\Sigma})^{\cl}=\plim{H\in\mathcal{Q}_{1}} (\tilde{X}_{H}^{0,\Sigma})^{\text{cl}}$. Thus,   $\rho_{m}(\tilde{v})=\rho_{m}(\tilde{v}')$ follows. Hence the assertion follows. \\
(3) When ``$m\geq 2$ and $r\geq 2$'', the  assertion follows from \cite{Ya2020}  Lemma 1.2.2. The implications (a)$\Rightarrow$(b)$\Rightarrow$(c) is clear. We show the implication  (c)$\Rightarrow$(a) when  $m\geq 3$ and $r=1$. We set  $\mathcal{Q}_{2}:=\{H\overset{\text{op}}\subset \Pi_{U}^{(m,\Sigma^{\dag})}\mid(\overline{\Pi}_{U}^{\Sigma^{\dag}})^{[m-2]}/(\overline{\Pi}_{U}^{\Sigma^{\dag}})^{[m]}\subset H,  r(U_{1,H})\geq 2\}$ and let $H$ be an element of   $\mathcal{Q}_{2}$.  Let   $\tilde{v}_{H},\tilde{v}_{H}'\in(\tilde{X}_{H}^{2,\Sigma^{\dag}})^{\text{cl}}$ be   the  images of $\tilde{v},\tilde{v}'\in(\tilde{X}^{m,\Sigma^{\dag}})^{\cl}$, respectively.  Then (c) implies that  the image of $I_{\tilde{v}}\cap H$ and $I_{\tilde{v}'}\cap H$ by the map $H\twoheadrightarrow H^{(2)}$ are commensurable. Hence  we get $\tilde{v}_{H}=\tilde{v}_{H}'$ by the case that $m\geq 2$ and $r\geq 2$.  By  Lemma \ref{gandr}(2)(3), $\mathcal{Q}_{2}$  is also cofinal in the set of  open subgroups of $\Pi_{U}^{(m,\Sigma^{\dag})}$ containing $(\overline{\Pi}_{U}^{\Sigma^{\dag}})^{[m-2]}/(\overline{\Pi}_{U}^{\Sigma^{\dag}})^{[m]}$.  Hence we obtain  that $\plim{H\in\mathcal{Q}_{2}}\overline{H}^{[2]}\xleftarrow{\sim}((\overline{\Pi}^{\Sigma^{\dag}}_{U})^{[m-2]}/(\overline{\Pi}^{\Sigma^{\dag}}_{U})^{[m]})^{[2]}=\{1\}$ and $(\tilde{X}^{m,\Sigma^{\dag}})^{\cl}=\plim{H\in\mathcal{Q}_{2}} (\tilde{X}_{H}^{2,\Sigma^{\dag}})^{\text{cl}}$. Thus,   $\tilde{v}=\tilde{v}'$ follows. Hence the first assertion follows. The second  assertion follows from the first assertion. 
\end{proof}

   In \cite{Ya2020} subsection 1.4, we obtained the group-theoretical reconstruction of inertia groups of $\overline{\Pi}^{m-2,\Sigma^{\dag}}_{U}$ from $\Pi^{(m,\Sigma^{\dag})}_{U}$  when  $m\geq 3$ and $r\geq 2$.   In the following lemma,  we  show a stronger (bi-anabelian) result by a method different  from  \cite{Ya2020} subsection 1.4. 

\begin{proposition}\label{inertiareco}
Let $i=1$, $2$.  Let $g_{i}$, $r_{i}\in \mathbb{Z}_{\geq 0}$ be integers.  Assume that  $(g_{1},r_{1})\neq (0,0),(0,1),(0,2)$ and that  $m\geq 2$.   Let $n\in\mathbb{Z}_{\geq1}$ be an integer  satisfying   $m> n$.  Let  $(X_{i},E_{i})$ be  a smooth curve of type $(g_{i},r_{i})$ over   $k$ and set   $U_{i}:=X_{i}-E_{i}$.  Let  $\Phi: \Pi^{(m,\Sigma)}_{U_{1}}\xrightarrow[G_{k}]{\sim}\Pi^{(m,\Sigma)}_{U_{2}}$ be  an isomorphism and    $\overline{\Phi}^{m-n}: \overline{\Pi}^{m-n,\Sigma}_{U_{1}}\xrightarrow{\sim}\overline{\Pi}^{m-n,\Sigma}_{U_{2}}$  the isomorphism induced by $\Phi$.  
\begin{enumerate} [(1)]
\item There exists a bijection $\mathcal{F}_{E}:=\mathcal{F}_{E,\Phi}:\tilde{E}_{1}^{m-n,\Sigma} \xrightarrow{\sim}\tilde{E}_{2}^{m-n,\Sigma} $  such that  the following diagram is commutative.
\begin{equation}\label{actiondiag}
\vcenter{\xymatrix@R=20pt{
\overline{\Pi}^{m-n,\Sigma}_{U_{1}}\ar[d]^{\overline{\Phi}^{m-n}} \ar@{}[r]|{\curvearrowright} &  \tilde{E}_{1}^{m-n,\Sigma} \ar[d]^{\mathcal{F}_{E}} \\
\overline{\Pi}^{m-n,\Sigma}_{U_{2}}\ar@{}[r]|{\curvearrowright} &  \tilde{E}_{2}^{m-n,\Sigma} \\
}}
\end{equation}
In particular,   $\overline{\Phi}^{m-n}$ preserves the  inertia groups.
\item Set $h:=m-n$. Assume that $(h,r_{1})\neq (1,2)$. Let $m'\in\mathbb{Z}_{\geq 0}$ be an integer satisfying $h>m'$. Then the bijection $\mathcal{F}_{E}^{m'}:  \tilde{E}_{1}^{m',\Sigma}\xrightarrow{\sim}\tilde{E}_{2}^{m',\Sigma}$ induced by $\mathcal{F}_{E}$ is a unique bijection satisfying  the following diagram is commutative.
\begin{equation}\label{innercommu}
\vcenter{\xymatrix@R=20pt{
  \tilde{E}_{1}^{m',\Sigma} \ar[d]^{\mathcal{F}_{E}^{m'}}\ar[r] &\text{Iner}(\overline{\Pi}_{U_{1}}^{h,\Sigma})/((\overline{\Pi}_{U_{1}}^{\Sigma})^{[m']}/(\overline{\Pi}_{U_{1}}^{\Sigma})^{[h]})\ar[d] \\
 \tilde{E}_{2}^{m',\Sigma}\ar[r] &\text{Iner}(\overline{\Pi}_{U_{2}}^{h,\Sigma})/((\overline{\Pi}_{U_{2}}^{\Sigma})^{[m']}/(\overline{\Pi}_{U_{2}}^{\Sigma})^{[h]})\\
}}
\end{equation}
  Here, $\tilde{E}_{i}^{m',\Sigma} \to\text{Iner}(\overline{\Pi}_{U_{i}}^{h,\Sigma})/((\overline{\Pi}_{U_{i}}^{\Sigma})^{[m']}/(\overline{\Pi}_{U_{i}}^{\Sigma})^{[h]})$ stands for the map induced by the natural map $\tilde{E}_{i}^{h,\Sigma}\rightarrow \text{Iner}(\overline{\Pi}_{U_{i}}^{h,\Sigma})$ and the right-hand vertical arrow stands for  the map induced by $\overline{\Phi}^{h}$. In particular, $\mathcal{F}_{E}^{m'}$ does not depend on $\mathcal{F}_{E}$.
\end{enumerate}
\end{proposition}
\begin{proof}
(1) Since $m\geq 2$, we obtain that $g_{1}=g_{2}$ and  $r_{1}=r_{2}$ by Proposition \ref{severalinvrecoprop}.  Let $i=1,2$. We may assume that $r_{i}\geq 1$.  We write $\mathcal{Q}_{i}:=\{H\overset{\text{op}}\subset \Pi^{(m,\Sigma)}_{U_{i}} \mid (\overline{\Pi}^{\Sigma}_{U_{i}})^{[m-n]}/(\overline{\Pi}^{\Sigma}_{U_{1}})^{[m]}\subset H\}$ and  $\overline{\mathcal{Q}}_{i}:=\{H'\overset{\text{op}}\subset \overline{\Pi}^{m,\Sigma}_{U_{i}}\mid(\overline{\Pi}^{\Sigma}_{U_{i}})^{[m-n]}/(\overline{\Pi}^{\Sigma}_{U_{1}})^{[m]}\subset H'\}$.  The map $\mathcal{Q}_{i}\rightarrow \overline{\mathcal{Q}}_{i},$  $ H\mapsto \overline{H}$ is surjective  by \cite{Lu1982} Lemma A.  Let   $N_{1}'\overset{\text{op}}\lhd \overline{\Pi}_{U_{1}}^{m,\Sigma}$. Let  $H'_{1}\overset{\text{op}}\subset  \overline{\Pi}_{U_{1}}^{m,\Sigma}$ containing $N'_{1}$.  Let  $H_{1}$ be an element of the inverse image  of $H_{1}'$ by  $\mathcal{Q}_{1}\rightarrow \overline{\mathcal{Q}}_{1}$.  Since $\Phi$ induces an  isomorphism $H^{(1)}_{1}\xrightarrow{\sim}\Phi(H_{1})^{(1)}$, we obtain that $r(U_{1,H_{1}})=r(U_{2,\Phi(H_{1})})$ by Proposition \ref{severalinvrecoprop}.  Thus, by \cite{Ta1999} Lemma 2.3, we obtain that 
\begin{equation*}
S_{N'_{1}}:= \left\{\phi: E_{1,N'_{1}} \xrightarrow{\sim}E_{2,\Phi(N'_{1})}\middle|\vcenter{\xymatrix@R=20pt@C=20pt{
\overline{\Pi}^{m,\Sigma}_{U_{1}}/N'_{1} \ar[d]^{\rotatebox{90}{$\sim$}} \ar@{}[r]|{\curvearrowright} & E_{1,N'_{1}} \ar[d]^{\phi} \\
\overline{\Pi}^{m,\Sigma}_{U_{2}}/\Phi(N'_{1})\ar@{}[r]|{\curvearrowright} &  E_{2,\Phi(N'_{1})}
}}\text{ is commutative.}\right\}\neq \emptyset,
\end{equation*}
where the left-hand vertical arrow is induced by $\Phi$.   We have that   the sets $\{S_{N'_{1}}\}_{N'_{1}\in \overline{\mathcal{Q}}_{1}, N'_{1}\overset{\text{op}}\lhd \overline{\Pi}_{U_{1}}^{m,\Sigma}}$ form  a projective system of non-empty finite sets,  that $\overline{\Pi}^{m-n,\Sigma}_{U_{i}}$$=$$\plim{N'\in \overline{\mathcal{Q}}_{i}, N'\overset{\text{op}}\lhd \overline{\Pi}_{U_{i}}^{m,\Sigma}}\overline{\Pi}^{m,\Sigma}_{U_{i}}/N'$, and that    $\tilde{E}_{i}^{m-n,\Sigma}$$ =$$\plim{N'\in \overline{\mathcal{Q}}_{i}, N'\overset{\text{op}}\lhd \overline{\Pi}_{U_{i}}^{m,\Sigma}}E_{i,N'}$.   Thus, there exists  a  bijection $\mathcal{F}_{E}:\tilde{E}_{1}^{m-n,\Sigma} \xrightarrow{\sim}\tilde{E}_{2}^{m-n,\Sigma}$ such that the diagram (\ref{actiondiag}) is commutative. \par
 The  inertia groups of $\overline{\Pi}_{U_{i}}^{m-n,\Sigma}$ are defined as the stabilizers of the action $\overline{\Pi}_{U_{i}}^{m-n,\Sigma}\curvearrowright\tilde{E}_{i}^{m-n,\Sigma}$. Hence the second assertion  follows from the first assertion.\\
(2) The commutativity of (\ref{innercommu}) follows from the commutativity of (\ref{actiondiag}).  By Lemma \ref{inerrev}(2), we obtain that the natural map $\tilde{E}_{i}^{h,\Sigma}\twoheadrightarrow\text{Iner}(\overline{\Pi}_{U_{i}}^{h,\Sigma})$ induces a bijection $\tilde{E}_{i}^{m',\Sigma}\xrightarrow{\sim}\text{Iner}(\overline{\Pi}_{U_{i}}^{h,\Sigma})/((\overline{\Pi}_{U_{i}}^{\Sigma})^{[m']}/(\overline{\Pi}_{U_{i}}^{\Sigma})^{[h]})$. Hence the first assertion  follows. The second assertion follows from the first assertion.
\end{proof}

%%%%%%

%%%%%%%%%%%%%%%%%%%%%%%%%%%%%%%%%%%%%%%%%%%%%%%%%%%%%%%%%%%%%%%%%%%%%%%%%%%%%%%%%%%%%%%%%%%
%%%%%%%%%%%%%%%%%%%%%%%%%%%%%%%%%%%%%           %%%%%%%%%%%%%%%%%%%%%%%%%%%%%%%%%%%%%%%%%%%%%%%%
%%%%%%%%%%%%%%%%%%%%%%%%%%%%%%%%%%%       %%      %%%%%%%%%%%%%%%%%%%%%%%%%%%%%%%%%%%%%%%%%%%%%%%
%%%%%%%%%%%%%%%%%%%%%%%%%%%%%%%%%%%%%%%%%%    %%%%%%%%%%%%%%%%%%%%%%%%%%%%%%%%%%%%%%%%%%%%%%%
%%%%%%%%%%%%%%%%%%%%%%%%%%%%%%%%%%%%%%%%%   %%%%%%%%%%%%%%%%%%%%%%%%%%%%%%%%%%%%%%%%%%%%%%%%%
%%%%%%%%%%%%%%%%%%%%%%%%%%%%%%%%%%%%%%%%   %%%%%%%%%%%%%%%%%%%%%%%%%%%%%%%%%%%%%%%%%%%%%%%%%%
%%%%%%%%%%%%%%%%%%%%%%%%%%%%%%%%%%%%%%%   %%%%%%%%%%%%%%%%%%%%%%%%%%%%%%%%%%%%%%%%%%%%%%%%%%%
%%%%%%%%%%%%%%%%%%%%%%%%%%%%%%%%%%%%%     %%%%%%%%%%%%%%%%%%%%%%%%%%%%%%%%%%%%%%%%%%%%%%%%%%%%
%%%%%%%%%%%%%%%%%%%%%%%%%%%%%%%%%%%%     %%%%%%%%%%%%%%%%%%%%%%%%%%%%%%%%%%%%%%%%%%%%%%%%%%%%%
%%%%%%%%%%%%%%%%%%%%%%%%%%%%%%%%%%%%              %%%%%%%%%%%%%%%%%%%%%%%%%%%%%%%%%%%%%%%%%%%%%%%%
%%%%%%%%%%%%%%%%%%%%%%%%%%%%%%%%%%%%%%%%%%%%%%%%%%%%%%%%%%%%%%%%%%%%%%%%%%%%%%%%%%%%%%%%%%%%

\section{The case of  finite fields}\label{sectionfin}

\hspace{\parindent}In this section,  we show  the (weak bi-anabelian  and strong bi-anabelian) $m$-step solvable Grothendieck conjecture for  affine hyperbolic curves over   finite fields (Theorem \ref{finGCweak} and Theorem \ref{finGCstrong}). 
In subsection \ref{subsectionsep}, we show the separatedness property of decomposition groups of $\Pi_{U}^{(m)}$. 
In  subsection \ref{subsectionreco}, we  show the group-theoretical reconstruction of decomposition groups of $\Pi_{U}^{(m-1)}$ from $\Pi_{U}^{(m)}$. 
In subsection \ref{subsectionfinweak} and subsection \ref{subsectionfinstrong}, we show the main results of this section. 
\\\ \\
{\bf Notaion of section \ref{sectionfin} } In this section, we use  the following notation in addition to Notation (see Introduction).  
 \begin{itemize}
\item  For $i=1,$ $2$, let  $k_{i}$ (resp. $k$) be a  finite field of  characteristic  $p_{i}$ (resp. $p$).
\item   For $i=1,$ $2$, let $(X_{i},E_{i})$ (resp. $(X,E)$) be a smooth  curve of type $(g_{i},r_{i})$ (resp. $(g,r)$) over $k_{i}$ (resp. $k$)  and set  $U_{i}:=X_{i}-E_{i}$ (resp. $U:=X-E$).
\end{itemize}

%%%%%%%%%%%%%%%%%%%%%%%%%%%%%%%%%%%%%%%%%%%%%%%%%%%%%%%%%%%%%%%%%%%%%%%%%%%%%%%%%%%%%%%%%%%
%%%%%%%%%%%%%%%%%%%%%%%%%%%%%%%%%%%%%           %%%%%%%%%%%%%       %%%%%%%%%%%%%%%%%%%%%%%%%%%%%%%
%%%%%%%%%%%%%%%%%%%%%%%%%%%%%%%%%%%       %%      %%%%%%%%%%%    %   %%%%%%%%%%%%%%%%%%%%%%%%%%%%%%%
%%%%%%%%%%%%%%%%%%%%%%%%%%%%%%%%%%%%%%%%%%    %%%%%%%%%%  %%%   %%%%%%%%%%%%%%%%%%%%%%%%%%%%%%%
%%%%%%%%%%%%%%%%%%%%%%%%%%%%%%%%%%%%%%%%%   %%%%%%%%%%%%%%%%   %%%%%%%%%%%%%%%%%%%%%%%%%%%%%%%
%%%%%%%%%%%%%%%%%%%%%%%%%%%%%%%%%%%%%%%%   %%%%%%%%%%%%%%%%%   %%%%%%%%%%%%%%%%%%%%%%%%%%%%%%%
%%%%%%%%%%%%%%%%%%%%%%%%%%%%%%%%%%%%%%%   %%%%%%%%%%%%%%%%%%   %%%%%%%%%%%%%%%%%%%%%%%%%%%%%%%
%%%%%%%%%%%%%%%%%%%%%%%%%%%%%%%%%%%%%     %%%%%%%%%%%%%%%%%%%   %%%%%%%%%%%%%%%%%%%%%%%%%%%%%%%
%%%%%%%%%%%%%%%%%%%%%%%%%%%%%%%%%%%%     %%%%%%%%%%   %%%%%%%%    %%%%%%%%%%%%%%%%%%%%%%%%%%%%%%%
%%%%%%%%%%%%%%%%%%%%%%%%%%%%%%%%%%%%              %%%%%   %%%%%%            %%%%%%%%%%%%%%%%%%%%%%%%%%%%%
%%%%%%%%%%%%%%%%%%%%%%%%%%%%%%%%%%%%%%%%%%%%%%%%%%%%%%%%%%%%%%%%%%%%%%%%%%%%%%%%%%%%%%%%%%%%

\subsection{The separatedness of decomposition groups of $\Pi_{U}^{(m)}$}\label{subsectionsep}
\hspace{\parindent}In this subsection, we  show the separatedness property of decomposition groups of $\Pi_{U}^{(m)}$. First, we define sections and quasi-sections of the natural projection $\text{pr}:\Pi_{U}^{(m)}\twoheadrightarrow G_{k}$. 

\begin{definition}\label{sectionsofpi}
 Let   $G$ be  an open subgroup of $G_{k}$ and denote by $\iota$ the natural inclusion $G\hookrightarrow G_{k}$.    Let $H$ be an open subgroup of $\Pi_{U}^{(m)}$. We define the set $\Sect(G,H):=\left\{s\in\Hom_{\text{cont}}(G,\Pi^{(m)}_{U})\mid \text{pr}\circ s=\iota\text{,  }s(G)\subset H\right\}.$ We call an element of  $\Sect(G,H)$ a $section$. We say that $s\in\Sect(G,H)$ is $geometric$, if  there exists  $\tilde{v}\in \tilde{X}^{m,\cl}$ such that $s(G) \subset D_{\tilde{v},\Pi_{U}^{(m)}}$. We define  $\GSect(G,H)$ to be  the set of all geometric sections in $\text{Sect}(G,H)$.  Moreover, we define the following sets
\begin{equation*}
\QSect(H):=\ilim{G\overset{\text{op}}\subset G_{k}}\Sect(G,H),\ \ \ \ \QSect^{\text{geom}}(H):=\ilim{G\overset{\text{op}}\subset G_{k}}\GSect(G,H),
\end{equation*}
where $G$ runs over all open subgroups of $G_{k}$. We  call  an element of $ \QSect(H)$ a $quasi$-$section$. For every $s\in \Sect(G,H)$, we write $[s]$ for the image of $s$ by  $\Sect(G,H)\rightarrow \QSect(H)$. 
\end{definition}

\begin{remark}\label{rem2.2}
 Let    $H$ be  an open subgroup of $\Pi_{U}^{(m)}$,  $G$  an open subgroup of $G_{k}$, and $s\in \Sect(G,\Pi^{(m)}_{U})$.  Then $s\mid_{G\cap s^{-1}(H)}$ yields an element $\tilde{s}\in\text{Sect}(G\cap s^{-1}(H),H)$, and $[\tilde{s}]\in\text{QSect}(H)$  is mapped to $[s]$ by  the natural map $\QSect(H)\rightarrow \QSect(\Pi^{(m)}_{U})$. In particular, the natural map $\QSect(H)\rightarrow \QSect(\Pi^{(m)}_{U})$ is bijective (as it is clearly injective).  The natural map  $\QSect^{\text{geom}}(H)\rightarrow \QSect^{\text{geom}}(\Pi^{(m)}_{U})$ is also bijective. 
\end{remark}

We define the map 
\begin{equation}\label{mapj}
j_{U}(G): \Sect(G,\Pi_{U}^{(1)}) \times \Sect(G,\Pi_{U}^{(1)}) \to H^{1}_{\text{cont}}(G,\overline{\Pi}^{1}_{U})
\end{equation}
which  sends a pair $(s_{1},s_{2})$ to the cohomology class of the (continuous) 1-cocycle $G\rightarrow \overline{\Pi}^{1}_{U}$, $\sigma\mapsto s_{1}(\sigma)s_{2}(\sigma)^{-1}$.

\begin{lemma}\label{lem2.1s} Let $G$ be  an open subgroup of $G_{k}$. 
\begin{enumerate}[(1)]
\item Let $A$ be a semi-abelian variety over $k$. Let $a$ be a $\overline{k}^{G}$-rational point of $A$ and $0$ the origin of $A$. Let $s_{a}$, $s_{0}\in \text{Hom}_{\text{cont}}(G,\pi_{1}(A)^{(\text{pro-}p')})$  be sections associated to $a$, $0$, respectively. Then the projective limit 
\[
\plim{ p\nmid n}A(\overline{k}^{G})/nA(\overline{k}^{G})\rightarrow H^{1}_{\text{cont}}(G,T_{p'}(A)).
\]
of  the Kummer  homomorphisms  maps $a$ to the class of the 1-cocycle $G\rightarrow T_{p'}(A)$, $\sigma\mapsto s_{a}(\sigma)s_{0}(\sigma)^{-1}$.

\item  Assume that $g\geq 1$ and $r=0$.  Let $s$, $s'$ be  elements of  $\GSect(G,\Pi^{(1)}_{X})$ and $\tilde{v}$, $\tilde{v}'$ elements of $\tilde{X}^{1,\text{cl}}$ satisfying $s(G)\subset D_{\tilde{v},\Pi_{X}^{(1)}}$ and  $s'(G)\subset D_{\tilde{v}',\Pi_{X}^{(1)}}$, respectively.   Let $v$, $v'\in X^{\text{cl}}$ be the images of $\tilde{v}$, $\tilde{v}'$ by the natural map $\tilde{X}^{1}\twoheadrightarrow X$. Then $v$, $v'$ are $\overline{k}^{G}$-rational and  $j_{X}(G)(s,s')$ coincides with the image of the degree $0$ divisor $v-v'$  by the composite of the  homomorphisms 
\begin{equation}\label{imp}
\xymatrix{
\text{Div}^{0}(X_{ \overline{k}^{G}})\ar[r]& J_{X}(\overline{k}^{G})\ar[r]^-{\sim}  & \plim{n} J_{X}(\overline{k}^{G})/nJ_{X}(\overline{k}^{G})\ar@{^{(}-_>}[r] & H^{1}_{\text{cont}}(G,T(J_{X})).
}
\end{equation}
\item Assume that $g=0$, $r=2$, and $E(\overline{k})=E(k)$.  Let $s$, $s'$ be  elements of  $\GSect(G,\Pi^{(1)}_{U})$ and $\tilde{v}$, $\tilde{v}'$ elements of $\tilde{X}^{1,\text{cl}}$ satisfying $s(G)\subset D_{\tilde{v},\Pi_{U}^{(1)}}$ and  $s'(G)\subset D_{\tilde{v}',\Pi_{U}^{(1)}}$, respectively.   Let $v$, $v'\in X^{\text{cl}}$ be the images of $\tilde{v}$, $\tilde{v}'$ by the natural map $\tilde{X}^{1}\twoheadrightarrow X$. Assume that $v$, $v'\not\in E$.  We fix an isomorphism $U\xrightarrow{\sim} \mathbb{P}_{k}^{1}- \{0,\infty\}=\mathbb{G}_{m,k}$,  and identify $U$ with $\mathbb{G}_{m,k}$. Then $v$, $v'$ are $\overline{k}^{G}$-rational and   $j_{U}(G)(s,s')$ coincides with the image of  $v/v'$  by the composite of the  maps 
\begin{equation}\label{mapj0}
\xymatrix{
 \mathbb{G}_{m,k}(\overline{k}^{G}) \ar[r]^-{\sim}& \plim{p\nmid n} \mathbb{G}_{m,k}(\overline{k}^{G})/\mathbb{G}_{m,k}(\overline{k}^{G})^{\times n}\ar[r]^-{\sim}  & H^{1}_{\text{cont}}(G,T_{p'}(\mathbb{G}_{m,k}))\ =\ H^{1}_{\text{cont}}(G,T(\mathbb{G}_{m,k})).\\
}
\end{equation}
\end{enumerate}
\end{lemma}
\begin{proof}
(1) When $A$ is an abelian variety, the assertion is proved in   \cite{St2010} Proposition 28 .   The proof for the case that $A$ is a semi-abelian variety is just the same as the proof for the case that $A$ is an abelian variety. \\
(2)  See \cite{Ta1997} LEMMA (2.6).\\
(3) The assertion follows from  (1).
\end{proof}

\begin{lemma}\label{seplemmacase10}  Assume that $(g,r)\neq (0,0)$, $(0,1)$. Let $\tilde{v},\tilde{v}'$ be elements of $\tilde{X}^{1,\cl}$. Consider the following conditions (a)-(d). 
\begin{enumerate}[(a)]
\item  $\tilde{v}=\tilde{v}'$.
\item $D_{\tilde{v},\Pi_{U}^{(1)}}=D_{\tilde{v}',\Pi_{U}^{(1)}}$.
\item $D_{\tilde{v},\Pi_{U}^{(1)}}$ and $D_{\tilde{v}',\Pi_{U}^{(1)}}$ are commensurable.
\item The image of $D_{\tilde{v},\Pi_{U}^{(1)}}\cap D_{\tilde{v}',\Pi_{U}^{(1)}}$ in $G_{k}$ is open.
\end{enumerate}
If   either ``$(g,r)\neq (0,2)$''  or ``$(g,r)=(0,2)$ and  $\tilde{v},\tilde{v}'\notin \tilde{E}^{1}$'' (resp. either ``$(g,r)\neq (0,2), (0,3)$'',  ``$(g,r)=(0,3)$ and $\tilde{v}\notin \tilde{E}^{1}$'',  or ``$(g,r)=(0,2)$ and $\tilde{v},\tilde{v}'\notin \tilde{E}^{1}$''), then the conditions  (a)-(c)  (resp. (a)-(d)) are equivalent.
\end{lemma}

\begin{proof}
The implications (a)$\Rightarrow$(b)$\Rightarrow$(c)$\Rightarrow$(d) are clear. We consider the following condition.
\begin{itemize}
\item [(d$'$)]The image of $D_{\tilde{v},\Pi_{U}^{(1)}}\cap D_{\tilde{v}',\Pi_{U}^{(1)}}$ in $G_{k}$ is open and the images of $\tilde{v}$, $\tilde{v}'$ by $\tilde{X}^{1}\twoheadrightarrow \tilde{X}^{0}(=X_{\overline{k}})$ are the same.
\end{itemize}
\noindent (Step $1$) In this step, we show that (d$'$)$\Rightarrow$(a).  Let  $G$ be the image of $D_{\tilde{v}}\cap D_{\tilde{v}'}$ in $G_{k}$, which is open in $G_{k}$ by the  assumption (d$'$). Since $G$ acts on  $I_{\tilde{v}'}\subset \overline{\Pi}_{U}^{1}$, we get the action $G\curvearrowright \overline{\Pi}_{U}^{1}/I_{\tilde{v}'}$.   The action  $G\curvearrowright \overline{\Pi}_{U}^{1}/I_{\tilde{v}'}$ has weights  $-1$ and $-2$ by Lemma \ref{wflemma}. Hence we obtain  that   $(\overline{\Pi}_{U}^{1}/I_{\tilde{v}'})^{G}=\left\{1\right\}$.  By the condition (d$'$), there exists $\gamma\in\overline{\Pi}_{U}^{1}$ such that $\tilde{v}'=\gamma\tilde{v}$.  Let $t\in G$ be an element and  $\tilde{t}\in D_{\tilde{v}}\cap D_{\tilde{v}'}$  an inverse image of $t$.   Since $\gamma\tilde{t}\gamma^{-1}\in \gamma D_{\tilde{v}}\gamma^{-1}=D_{\tilde{v}'}$ and $\tilde{t}^{-1}\gamma\tilde{t}\in \overline{\Pi}^{1}_{U}$, we obtain that $\tilde{t}^{-1}\gamma\tilde{t}\gamma^{-1}\in D_{\tilde{v}'}\cap \overline{\Pi}^{1}_{U}=I_{\tilde{v}'}$. Hence we get   $\tilde{t}^{-1}\gamma\tilde{t}\equiv\gamma$ (mod $I_{\tilde{v}'})$ for any $t\in G$. Thus, $\gamma$ (mod $I_{\tilde{v}'}$) $\in (\overline{\Pi}_{U}^{1}/I_{\tilde{v}'})^{G}=\{1\}$ and hence $\gamma\in I_{\tilde{v}'}$. Therefore,  $\tilde{v}=\gamma^{-1} \tilde{v}'=\tilde{v}'$. 
\begin{spacing}{0.7}
\ \\
\end{spacing}\noindent
(Step $2$)  In this step,  we  show that (d)$\Rightarrow$(d$'$) when either ``$(g,r)\neq (0,2), (0,3)$'',  ``$(g,r)=(0,3)$ and $\tilde{v}\notin \tilde{E}^{1}$'',  or ``$(g,r)=(0,2)$ and $\tilde{v},\tilde{v}'\notin \tilde{E}^{1}$''. Let $G$ be an open subgroup of the image of $D_{\tilde{v}}\cap D_{\tilde{v}'}$ in $G_{k}$ and $v_{G},v_{G}'\in X_{\overline{k}^{G}}$ the  images of $\tilde{v},\tilde{v}'$.  (By definition,  $v_{G}$ and $v_{G}'$ are $\overline{k}^{G}$-rational points of $X_{\overline{k}^{G}}$.) \par
First, assume that $g\geq 1$.    We have    $\Pi_{U}^{(1)}\twoheadrightarrow \Pi^{(1)}_{X}\twoheadrightarrow G_{k}$. Let   $\tilde{v}_{X}$,  $\tilde{v}'_{X}$ be the images of  $\tilde{v}$, $\tilde{v}'$ by $\tilde{X}^{1}\twoheadrightarrow \tilde{X}^{X,1}$, respectively. Then   the condition (d) for $U$, $\tilde{v}$, $\tilde{v}'$  implies the condition (d) for $X$, $\tilde{v}_{X}$, $\tilde{v}_{X}$.  Moreover, we have  natural surjective morphisms $\tilde{X}^{1}\twoheadrightarrow \tilde{X}^{X,1}\twoheadrightarrow \tilde{X}^{0}=\tilde{X}^{X,0}$. Thus, it is sufficient  to consider the case that $r=0$, i.e., $U=X$.  Let $s\in \Sect(G,\Pi_{X}^{(1)})$ be  the unique section  which satisfies $s(G)\subset  D_{\tilde{v}}\cap D_{\tilde{v}'}$. By Lemma \ref{lem2.1s}(2), the image of the degree $0$ divisor $v_{G}-v_{G}'$ on  $X_{\overline{k}^{G}}$ by (\ref{imp}) coincides with  $j_{X}(G)(s,s)=0$, hence we  obtain that $v_{G}=v_{G}'$.  The set of  all open subgroups of the image of $D_{\tilde{v}}\cap D_{\tilde{v}'}$ in $G_{k}$ is cofinal in the set of  all open subgroups of $G_{k}$, hence the images of $\tilde{v}$, $\tilde{v}'$ in $\tilde{X}^{0}$ are the same. Thus,  (d)$\Rightarrow$(d$'$) follows.\par 
Next, assume that either   ``$g=0$ and $r\geq 4$'',  ``$(g,r)=(0,3)$ and $\tilde{v}\not\in \tilde{E}^{1}$'', or ``$(g,r)=(0,2)$ and $\tilde{v},\tilde{v}'\notin \tilde{E}^{1}$''.  By taking an enough large $k$ if necessary, we may assume that $E(k)=E(\overline{k})$. By these assumptions, there exists  a subset $S\subset E$ with $|S|=2$  which does not contain the images of $\tilde{v},\tilde{v}'$. We fix an isomorphism $X-S\cong \mathbb{P}^{1}_{k}-\{0,\infty\}=\mathbb{G}_{m,k}$ and  identify $X-S$ with $\mathbb{G}_{m,k}$. Let  $s\in \Sect(G,\Pi_{X-S}^{(1)})$ be  the unique section  which satisfies $s(G)\subset  D_{\tilde{v}}\cap D_{\tilde{v}'}$. By Lemma \ref{lem2.1s}(3), the image of  $v_{G}/v_{G}'$ by (\ref{mapj0})  coincides with $j_{U}(G)(s,s)=0$, hence we obtain that $v_{G}=v_{G}'$. The set of  all open subgroups of the image of $D_{\tilde{v}}\cap D_{\tilde{v}'}$ in $G_{k}$ is cofinal in the set of  all open subgroups of $G_{k}$, hence the images of $\tilde{v}$, $\tilde{v}'$ in $\tilde{X}^{0}$ are the same.  Thus,    (d)$\Rightarrow(\text{d}')$ follows.
\begin{spacing}{0.7}
\ \\
\end{spacing}\noindent
 (Step $3$) Finally, we show that (c)$\Rightarrow$(d$'$) when either ``$(g,r)\neq (0,2)$''  or ``$(g,r)=(0,2)$ and  $\tilde{v},$ $\tilde{v}'\notin \tilde{E}^{1}$''. By  (Step $2$), we may assume that  $(g,r)=(0,3)$ and $\tilde{v},\ \tilde{v}'\in \tilde{E}^{1}$. Then  (c) implies that  $I_{\tilde{v},\Pi_{U}^{(1)}}$ and $I_{\tilde{v}',\Pi_{U}^{(1)}}$ are commensurable. Since $(g,r)=(0,3)$,   the images of $\tilde{v}$, $\tilde{v}'$ by $\tilde{E}^{1}\twoheadrightarrow \tilde{E}^{0}(=E_{\overline{k}})$ are the same by Lemma \ref{inerrev}(2) (c)$\Rightarrow$(d). Hence (c)$\Rightarrow$(d$'$) follows.
\end{proof}

\begin{remark}\label{seprem}
In the case that    $(g,r)=(0,3)$ and  $\tilde{v},\ \tilde{v}'\in \tilde{E}^{1}$,  the implication (d)$\Rightarrow$(d$'$) in the proof of  Lemma \ref{seplemmacase10} is false.  Indeed, for simplicity, consider the case that  $E\subset X(k)$ and set  $E=\{v_{1},v_{2},v_{3}\}$. Let $\tilde{v}_{i}\in \tilde{E}^{1}$ be a point above $v_{i}$ for each $i=1,2,3$ and  $\rho:\Pi_{U}^{(1)}\twoheadrightarrow \Pi_{U}^{(1)}/I_{\tilde{v}_{1}}$ the natural surjection. (Observe that $I_{\tilde{v}_{1}}$ is normal in $\Pi_{U}^{(1)}$, since $v_{1}\in E(k)$.)   We have that  $\rho(D_{\tilde{v}_1})\subset \Pi_{U}^{(1)}/I_{\tilde{v}_{1}}=\rho(D_{\tilde{v}_2})$, since the tame fandamental group for a hyperbolic curve of type $(0,2)$ coincides with the decomposition group  of a cusp. This  implies  $ D_{\tilde{v}_1}\subset D_{\tilde{v}_2}\cdot I_{\tilde{v}_{1}}$. Let $t$ be an element of $ G_{k}$ and  $\tilde{t}\in D_{\tilde{v}_{1}}$  an inverse image of $t$. Then there exist $s\in D_{\tilde{v}_2}$ and $\gamma\in I_{\tilde{v}_{1}}$ such that $\tilde{t}=s\gamma$. Hence $s=\tilde{t}\gamma^{-1}\in D_{\tilde{v}_{1}}\cap D_{\tilde{v}_2}$ and $s$ maps to $t$ by $\Pi_{U}^{(1)}\twoheadrightarrow G_{k}$. Thus, the image of $D_{\tilde{v}_{1}}\cap D_{\tilde{v}_2}$ in $G_{k}$ is the whole of $G_{k}$.\par
%However, (a)$\Leftrightarrow$(b)$\Leftrightarrow$(c)  is also true when  $g=0$, $r\geq 3$ and $\tilde{v},\tilde{v}'\in \tilde{E}^{1,\cl}$. Indeed, if (c) holds, then $I_{\tilde{v},\Pi^{(1,\pro\ell)}_{U}}$ and $I_{\tilde{v}',\Pi^{(1,\pro\ell)}_{U}}$ are commensurable  for any prime $\ell\neq p$.   Hence ($a_{0}$) is true by Lemma \ref{inersep1}.  Thus, the claim  follows from  step 3 in the proof  of Lemma \ref{seplemmacase10}.\par
\end{remark}

\begin{proposition}\label{sepprop}  Assume that  $(g,r)\neq(0,0)$, $(0,1)$.  Let $\tilde{v},\tilde{v}'$ be elements of $\tilde{X}^{m,\cl}$.    Consider the following conditions (a)-(d). 
\begin{enumerate}[(a)]
\item  $\tilde{v}=\tilde{v}'$.
\item $D_{\tilde{v},\Pi_{U}^{(m)}}=D_{\tilde{v}',\Pi_{U}^{(m)}}$.
\item $D_{\tilde{v},\Pi_{U}^{(m)}}$ and $D_{\tilde{v}',\Pi_{U}^{(m)}}$ are commensurable.
\item The image of $D_{\tilde{v},\Pi_{U}^{(m)}}\cap D_{\tilde{v}',\Pi_{U}^{(m)}}$ in $G_{k}$ is open.
\end{enumerate}
If  either ``$(g,r)\neq (0,2)$''  or ``$(g,r)= (0,2)$ and $\tilde{v},\tilde{v}'\notin \tilde{E}^{m}$'' (resp. either ``$(g,r)\neq (0,2)$ and $(m,g,r)\neq (1,0,3)$'',  ``$(m,g,r)=(1,0,3)$ and $\tilde{v}\notin \tilde{E}^{1}$'',  or ``$(g,r)=(0,2)$ and $\tilde{v},\tilde{v}'\notin \tilde{E}^{m}$''), then  the conditions  (a)-(c)  (resp. (a)-(d)) are equivalent.
\end{proposition}
\begin{proof}
If either $m=1$  or  $(g,r)=(0,2), \ (1,0)$, then the assertion follows from Lemma \ref{seplemmacase10}. Thus, we may assume that $\overline{\Pi}^{m}_{U}$ is not abelian (see (\ref{trivialcong})). (a)$\Rightarrow$(b)$\Rightarrow$(c)$\Rightarrow$(d) are clear. First, we show that (d)$\Rightarrow$ (a) when either ``$(g,r)\neq (0,2)$ and $(m,g,r)\neq (1,0,3)$'',  ``$(m,g,r)=(1,0,3)$ and $\tilde{v}\notin \tilde{E}^{1}$'',  or ``$(g,r)=(0,2)$ and $\tilde{v},\tilde{v}'\notin \tilde{E}^{m}$''.  
We set  $\mathcal{Q}_{1}:=\{H\overset{\text{op}}\subset \Pi_{U}^{(m)}\mid\overline{\Pi}_{U}^{[m-1]}/\overline{\Pi}_{U}^{[m]}\subset H,  (g(U_{1,H}),r(U_{1,H}))\neq (0,2), (0,3)\}$.  Fix an element $H\in\mathcal{Q}_{1}$.  Let   $\tilde{v}_{H},\tilde{v}_{H}'\in\tilde{X}_{H}^{1,\text{cl}}$ be   the  images of $\tilde{v},\tilde{v}'\in\tilde{X}^{m,\cl}$, respectively. (d) implies that  the image of $(D_{\tilde{v}}\cap H)\cap ( D_{\tilde{v}'}\cap H)$ by $\pr$ is open in $\pr(H)$. Hence  the image of $D_{\tilde{v}_{H}}\cap D_{\tilde{v}_{H}'}$ by $H^{(1)}\twoheadrightarrow \pr(H)$ is also open  in $\pr(H)$. Thus, we get $\tilde{v}_{H}=\tilde{v}_{H}'$ by Lemma \ref{seplemmacase10}.  By  Lemma \ref{gandr}(2), $\mathcal{Q}_{1}$  is cofinal in the set of  open subgroups of $\Pi_{U}^{(m)}$ containing $\overline{\Pi}_{U}^{[m-1]}/\overline{\Pi}_{U}^{[m]}$.  Hence we obtain  that $\overline{\Pi}_{U}^{[m-1]}/\overline{\Pi}_{U}^{[m]}\xrightarrow{\sim}\plim{H\in \mathcal{Q}_{1}}\overline{H}^{1}$ and $\tilde{X}^{m,\cl}=\plim{H\in\mathcal{Q}_{1}} \tilde{X}_{H}^{1,\cl}$. Thus,  $\tilde{v}=\tilde{v}'$ follows. Next, we show that  (c)$\Rightarrow$ (a) when either ``$(g,r)\neq (0,2)$''  or ``$(g,r)= (0,2)$ and $\tilde{v},\tilde{v}'\notin \tilde{E}^{m}$''. By the implications (c)$\Rightarrow$(d)$\Rightarrow$(a) when $(g,r)= (0,2)$ and $\tilde{v},\tilde{v}'\notin \tilde{E}^{m}$, we may  assume that $(g,r)\neq (0,2)$. We set  $\mathcal{Q}_{2}:=\{H\overset{\text{op}}\subset \Pi_{U}^{(m)}\mid\overline{\Pi}_{U}^{[m-1]}/\overline{\Pi}_{U}^{[m]}\subset H \}$and let $H$ be an element of $\mathcal{Q}_{2}$.  Let   $\tilde{v}_{H},\tilde{v}_{H}'\in\tilde{X}_{H}^{1,\text{cl}}$ be   the  images of $\tilde{v},\tilde{v}'\in\tilde{X}^{m,\cl}$, respectively. Then (c) implies that  the images of $D_{\tilde{v}}\cap H$ and $D_{\tilde{v}'}\cap H$ by the map $H\twoheadrightarrow H^{(1)}$ are commensurable. Thus,  we get $\tilde{v}_{H}=\tilde{v}_{H}'$ by Lemma \ref{seplemmacase10}. Since $\tilde{X}^{m,\text{cl}}=\plim{H\in\mathcal{Q}_{2}}\tilde{X}_{H}^{1,\text{cl}}$, we obtain that     $\tilde{v}=\tilde{v}'$. Therefore, the assertion follows.
\end{proof}

\begin{corollary}\label{2.6c}
Assume that $(g,r)\neq(0,0)$, $(0,1)$, $(0,2)$ and that $(m,g,r)\neq  (1,0,3)$.   Let $G$ be an open subgroup of  $G_{k}$. Then there exists  a unique  map 
\[
\phi(G,\Pi^{(m)}_{U}): \Sect^{\text{geom}}(G,\Pi^{(m)}_{U})\rightarrow \tilde{X}^{m,\text{cl}}
\]
such that  $s(G)\subset D_{\phi(G,\Pi^{(m)}_{U})(s)}$ for any $s\in  \Sect^{\text{geom}}(G,\Pi_{U}^{(m)})$.  Moreover, $\phi(G,\Pi^{(m)}_{U})$ is $\Pi_{U}^{(m)}$-equivariant.
\end{corollary}
\begin{proof}
For any $s\in  \GSect(G,\Pi_{U}^{(m)})$, there exists $\tilde{v}\in \tilde{X}^{m,\text{cl}}$ such that  $s(G)\subset D_{\tilde{v}}$  by definition. Hence  the existence part  follows. Further,  an element $\tilde{v}'\in\tilde{X}^{m,\text{cl}}$ satisfying  $s(G)\subset D_{\tilde{v}'}$ is unique  by Proposition \ref{sepprop} (a)$\Leftrightarrow$(d). Hence the uniqueness part follows. The map $\phi(G,\Pi^{(m)}_{U})$ is  $\Pi_{U}^{(m)}$-equivariant by the  uniqueness. Therefore, the assertion follows. 
\end{proof}

 Taking the inductive limit running over all open subgroups of $G_{k}$, we obtain the morphism  $\phi(\Pi^{(m)}_{U}):=\plim{G\subset G_{k}}\phi(G,\Pi^{(m)}_{U}):\QSect^{\text{geom}}(\Pi_{U}^{(m)})\twoheadrightarrow \tilde{X}^{m,\cl}$ which is compatible with the actions of $\Pi_{U}^{(m)}$.

%%%%%%%%%%%%%%%%%%%%%%%%%%%%%%%%%%%%%%%%%%%%%%%%%%%%%%%%%%%%%%%%%%%%%%%%%%%%%%
%%%%%%%%%%%%%%%%%%%%%%%%%%%%%%%%%%%%%           %%%%%%%%%%%%%%%%%           %%%%%%%%%%%%
%%%%%%%%%%%%%%%%%%%%%%%%%%%%%%%%%%%       %%      %%%%%%%%%%%%%%%       %%      %%%%%%%%%
%%%%%%%%%%%%%%%%%%%%%%%%%%%%%%%%%%%%%%%%%%    %%%%%%%%%%%%%%%%%%%%%    %%%%%%%%%%
%%%%%%%%%%%%%%%%%%%%%%%%%%%%%%%%%%%%%%%%%   %%%%%%%%%%%%%%%%%%%%%%   %%%%%%%%%%%
%%%%%%%%%%%%%%%%%%%%%%%%%%%%%%%%%%%%%%%%   %%%%%%%%%%%%%%%%%%%%%%   %%%%%%%%%%%%
%%%%%%%%%%%%%%%%%%%%%%%%%%%%%%%%%%%%%%%   %%%%%%%%%%%%%%%%%%%%%%   %%%%%%%%%%%%%
%%%%%%%%%%%%%%%%%%%%%%%%%%%%%%%%%%%%%     %%%%%%%%%%%%%%%%%%%%%%   %%%%%%%%%%%%%%
%%%%%%%%%%%%%%%%%%%%%%%%%%%%%%%%%%%%     %%%%%%%%%%    %%%%%%%%%     %%%%%%%%%%%%%%
%%%%%%%%%%%%%%%%%%%%%%%%%%%%%%%%%%%%              %%%%%    %%%%%%%%                 %%%%%%%%%
%%%%%%%%%%%%%%%%%%%%%%%%%%%%%%%%%%%%%%%%%%%%%%%%%%%%%%%%%%%%%%%%%%%%%%%%%%%%%%

\subsection{The group-theoretical reconstruction of decomposition groups  of  $\Pi_{U}^{(m)}$}\label{subsectionreco}
\hspace{\parindent}In this subsection, we  show   that  the $\Pi_{U}^{(m-n)}$-set $\text{Dec}(\Pi_{U}^{(m-n)})$ is  reconstructed  group-theoretically from  $\Pi_{U}^{(m)}$  (if $(m,g,r)$ and $n$ satisfy certain conditions). \par
First, we consider the group-theoretical characterization of geometric sections.  In the following lemma, we use the Lefschetz trace formula (see \cite{Ta1997} Proposition (0.7)).

\begin{lemma}\label{7}
Assume that $(g,r)\neq (0,0)$, $(0,1)$, and that $m\geq 2$. Let $G$ be an open subgroup of $G_{k}$,   $n\in\mathbb{Z}_{\geq1}$ an integer satisfying $m>n$,  $\ell$ a prime different from $p$,  and   $s$ an element of $\Sect(G,\Pi_{U}^{(m-n)})$. Let   $\rho:\Pi^{(m)}_{U}\twoheadrightarrow\Pi^{(m-n)}_{U}$ be the natural projection.  Then the following conditions are equivalent.
\begin{enumerate}[(a)]
\item   $s$ is geometric.
\item  For every open subgroup $H$ of $\Pi_{U}^{(m-n)}$ containing $s(G)$, the set $X_{H}(\overline{k}^{G})$ is non-empty.
\item  
For every open subgroup $M$ of  $\Pi_{U}^{(m)}$ containing $\rho^{-1}(s(G))$, 
\[1+|\overline{k}^{G}|-\text{tr}_{\mathbb{Z}_{\ell}}(\text{Fr}_{\overline{k}^{G}}\mid \overline{M}^{1,\pro\ell}/W_{-2}(\overline{M}^{1,\pro\ell}))>0. \]
\end{enumerate}
\end{lemma}
\begin{proof} (Similar to  \cite{Ta1997} Proposition (2.8)(iv).)
First, we show that (a)$\Rightarrow$ (b).  Let $\tilde{v}\in \tilde{X}^{m,\cl}$ such that $s(G)\subset D_{\tilde{v}}$. Then  $\pr(D_{\tilde{v}}\cap H)\supset \pr(s(G))=G$. Hence we get $\overline{k}^{G}\supset \overline{k}^{\pr(D_{\tilde{v}}\cap H)}=\kappa(v_{H})$, where $v_{H}$ stands for the image of $\tilde{v}$ by $\tilde{X}^{m}\twoheadrightarrow X_{H}$. Thus, we obtain  that the set $X_{H}(\overline{k}^{G})$ is non-empty. 
Next, we show that (b)$\Rightarrow$ (a). We have $X_{s(G)}(\overline{k}^{G})=\plim{H}X_{H}(\overline{k}^{G})$, where $H$ runs over all open subgroups of $\Pi_{U}^{(m-n)}$ containing $s(G)$. Since $X_{H}(\overline{k}^{G})$ is finite and non-empty,  $X_{s(G)}(\overline{k}^{G})$ is also non-empty by Tychonoff's theorem. Let $v\in X_{s(G)}(\overline{k}^{G})$. Let  $\tilde{v}\in \tilde{X}^{m-n,\cl}$  be a point above $v$. Then we get $\pr(D_{\tilde{v}}\cap s(G))=G=\pr(s(G))$. Since $\pr|_{s(G)}$ is injective, we obtain that $D_{\tilde{v}}\supset  s(G)$ and  hence $s$ is geometric.  Finally, we show that (b)$\Leftrightarrow$(c).  Note that the map   $\{H \overset{\text{op}}\subset \Pi_{U}^{(m-n)}\mid  s(G)\subset H\}\to \{M \overset{\text{op}}\subset \Pi_{U}^{(m)}\mid \rho^{-1}(s(G))\subset M\}$, $H\mapsto \rho^{-1}(H)$ is bijective. Since $n\geq1$, we have that $\overline{M}^{1,\pro\ell}/W_{-2}(\overline{M}^{1,\pro\ell})\cong T_{\ell}(J_{X_{M}})$ by Lemma \ref{quotientlemma1.1} and  Lemma \ref{wflemma}.  Hence  the assertion follows from  the fact that 
\begin{equation*}\label{eq-1.3-1}
|X_{\rho(M)}(\overline{k}^{G})|=|X_{M}(\overline{k}^{G})|=1+|\overline{k}^{G}|-\text{tr}_{\mathbb{Z}_{\ell}}(\text{Fr}_{\overline{k}^{G}}\mid \overline{M}^{1,\pro\ell}/W_{-2}(\overline{M}^{1,\pro\ell}))\text{ \  \ (Lefschetz trace formula)}.
\end{equation*}
\end{proof}

Next, we define an equivalence relation on  $\QSect^{\text{geom}}(\Pi_{U}^{(1)})$ as follows.

\begin{definition}\label{relation1}
\begin{enumerate}[(1)]
\item  Let $G$ be an open subgroup of $G_{k}$  satisfying $E(\overline{k}^{G})=E(\overline{k})$. Let $s_{G}$, $s'_{G}$ be elemets of $\GSect(G,\Pi_{U}^{(1)})$.  Then  we write  $s_{G}\sim_{G} s'_{G}$ when
 \begin{equation*}
\begin{cases}
\ \ j_{X}(G)(s_{G},s'_{G})=0 & \ \ (\text{if } g\geq 1)\\
\ \ {}^{\exists}w\in E_{\overline{k}^{G}}\text{ such that, }{}^{\forall}S\subset E_{\overline{k}^{G}}-\{w\}\text{ satisfying }|S|=2, \ j_{X_{\overline{k}^{G}}-S}(G)(s_{G},s'_{G})=0.  & \ \ (\text{if }g=0).
\end{cases}
\end{equation*} 
\item 
 Let $\tilde{s},\ \tilde{s}'$ be elements of $\QSect^{\text{geom}}(\Pi_{U}^{(1)})$.  Then we write  $\tilde{s}\sim \tilde{s}'$ when   there exist   an open subgroup $G$ of $G_{k}$ and elements $s_{G}$, $s'_{G}\in \GSect(G,\Pi_{U}^{(1)})$  satisfying   $E(\overline{k}^{G})=E(\overline{k})$, $\tilde{s}=[s_{G}]$ and  $\tilde{s}'=[s'_{G}]$ such that  $s_{G}\sim_{G} s'_{G}$ holds.  
\end{enumerate}
\end{definition}
\begin{lemma}
Assume that $(g,r)\neq (0,0)$, $(0,1)$, $(0,2)$, $(0,3)$, $(0,4)$.   Let $\tilde{s},\ \tilde{s}'$ be elements of $\QSect^{\text{geom}}(\Pi_{U}^{(1)})$.   Then  $\tilde{s}\sim \tilde{s}'$ if and only if  the images of $\phi(\Pi_{U}^{(1)})(\tilde{s})$, $\phi(\Pi_{U}^{(1)})(\tilde{s}')$ in $\tilde{X}^{0,\text{cl}}$ are the same. In particular, the relation $\sim$ is an equivalence relation of  $ \QSect^{\text{geom}}(\Pi_{U}^{(1)})$, and $\phi(\Pi_{U}^{(1)})$ induces a  $G_{k}$-equivariant  bijection $\overline{\phi}(\Pi_{U}^{(1)}): \QSect^{\text{geom}}(\Pi_{U}^{(1)})/\sim\xrightarrow{\sim} \tilde{X}^{0,\cl}$.
\end{lemma}
\begin{proof}
We show the first assertion. The  ``if'' part follows from Lemma \ref{lem2.1s}(2)(3), since  we can take $w$ as the image of $\phi(\Pi_{U}^{(1)})(\tilde{s})$ when $g=0$. We show the  ``only if'' part.   Let $G$  be an  open subgroup of $G_{k}$ satisfying $E(\overline{k}^{G})=E(\overline{k})$   and  $s_{G}$, $s_{G}'$   elements in $\Sect^{\text{geom}}(G,\Pi_{U}^{(1)})$ satisfying  $\tilde{s}=[s_{G}]$, $\tilde{s}'=[s_{G}']$, respectively.  Let $x_{s_{G}}$, $x_{s_{G}'}$ be the images of $\phi(\Pi_{U}^{(1)})(\tilde{s})$, $\phi(\Pi_{U}^{(1)})(\tilde{s}')$ in $X_{\overline{k}^{G}}$, respectively.   Assume that $x_{s_{G}}\sim_{G}x_{s_{G}'}$. When $g\geq 1$,  $x_{s_{G}}=x_{s_{G}'}$ follows by Lemma \ref{lem2.1s}(2).    When $g=0$, there exists $w\in E_{\overline{k}^{G}}$   and  $S' \subset E_{\overline{k}^{G}}- \{w,x_{s_{G}},x_{s_{G}}'\}$  satisfying $|S'|=2$ such that $j_{X_{\overline{k}^{G}}-S'}(G)(s_{G},s'_{G})=0$, since $r\geq 5$. Hence we get $x_{s_{G}}=x_{s'_{G}}$ by Lemma \ref{lem2.1s}(3).  Considering  all $G'\overset{\text{op}}\subset G$, the images of $\phi(\Pi_{U}^{(1)})(\tilde{s})$, $\phi(\Pi_{U}^{(1)})(\tilde{s}')$ in $\tilde{X}^{0,\text{cl}}$ are the same. Hence  the ``only if'' part follows. The second and third assertions follow from the first assertion.
\end{proof}
\noindent
Under the assumption $(g,r)\neq (0,0)$, $(0,1)$, $(0,2)$, $(0,3)$, $(0,4)$, we consider the following commutative diagram of the natural bijections.  

\[
\xymatrix@C=60pt{
&\QSect^{\text{geom}}(\Pi_{U}^{(1)})\ar[ld]\ar[r]^-{\phi:=\phi(\Pi_{U}^{(1)})}\ar[d]_{\rho}&  \tilde{X}^{1,\cl}\ar[d]\\
\overline{\Pi}^{1}_{U}\backslash\QSect^{\text{geom}}(\Pi_{U}^{(1)})\ar[r]^-{\alpha_{U}}\ar[d]_{\beta_{v}}&\QSect^{\text{geom}}(\Pi_{U}^{(1)})/\sim\ar[r]^-{\overline{\phi}:=\overline{\phi}(\Pi_{U}^{(1)})}_-{\sim} & \tilde{X}^{0,\cl}\\
\overline{\Pi}^{1}_{U_{v}}\backslash\QSect^{\text{geom}}(\Pi_{U_{v}}^{(1)})\ar@/_25pt/[rru]_{\psi}\ar@{.>}[ru]^{\alpha_{U,v}}\\
}
\]
Here, we write   $U_{v}:=U\cup \{v\}$ for any $v\in E(k)$. Here,   $\alpha_{U}: \overline{\Pi}^{1}_{U}\backslash\QSect^{\text{geom}}(\Pi_{U}^{(1)})\to \QSect^{\text{geom}}(\Pi_{U}^{(1)})/\sim$,  $\beta_{v}:\overline{\Pi}^{1}_{U}\backslash\QSect^{\text{geom}}(\Pi_{U}^{(1)})\to \overline{\Pi}^{1}_{U_{v}}\backslash\QSect^{\text{geom}}(\Pi_{U_{v}}^{(1)})$, and  $\rho: \QSect^{\text{geom}}(\Pi_{U}^{(1)})\rightarrow \QSect^{\text{geom}}(\Pi_{U}^{(1)})/\sim$ are  the natural surjections, and $\psi$ is   the map induced by $\phi(\Pi^{(1)}_{U_{v}}):\QSect^{\text{geom}}(\Pi_{U_{v}}^{(1)})\rightarrow (\tilde{X}^{U_{v},1})^{\text{cl}}$.  We write $\alpha_{U,v}:=\overline{\phi}^{-1}\circ \psi$. (Since $\beta_{v}$ is surjective,  $\alpha_{U,v}$ is a unique map  such that $\alpha_{U}=\alpha_{U,v}\circ \beta_{v}$.)  We define $\QSect^{\text{geom,c}}(\Pi_{U}^{(1)}):=\{\tilde{s}\in \QSect^{\text{geom}}(\Pi_{U}^{(1)})\mid |\alpha_{U}^{-1}(\rho(\tilde{s})|>1\}$. 

\begin{lemma}\label{reconstructiondecoabel}
Assume that $(g,r)\neq (0,0)$, $(0,1)$, $(0,2)$, $(0,3)$, $(0,4)$.   Let $\tilde{s}$ be an element of $\QSect^{\text{geom}}(\Pi_{U}^{(1)})$. 
\begin{enumerate}[(1)]
\item  Assume that $r\neq 1$. Then  $\tilde{s}\in \QSect^{\text{geom,c}}(\Pi_{U}^{(1)})$  if and only if $\overline{\phi}(\rho(\tilde{s}))\in \tilde{E}^{0}$.   In particular, $\overline{\phi}$ induces a bijection $\QSect^{\text{geom,c}}(\Pi_{U}^{(1)})/\sim \to \tilde{E}^{0}$.
\item  Assume that $r\neq 1$, $2$ (resp. $r=2$), and that $\tilde{s}\in \QSect^{\text{geom,c}}(\Pi_{U}^{(1)})$. Let $G$ be the open subgroup  of $G_{k}$ such that $(G_{k}:G)$ is equal to $r!$, the factorial of  $r$.   Then  $\overline{\phi}(\rho(\tilde{s}))$ is $\overline{k}^{G}$-rational and its image in $E(\overline{k}^{G})$ is  a unique   element (resp. an element)  $x_{\tilde{s}}$ satisfying  $|\alpha_{U_{\overline{k}^{G}}, x_{\tilde{s}}}^{-1}(\rho(\tilde{s}))|=1$.
\item   When $r\geq 2$ and $\tilde{s}\in \QSect^{\text{geom,c}}(\Pi_{U}^{(1)})$  (resp. either  $r< 2$ or $\tilde{s}\not\in \QSect^{\text{geom,c}}(\Pi_{U}^{(1)})$),  let $D_{\tilde{s}}$ be the subgroup 
\[
\langle\{\text{Im}(s)\cdot I_{x_{\tilde{s}},\Pi_{U}^{(1)}}\mid G\overset{\text{op}}\subset G_{k},\ s\in \Sect^{\text{geom}}(G,\Pi_{U}^{(1)})\text{ satisfying }\tilde{s}=[s]\}\rangle
\]
\[ (\text{resp. }\langle\{\text{Im}(s)\mid G\overset{\text{op}}\subset G_{k},\ s\in \Sect^{\text{geom}}(G,\Pi_{U}^{(1)})\text{ satisfying }\tilde{s}=[s]\}\rangle)
\]  
 of $\Pi_{U}^{(1)}$.  (Note that, when $r=2$, $x_{\tilde{s}}$ is not unique but $I_{x_{\tilde{s}},\Pi_{U}^{(1)}}$ does not depend on the choice of $x_{\tilde{s}}$.) Then $D_{\phi(\Pi_{U}^{(1)})(\tilde{s}),\Pi_{U}^{(1)}}$ coincides with  $D_{\tilde{s}}$.
\end{enumerate}
\end{lemma}

\begin{proof}
(1) We have that  $\overline{\phi}(\rho(\tilde{s}))\not\in \tilde{E}^{0}$ implies  $|\alpha_{U}^{-1}(\rho(\tilde{s}))|=1$, since $\phi$ is injective  on the subset  $\{\tilde{s}'\in \QSect^{\text{geom}}(\Pi_{U}^{(1)})\mid \overline{\phi}(\rho(\tilde{s}'))\not\in\tilde{E}^{0}\}$.  Hence it is sufficient to show that   $\overline{\phi}(\rho(\tilde{s}))\in \tilde{E}^{0}$ implies  $|\alpha_{U}^{-1}(\rho(\tilde{s}))|>1$. Assume that $\overline{\phi}(\rho(\tilde{s}))\in \tilde{E}^{0}$.   If  there exist  $a\in \overline{\Pi}^{1}_{U}$ and $\tilde{s}'\in \phi^{-1}(\phi(\tilde{s}))$  such that $a\cdot \tilde{s}=\tilde{s}'$, then $a\cdot \phi(\tilde{s})=\phi(\tilde{s}')$ and hence $a\in I_{\overline{\phi}(\rho(\tilde{s}))}$ follows. Thus,  we obtain that  
$I_{\overline{\phi}(\rho(\tilde{s}))}\backslash \phi^{-1}(\phi(\tilde{s}))=\overline{\Pi}_{U}^{1}\backslash \phi^{-1}(\phi(\tilde{s}))\ (\subset \overline{\Pi}_{U}^{1}\backslash \rho^{-1}(\rho(\tilde{s}))=\alpha_{U}^{-1}(\rho(\tilde{s})))$.  We know that $I_{\overline{\phi}(\rho(\tilde{s}))}$ is isomorphic to an inertia group of $\Pi_{U}$ by Lemma \ref{decisom2}. (Here, we use the assumption ``$r\neq 1$''.) For any finite extension field $k'$ over $k$, we have that  $H^{1}(G_{k'},\hat{\mathbb{Z}}^{p'}(1))=k'^{\times}$. Thus,    we obtain that  $I_{\overline{\phi}(\rho(\tilde{s}))}\backslash \phi^{-1}(\phi(\tilde{s}))\cong \ilim{k'/k: \text{fin}}H^{1}(G_{k'},\hat{\mathbb{Z}}^{p'}(1))\xrightarrow{\sim}\overline{k}^{\times}$.  Hence  $|\alpha_{U}^{-1}(\rho(\tilde{s}))|>1$ follows.\\
(2) Let $S_{r}$ be the  symmetric group of degree $r$. Then we have a permutation action  $S_{r}\curvearrowright E(\overline{k})$.  Since $G_{k}\cong \hat{\mathbb{Z}}$ and the natural action $G_{k}\curvearrowright E(\overline{k})$ factors through the permutation action, $G$ acts trivially on $E(\overline{k})$. Hence $\overline{\phi}(\rho(\tilde{s}))$ is $\overline{k}^{G}$-rational.   By applying the  arguments of $U$ and $\alpha_{U}$ in (1) to $U_{\overline{k}^{G},x_{\tilde{s}}}$ and $\alpha_{U_{\overline{k}^{G}},x_{\tilde{s}}}$, the second assertion follows when $r\neq 1$, $2$.  When $r=2$, the second assertion is clearly true, since $\alpha_{U_{\overline{k}^{G}},x}$ is bijective for any $x\in E(\overline{k}^{G})$.\\
(3) The group $D_{\tilde{s}}$ is clearly  contained in $D_{\phi(\tilde{s})}$.   Since $G_{k}\cong \hat{\mathbb{Z}}$, there exist  an open subgroup $G$ of $G_{k}$ and a  section $s\in\GSect(G,\Pi_{U}^{(1)})$ satisfying  $\tilde{s}=[s]$ such that  $\text{Im}(s)\cdot  I_{\phi(\tilde{s})}=D_{\phi(\tilde{s})}$.   Hence $D_{\tilde{s}}$ coincide with  $D_{\phi(\tilde{s})}$.  (Note that, when $r<2$, the inertia group  is trivial.)
\end{proof}

The following is  the main result of this subsection.

\begin{proposition}\label{bidecoabel}
Assume that $(g_{1},r_{1})\neq (0,0)$, $(0,1)$, $(0,2)$, and that $m$ satisfies
\begin{equation*}
\begin{cases}
\ \ m\geq  2 & \ \ (\text{if } (g_{1},r_{1})\neq (0,3), (0,4))\\
\ \ m\geq 3 & \ \ (\text{if } (g_{1},r_{1})= (0,3), (0,4)).
\end{cases}
\end{equation*} 
 Let  $n\in\mathbb{Z}_{\geq 1}$ be an integer satisfying $m> n$.  Let  $\Phi: \Pi^{(m)}_{U_{1}}\xrightarrow{\sim}\Pi^{(m)}_{U_{2}}$ be  an isomorphism and   $\Phi^{m-n}: \Pi^{(m-n)}_{U_{1}}\xrightarrow{\sim}\Pi^{(m-n)}_{U_{2}}$   the isomorphism induced by $\Phi$  (Proposition \ref{Gkisomorphic}(1)). 
 \begin{enumerate}[(1)]
 \item   $\Phi^{m-n}$   preserves decomposition groups.
 \item $\Phi$ induces  a unique  bijection $\tilde{f}_{\Phi}^{m-n,\text{cl}}: \tilde{X}_{1}^{m-n,\cl}\xrightarrow{\sim}\tilde{X}_{2}^{m-n,\cl}$   such that   the  diagram 
\begin{equation}\label{actiondiagdeco2}
\vcenter{\xymatrix@R=20pt{
\Pi^{(m-n)}_{U_{1}}\ar[d]^{\Phi^{m-n}} \ar@{}[r]|{\curvearrowright} &  \tilde{X}_{1}^{m-n,\cl}\ar[d]^{\tilde{f}_{\Phi}^{m-n,\text{cl}}} \ar[r]&\text{Dec}(\Pi^{(m-n)}_{U_{1}}) \ar[d]^{\rho_{\Phi^{m-n}}}\\
\Pi^{(m-n)}_{U_{2}}\ar@{}[r]|{\curvearrowright} &  \tilde{X}_{2}^{m-n,\cl} \ar[r]&\text{Dec}(\Pi^{(m-n)}_{U_{2}}) \\
}}
\end{equation}
  is commutative, where    $\rho_{\Phi^{m-n}}$ stands for  the bijection induced by $\Phi^{m-n}$ and  $ \tilde{X}_{i}^{m-n,\cl} \to\text{Dec}(\Pi^{(m-n)}_{U_{i}})$ stands for  the natural map.  In particular, a bijection $f^{\text{cl}}_{\Phi}:X_{1}^{\text{cl}}\rightarrow X_{2}^{\text{cl}}$ is  induced by dividing $\tilde{f}_{\Phi}^{m-n,\text{cl}}$ by the actions in (\ref{actiondiagdeco2}).
\item   If, moreover, $(m,r_{1})\neq (2,1)$, then    $\tilde{f}_{\Phi}^{m-n,\text{cl}}(\tilde{U}_{1}^{m-n,\text{cl}})=\tilde{U}_{2}^{m-n,\text{cl}}$ holds.  In particular, $f_{\Phi}^{\text{cl}}(U_{1}^{\text{cl}})=U_{2}^{\text{cl}}$ holds.
% \item Let $(m',n')\in\mathbb{Z}_{\geq m}\times \mathbb{Z}_{\geq 1}$ be a pair satisfying $m'-n'=m-n$. Let $\Phi': \Pi^{(m')}_{U_{1}}\xrightarrow{\sim}\Pi^{(m')}_{U_{2}}$ be  an isomorphism satisfying $\Phi'^{m}=\Phi$. Then  $\mathcal{F}_{\Phi'}^{m'-n'}=\mathcal{F}_{\Phi}^{m-n}$ holds.
\end{enumerate}
\end{proposition}

\begin{proof}
By  Proposition  \ref{severalinvrecoprop}, and  Proposition \ref{Gkisomorphic}, we obtain that $g_{1}=g_{2}$, $r_{1}=r_{2}$, $|k_{1}|=|k_{2}|$, and  $\Phi$ induces an isomorphism $G_{k_{1}}\rightarrow G_{k_{2}}$ which  preserves the Frobenius elements. By Lemma \ref{wflemma}, Proposition \ref{Gkisomorphic}(4),  and  Lemma \ref{7},   the natural bijection $\QSect(\Pi_{U_{1}}^{(m-n)})\xrightarrow{\sim}\QSect(\Pi_{U_{2}}^{(m-n)})$ induced by $\Phi$ induces a bijection  $\QSect^{\text{geom}}(\Pi_{U_{1}}^{(m-n)})\xrightarrow{\sim}\QSect^{\text{geom}}(\Pi_{U_{2}}^{(m-n)})$. \par
First, we consider the case that $m=2$ (note that this implies automatically that $n=1$, $(g_{1},r_{1})\neq (0,3)$, $(0,4)$).   By Proposition  \ref{inertiareco}(1), $\Phi$ induces a bijection $\text{Iner}(\Pi_{U_{1}}^{(1)})\xrightarrow{\sim}\text{Iner}(\Pi_{U_{2}}^{(1)})$.  Hence, by Lemma \ref{reconstructiondecoabel}(1)(2)(3),  $\Phi$ induces a bijection $\text{Dec}(\Pi_{U_{1}}^{(1)})\xrightarrow{\sim}\text{Dec}(\Pi_{U_{2}}^{(1)})$.  By Proposition \ref{sepprop}, we have that the natural map $\tilde{X}_{i}^{1,\text{cl}}\rightarrow \text{Dec}(\Pi_{U_{i}}^{(1)})$ is bijective. Thus, the assertions (1)(2) follow.  When $m=2$ and $r_{1}\neq 1$, the assertion (3) follows from Lemma \ref{reconstructiondecoabel}(1).\par
Next, we consider general $m$. Set $\mathcal{Q}_{i}:=\{H\overset{\text{op}}\lhd \Pi^{(m)}_{U_{i}}\mid \overline{\Pi}^{[m-(n+1)]}_{U_{i}}/\overline{\Pi}^{[m]}_{U_{i}}\subset H, (g(U_{i,H}),r(U_{i,H}))\neq (0,3), (0,4),\ r(U_{i,H})\neq 1\}$.    Fix  an element $H_{1}\in\mathcal{Q}_{1}$ and set $H_{2}:=\Phi(H_{1})$. By the case that $m=2$,  the isomorphism $H_{1}^{(n+1)}\xrightarrow{\sim}H_{2}^{(n+1)}$ induced by $\Phi$  induces a bijection $\rho: \text{Dec}(H_{1}^{(1)})\rightarrow \text{Dec}(H_{2}^{(1)})$ and a unique bijection  $\tilde{X}_{1,H_{1}}^{1,\cl}\xrightarrow{\sim}\tilde{X}_{2,H_{2}}^{1,\cl}$   such that   the  diagram 
\begin{equation*}
\vcenter{\xymatrix@R=15pt{
\tilde{X}_{1,H_{1}}^{1,\cl}\ar[d]\ar[r]^{\sim}&\text{Dec}(H_{1}^{(1)}) \ar[d]^{\rho}\\
\tilde{X}_{2,H_{2}}^{1,\cl} \ar[r]^{\sim}&\text{Dec}(H_{2}^{(1)}) \\
}}
\end{equation*}
is commutative. Since $\overline{\Pi}^{[m-(n+1)]}_{U_{i}}/\overline{\Pi}^{[m]}_{U_{i}}\subset H_{i}$, $H_{i}^{(1)}$ is a subquotient of $\Pi_{U_{i}}^{(m-n)}$, hence we have the natural map $\text{Dec}(\Pi_{U_{i}}^{(m-n)})\rightarrow \text{Dec}(H_{i}^{(1)})$. Let $H'_{1}$ be an element of $\mathcal{Q}_{1}$ satisfying $H_{1}'\subset H_{1}$.  Set $H'_{2}:=\Phi(H'_{1})$.  For any  decomposition group $D'$ of $H'^{(1)}_{i}$,  there exists a unique decomposition group $D$ of $H^{(1)}_{i}$ such that $D$ contains the image of $D'$ by $H'_{i}\rightarrow H_{i}$ by Lemma \ref{seplemmacase10}(c)$\Rightarrow$(a). 
Hence we obtain  a map  $\text{Dec}(H'^{(1)}_{i})\rightarrow\text{Dec}(H^{(1)}_{i})$, sending $D'$ to the unique  element containing the image of $D'$ in $H_{i}^{(1)}$. (Note that this map is compatible with the actions of  $\Pi_{U_{i}}^{(m-n)}/\overline{H}_{i}'^{[1]}$ as $H_{i}$ and $H'_{i}$ are normal in $\Pi_{U_{i}}^{(m)}$.)   By construction of these maps,  the diagram
\begin{equation*}
\vcenter{\xymatrix@R=15pt{
\tilde{X}_{i}^{m-n,\cl}\ar[d]\ar@/_28pt/[dd]\ar[r]^{\sim}&\text{Dec}(\Pi_{U_{i}}^{(m-n)})\ar[d]\ar@/^28pt/[dd]\\
\tilde{X}_{i,H'_{i}}^{1,\cl} \ar[d]\ar[r]^{\sim}&\text{Dec}(H'^{(1)}_{i}) \ar[d]\\
\tilde{X}_{i,H_{i}}^{1,\cl} \ar[r]^{\sim}&\text{Dec}(H_{i}^{(1)}) \\
}}
\end{equation*}
is commutative, where  the right-hand vertical maps and   the  horizontal maps  are the natural maps and the upper-horizontal map is bijective by Proposition \ref{sepprop}.  By Lemma \ref{gandr}, $\mathcal{Q}_{i}$ is cofinal in the set of all open normal subgroups of  $\Pi^{(m)}_{U_{i}}$.  Since $\underset{H\in\mathcal{Q}_{i}}\cap H^{[1]}=(\overline{\Pi}^{[m-(n+1)]}_{U_{i}}/\overline{\Pi}^{[m]}_{U_{i}})^{[1]}=\overline{\Pi}^{[m-n]}_{U_{i}}/\overline{\Pi}^{[m]}_{U_{i}}$, we have that $ \tilde{X}_{i}^{m-n,\text{cl}}\xrightarrow{\sim} \plim{H\in\mathcal{Q}_{i}}\tilde{X}_{i,H}^{1,\text{cl}}$. 
Thus,  we obtain a bijection  $\tilde{X}_{i}^{m-n,\text{cl}}(\xrightarrow{\sim}\text{Dec}(\Pi_{U_{i}}^{(m-n)}))\xrightarrow{\sim}\plim{H\in\mathcal{Q}_{i}}\text{Dec}(H_{i}^{(1)})$ which  is compatible with the actions of $\Pi_{U_{i}}^{(m-n)}$ on $\tilde{X}_{i}^{m-n,\text{cl}}$ and $\plim{H\in\mathcal{Q}_{i}}\text{Dec}(H_{i}^{(1)})$. Hence  there exists  a bijection $\tilde{X}_{1}^{m-n,\cl}\xrightarrow{\sim}\tilde{X}_{2}^{m-n,\cl}$ such that (\ref{actiondiagdeco2}) is commutative.  Therefore, the assertions (1)(2) follow.  The assertion  (3)  follows from the case that $m=2$.
\end{proof}

%%%%%%%%%%%%%%%%%%%%%%%%%%%%%%%%%%%%%%%%%%%%%%%%%%%%%%%%%%%%%%%%%%%%%%%%%%%%%%%%%%%%%%%%%%%%%
%%%%%%%%%%%%%%%%%%%%%%%%%%%%%%%%%%%%%           %%%%%%%%%%%%%%%%%%%%%%%%%%%%%            %%%%%%%%%%%%%%%
%%%%%%%%%%%%%%%%%%%%%%%%%%%%%%%%%%%       %%      %%%%%%%%%%%%%%%%%%%%%%%%%      %%%        %%%%%%%%%%%%%%
%%%%%%%%%%%%%%%%%%%%%%%%%%%%%%%%%%%%%%%%%%      %%%%%%%%%%%%%%%%%%%%%%%%      %%%%%    %%%%%%%%%%%%%%
%%%%%%%%%%%%%%%%%%%%%%%%%%%%%%%%%%%%%%%%%%%     %%%%%%%%%%%%%%%%%%%%%%%%%%%%%%%    %%%%%%%%%%%%%%%
%%%%%%%%%%%%%%%%%%%%%%%%%%%%%%%%%%%%%%%%%%      %%%%%%%%%%%%%%%%%%%%%%%%%%%         %%%%%%%%%%%%%%%%%
%%%%%%%%%%%%%%%%%%%%%%%%%%%%%%%%%%%%%%%%    %%%%%%%%%%%%%%%%%%%%%%%%%%%%%%%%%    %%%%%%%%%%%%%%%%
%%%%%%%%%%%%%%%%%%%%%%%%%%%%%%%%%%%%%%    %%%%%%%%%%%%%%%%%%%%%%%%%%%%%   %%%%%     %%%%%%%%%%%%%%%
%%%%%%%%%%%%%%%%%%%%%%%%%%%%%%%%%%%%%    %%%%%%%%%%%%%%%%%    %%%%%%%%%%%     %%%       %%%%%%%%%%%%%%%%
%%%%%%%%%%%%%%%%%%%%%%%%%%%%%%%%%%%%                      %%%%%%%%     %%%%%%%%%%%%              %%%%%%%%%%%%%%%%
%%%%%%%%%%%%%%%%%%%%%%%%%%%%%%%%%%%%%%%%%%%%%%%%%%%%%%%%%%%%%%%%%%%%%%%%%%%%%%%%%%%%%%%%%%%%%%%

\subsection{The weak bi-anabelian results over finite fields}\label{subsectionfinweak}
\hspace{\parindent}In this subsection,  we show  the  weak bi-anabelian $m$-step solvable Grothendieck conjecture for  affine hyperbolic curves over   finite fields. In other words, we show that  $\Pi^{(m)}_{U_{1}}\xrightarrow{\sim} \Pi^{(m)}_{U_{2}}$ implies $U_{1}\xrightarrow{\sim}U_{2}$ (under  certain assumptions on $(m,g,r)$, see Theorem \ref{finGCweak}). \par

Let   $\text{ord}_{v}: K(U)^{\times}\twoheadrightarrow \mathbb{Z}$ be the unique surjective valuation associated to $v\in X^{\cl}$ and $K(U)_{v}$ the $v$-adic completion of $K(U)$. We also write $\text{ord}_{v}$ for the surjective valuation $K(U)_{v}^{\times}\twoheadrightarrow \mathbb{Z}$ induced by $\text{ord}_{v}:K(U)^{\times}\twoheadrightarrow \mathbb{Z}$.  Let   $O_{X,v}:=\left\{a\in K(U)\mid \text{ord}_{v}(a)\geq 0\right\}$ be  the  valuation ring of $K(U)$ at $v$, $O_{v}:=\left\{a\in K(U)_{v}\mid \text{ord}_{v}(a)\geq 0\right\}$ the  valuation ring of $K(U)_{v}$, $m_{X,v}$ the maximal ideal of  $O_{X,v}$, and $m_{v}$ the maximal ideal of  $O_{v}$. We have  $\Gamma(U, O_{X})=\left\{ a\in K(U)\mid \text{ord}_{v}(a)\geq 0\text{ for each } v\in U^{\text{cl}}\right\}$.  The following lemma is  shown in \cite{Ta1997} section $4$ where  $\Pi_{U}^{(m)}$ is replaced by $\Pi_{U}$. The case of $\Pi_{U}^{(m)}$ can be settled  by using  Proposition \ref{bidecoabel}.

\begin{lemma}\label{reconstructionlemmaabel}  Assume that  $U_{1}$ is affine  hyperbolic and that
\begin{equation}\label{assumptionmg}
\begin{cases}
\ \ m\geq  2 & \ \  (\text{if }r_{1}\geq 2\text{ and  } (g_{1},r_{1})\neq (0,3), (0,4))\\
\ \ m\geq 3 & \ \  (\text{if } r_{1}<2\text{ or  } (g_{1},r_{1})= (0,3), (0,4)).
\end{cases}
\end{equation} 
Let  $\Phi: \Pi^{(m)}_{U_{1}}\xrightarrow{\sim}\Pi^{(m)}_{U_{2}}$ be  an isomorphism. Then the following hold.
\begin{enumerate}[(1)]
\item $\Phi$ induces a natural   isomorphism of multiplicative groups $F(\Phi):K(U_{1})^{\times}\rightarrow K(U_{2})^{\times}$ such that,  for each $v\in X_{1}^{\cl}$, the diagram
\begin{equation*}
\vcenter{
\xymatrix@C=60pt{
K(U_{1})^{\times}\ar[r]^{\text{ord}_{v}} \ar[d]^{F(\Phi)}& \mathbb{Z}\ar@{}[d]|*=0[@]{\rotatebox{90}{$=$}} \\
K(U_{2})^{\times}\ar[r]^{\text{ord}_{f^{\text{cl}}_{\Phi}(v)}}& \mathbb{Z}
}
}
\end{equation*}
is commutative. Here, $f^{\text{cl}}_{\Phi}$ stands for the bijection $X_{1}^{\text{cl}}\xrightarrow{\sim} X_{2}^{\text{cl}}$ defined in   Proposition \ref{bidecoabel}(2). Moreover, $F(\Phi)$ does not depend on $m$. (In other words, if $\Phi': \Pi^{(m')}_{U_{1}}\xrightarrow{\sim}\Pi^{(m')}_{U_{2}}$ is an isomorphism for some $m'\geq m$ and the isomorphism induced by $\Phi'$  on $\Pi^{(m)}_{U_{1}}$ (Proposition \ref{Gkisomorphic}(1)) coincides with $\Phi$, then $F(\Phi')=F(\Phi)$ holds.)  
\item   $F(\Phi)(1+m_{X_{1},v})=1+m_{X_{2},f^{\text{cl}}_{\Phi}(v)}$ for each $v\in E_{1}$.
\end{enumerate}
\end{lemma}
\begin{proof} 
By  Proposition  \ref{severalinvrecoprop}, and  Proposition \ref{Gkisomorphic}, we obtain that $g_{1}=g_{2}$, $r_{1}=r_{2}$, $|k_{1}|=|k_{2}|$, and  $\Phi$ induces an isomorphism $G_{k_{1}}\rightarrow G_{k_{2}}$ which  preserves the Frobenius elements. By Proposition \ref{bidecoabel}(1), $\Phi$ induces a bijection between  $\text{Dec}(\Pi_{U_{1}}^{(m-1)})$ and  $\text{Dec}(\Pi_{U_{2}}^{(m-1)})$. \\
(1) Let  $v$ be a closed point of $X$, $\tilde{v}^{m-1}$  an inverse image of $v$ in $\tilde{X}^{m-1}$, and  $\tilde{v}$ an inverse image of $\tilde{v}^{m-1}$ in  $\tilde{X}$. We have that the  natural projection $D_{\tilde{v},\Pi_{U}}\twoheadrightarrow D_{\tilde{v}^{m-1},\Pi_{U}^{(m-1)}}$ is an isomorphism by Lemma \ref{decisom2}. (When $m=2$, we need $r\geq 2$ here.) In particular, we obtain that $D_{\tilde{v},\Pi_{U}}^{\text{ab}}\xrightarrow{\sim}D^{\text{ab}}_{\tilde{v}^{m-1},\Pi^{\text{(m-1)}}_{U}}$.  Let   $F_{v}$ be the inverse image of  the subgroup  $\langle\text{Fr}_{k}\rangle$ by   $D^{\text{ab}}_{\tilde{v}^{m-1},\Pi^{\text{(m-1)}}_{U}}\rightarrow G_{k}$.    By  class field theory, we get
\begin{equation}\label{fnogun}
F_{v}\xleftarrow{\sim}
\begin{cases}
K(U)_{v}^{\times}/O_{v}^{\times} & (\text{if }v\in U)\\
K(U)_{v}^{\times}/(1+m_{v}) & (\text{if }v\in E),
\end{cases}
\end{equation}
where the isomorphism is induced by the local reciprocity isomorphism $\widehat{K(U)^{\times}_{v}} \xrightarrow{\rho_{v}}G_{K(U)_{v}}^{\ab}$. Further, we define the following group.
\begin{equation*}
K(\Pi_{U}^{(m)})^{\times}:=\Ker(\prod_{v\in X^{\cl}}{}^{'} F_{v}\rightarrow \Pi_{U}^{\ab})
\end{equation*}
Here,  $\prod ' F_{v}$ stands for the restricted direct product  of $F_{v}$ ($v\in X^{\cl}$) with respect to $\Ker(F_{v}\rightarrow G_{k})$ (which turns out to coincide with the direct sum of $F_{v}$ ($v\in X^{\text{cl}}$).  By definition and global class field theory, we obtain the following commutative diagram
\begin{equation*}\label{gcft}
\vcenter{
\xymatrix{
1 \ar[r]& K(U)^{\times} \ar[r]\ar@{.>}[d] & \mathbb{A}_{K(U)}^{\times}\ar[r] \ar@{->>}[d]&G_{K(U)}^{\ab}\ar@{->>}[d]\\
1 \ar[r]&K(\Pi_{U}^{(m)})^{\times} \ar[r]& \prod_{v\in X^{\cl}}' F_{v}\ar[r]&\Pi_{U}^{\ab},
}
}
\end{equation*}
where  $\mathbb{A}_{K(U)}^{\times}$ is the idele group of $K(U)$ (i.e. the  restricted direct product  of $K(U)^{\times}_{v}$ ($v\in X^{\cl}$) with respect to $O_{v}^{\times}$).  The lower horizontal sequence is  exact by definition. The upper  horizontal sequence is  exact by global class field theory.  The left-hand vertical arrow  turns out to be an isomorphism (by the assumption that $U_{1}$ is affine). The first assertion follows from this isomorphism and the commutativity of the  following diagram.
\begin{equation*}\label{e1}
\vcenter{
\xymatrix{
 K(U)^{\times}\ar@{->>}[rr]^{\text{ord}_{v}} \ar[d]_{\rotatebox{90}{$\sim$}}&& \mathbb{Z}\ar[d]_{\rotatebox{90}{$\sim$}}\\
 K(\Pi_{U}^{(m)})^{\times}\ar[r] 					& F_{v}\ar@{->>}[r]				& F_{v}/\Ker(F_{v}\rightarrow G_{k}).
}
}
\end{equation*}
Here, the right-hand vertical arrow stands for the morphism  $\mathbb{Z}\rightarrow F_{v}/\Ker(F_{v}\rightarrow G_{k})(\hookrightarrow G_{k}), 1\mapsto \text{Fr}_{\kappa(v)}(=\text{Fr}_{k}^{[G_{k}:G_{\kappa(v)}]})$. The second assertion follows from the construction of $F(\Phi)$. \\
(2) Let $v$ be an element of $E$. Then, by  the isomorphism $K(U)^{\times}\xrightarrow{\sim}K(\Pi_{U}^{(m)})^{\times}$, the subgroup $1+m_{X,v}\subset K(U)^{\times}$ corresponds to  $\Ker(K(\Pi^{(m)}_{U})^{\times}\rightarrow F_{v})\subset K(\Pi_{U}^{(m)})^{\times}$. Hence the assertion follows.
\end{proof}

We define $K(\Pi_{U}^{(m)}):=K(\Pi_{U}^{(m)})^{\times}\cup\left\{ *\right\}$. By Lemma \ref{reconstructionlemmaabel},  we obtain  an isomorphism  of  multiplicative monoids $F(\Phi):K(U_{1})\xrightarrow{\sim} K(U_{2})$ (with $0\mapsto 0$)  under the assumption of Lemma \ref{reconstructionlemmaabel}.

\begin{lemma}\label{prop1.14}
Assume that   $U_{1}$ is affine  hyperbolic.  Let $n\in\mathbb{Z}_{\geq 0}$ be an integer satisfying  $m\geq n$. Let $H_1$, $H'_{1}$ be open subgroups of  $\Pi^{(m)}_{U_{1}}$ that  satisfy    $\overline{\Pi}_{U_{1}}^{[m-n]}/\overline{\Pi}_{U_{1}}^{[m]} \subset H_{1}'\subset H_{1}$.  We assume that $(n,g(U_{H_{1}}),r(U_{H_{1}}))$  satisfies  the assumption for  $(m,g_{1},r_{1})$ in (\ref{assumptionmg}). (Thus,  $(n,g(U_{H'_{1}}),r(U_{H'_{1}}))$ satisfies the same assumption, since  $H_{1}'\subset H_{1}$.)   Let   $\Phi: \Pi^{(m)}_{U_{1}}\xrightarrow{\sim}\Pi^{(m)}_{U_{2}}$ be an isomorphism, $H_{2}:=\Phi(H_{1})$, and $H_{2}':=\Phi(H_{1}')$. Then   the  following diagram is commutative.
\begin{equation*}\vcenter{
\xymatrix@C=60pt{
K(U_{1,H'_{1}}) \ar[r]^-{F(\Phi|_{H'_{1}})} & K(U_{2,H'_{2}})\\
K(U_{1,H_{1}}) \ar[r]^{F(\Phi|_{H_{1}})}\ar@{^{(}-_>}[u]  & K(U_{2,H_{2}})\ar@{^{(}-_>}[u] 
}
}
\end{equation*} 
Here,   $F(\Phi|_{H_{1}})$ (resp. $F(\Phi|_{H'_{1}})$ ) stands for  the isomorphism  of  multiplicative monoids  induced by  the isomorphism $H_{1}^{(n)}\xrightarrow{\sim} H_{2}^{(n)} $ (resp.  $H'_{1}{}^{(n)}\xrightarrow{\sim} H'_{2}{}^{(n)}$) (see Lemma \ref{reconstructionlemmaabel}(1)). 
\end{lemma}
\begin{proof}
Let  $\tilde{v}_{i}$ be an element of $\tilde{X}_{i,H_{i}}$ and   $\tilde{v}^{n}_{i}$  (resp. $v_{i}$, resp. $v'_{i}$)  an image of $\tilde{v}_{i}$ in  $\tilde{X}_{i,H_{i}}^{n}$ (resp.   $X_{i,H_{i}}$, resp.  $X_{i,H'_{i}}$).  By Lemma \ref{decisom2}, we obtain that $D_{\tilde{v}_{i},\Pi_{U_{i,H_{i}}}}^{\text{ab}}\xrightarrow{\sim}D^{\text{ab}}_{\tilde{v}_{i}^{n},H_{i}^{(n)}}$.  The transfer homomorphism  $D_{\tilde{v}_{i},\Pi_{U_{i,H_{i}}}}^{\text{ab}}\rightarrow D_{\tilde{v}_{i},\Pi_{U_{i,H'_{i}}}}^{\text{ab}}$    yields the natural injection $F_{v_{i}}\hookrightarrow F_{v'_{i}}$ (cf. \cite{Sa2021} section  2), where $F_{v_{i}}$ and $F_{v'_{i}}$ are defined in (\ref{fnogun}). The assertion follows from this and the  various constructions. 
\end{proof}

Next, we consider  the addition of  $K(U)$.

\begin{lemma}\label{Talem}
Let  $i=1,2$. Let  $t_{i}$ be an algebraically closed field, $Y_{i}$ a proper, smooth, connected curve over $t_{i}$. Let   $T_{i}$ be a  subset of   $(Y_{i})^{\cl}$. Assume that we are given an isomorphism $\overline{F}:K(Y_{1})\rightarrow K(Y_{2})$ as multiplicative monoids and a bijection $\overline{f}:(Y_{1})^{\cl}\rightarrow (Y_{2})^{\cl}$ with $\overline{f}(T_{1})=T_{2}$, satisfying the following conditions.
\begin{enumerate}[(i)]
\item For each $P\in (Y_{1})^{\cl}$, the following diagram is commutative.
\begin{equation*}
\vcenter{
\xymatrix@C=60pt{
K(Y_{1})^{\times}\ar[r]^{\text{ord}_{P}} \ar[d]^{\overline{F}}& \mathbb{Z}\ar@{}[d]|*=0[@]{\rotatebox{90}{$=$}} \\
K(Y_{2})^{\times}\ar[r]^{\text{ord}_{\overline{f}(P)}}& \mathbb{Z}
}
}
\end{equation*}
\item For each $P\in T_{1}$, $\overline{F}(1+m_{Y_{1},P})=1+m_{Y_{2},\overline{f}(P)}$.
\item $|T_{1}|\geq 3$.
\end{enumerate}
Then $\overline{F}:K(Y_{1})\rightarrow K(Y_{2})$ is additive.
\end{lemma}
\begin{proof}
See \cite{Ta1997} Lemma 4.7.
\end{proof}

We obtain the first main result of this section.

\begin{theorem}[Weak bi-anabelian result over finite fields]\label{finGCweak}
 Assume that $U_{1}$ is  affine hyperbolic and that $m$ satisfies
\begin{equation*}
\begin{cases}
\ \ m\geq  2 & \ \ (\text{if }r_{1}\geq  3 \text{ and }(g_{1},r_{1})\neq (0,3), (0,4))\\
\ \ m\geq 3 & \ \ (\text{if }r_{1}<3 \text{ or } (g_{1},r_{1})= (0,3), (0,4))
\end{cases}
\end{equation*}
 (see Notation  of section \ref{sectionfin}). Let  $\Phi: \Pi^{(m)}_{U_{1}}\xrightarrow{\sim}\Pi^{(m)}_{U_{2}}$ be an isomorphism.  Let $\mathcal{Q}_{1}$ be the set of all  finite  extensions  of $k_{1}$ and set  $\Phi_{L_{1}}:=\Phi|_{\Pi^{(m)}_{U_{1,L_{1}}}}$. Then the  multiplicative monoid isomorphism $F(\Phi):K(U_{1})\xrightarrow{\sim}K(U_{2})$   is additive and $\{F(\Phi_{L_{1}})\}_{L_{1}\in\mathcal{Q}_{1}}$ induces  scheme isomorphisms $\tilde{f}^{0}_{\Phi}:\tilde{X}^{0}_{1}\xrightarrow{\sim} \tilde{X}^{0}_{2}$ and $f_{\Phi}:X_{1}\xrightarrow{\sim} X_{2}$ which satisfy the following conditions (i)-(iii).
 \begin{enumerate}[(i)]
 \item The isomorphisms $\tilde{f}^{0}_{\Phi}$,  $f_{\Phi}$ induce isomorphisms $\tilde{U}^{0}_{1}\xrightarrow{\sim}\tilde{U}_{2}^{0}$, $U_{1}\xrightarrow{\sim}U_{2}$, respectively.
 \item   The maps $\tilde{f}^{0}_{\Phi}|_{\tilde{X}_{1}^{0,\text{cl}}}$, $f_{\Phi}|_{X_{1}^{\text{cl}}}$  coincide with  the bijections $\tilde{f}^{0,\text{cl}}_{\Phi}: \tilde{X}^{0,\text{cl}}_{1}\xrightarrow{\sim} \tilde{X}^{0,\text{cl}}_{2}$, $f^{\text{cl}}_{\Phi}: X^{\text{cl}}_{1}\xrightarrow{\sim} X^{\text{cl}}_{2}$   defined in   Proposition \ref{bidecoabel}(2), respectively.
 \item Let $\Phi^{\text{ab}}$ be the element  of $\Isom(\Pi_{U_{1}}^{\text{ab}},\Pi_{U_{2}}^{\text{ab}})$ induced by $\Phi$. Then the image of $f_{\Phi}|_{U_{1}}:U_{1}\xrightarrow{\sim}U_{2}$ by the natural map $
 \Isom(U_1,U_2)\rightarrow \Isom(\Pi_{U_{1}}^{\text{ab}},\Pi_{U_{2}}^{\text{ab}})$ coincides with $\Phi^{\text{ab}}$. 
 \item Let  $\overline{\Phi}^{1}$ be an element of $\Isom(\overline{\Pi}_{U_{1}}^{1},\overline{\Pi}^{1}_{U_{2}})$ induced by $\Phi$. Then the image of $\tilde{f}_{\Phi}^{0}|_{\tilde{U}^{0}_{1}}: \tilde{U}^{0}_{1}\xrightarrow{\sim}\tilde{U}^{0}_{2}$ by the natural map $
 \Isom(\tilde{U}^{0}_{1},\tilde{U}^{0}_{2})\rightarrow \Isom(\overline{\Pi}_{U_{1}}^{1},\overline{\Pi}_{U_{2}}^{1})$ coincides with $\overline{\Phi}^{1}$. 
  \end{enumerate}
\noindent  In particular, the following holds.
\begin{equation}\label{eq123}
\Pi_{U_{1}}^{(m)}\xrightarrow[]{\sim} \Pi^{(m)}_{U_{2}}\iff U_{1}\xrightarrow[\text{scheme}]{\sim}U_{2}
\end{equation}
\end{theorem}

\begin{proof}
By  Proposition  \ref{severalinvrecoprop}, and  Proposition \ref{Gkisomorphic}, we obtain that $g_{1}=g_{2}$, $r_{1}=r_{2}$, $|k_{1}|=|k_{2}|$, and  $\Phi$ induces an isomorphism $G_{k_{1}}\xrightarrow{\sim} G_{k_{2}}$ which  preserves the Frobenius elements.\par
 Let $\mathcal{Q}_{2}$ be the set of all  finite  extensions  of $k_{2}$.  The isomorphism $G_{k_{1}}\xrightarrow{\sim} G_{k_{2}}$ induced by $\Phi$ (Proposition \ref{Gkisomorphic}(1)) induces a  bijection   $\rho: \mathcal{Q}_{1}\xrightarrow{\sim}\mathcal{Q}_{2}$. For each $i=1$, $2$ and each $P\in  (X_{i,\overline{k}_{i}})^{\cl}$, we have  $K(U_{i,\overline{k}_{i}})=\ilim{L_{i}\in\mathcal{Q}_{i}} K(U_{i,L_{i}})$, $\text{ord}_{P}=\plim{L_{i}\in\mathcal{Q}_{i}}\text{ord}_{P_{L_{i}}}$,  and $1+m_{X_{i,\overline{k}_{i}},P}=\ilim{L_{i}\in\mathcal{Q}_{i}}(1+m_{X_{i,L_{i}},P_{L_{i}}})$, where $P_{L_{i}}\in (X_{i,L_{i}})^{\cl}$ stands for  the image of $P$.  By Lemma  \ref{reconstructionlemmaabel}(1) and   Lemma \ref{prop1.14}, we obtain  an isomorphism  of  multiplicative monoids
\begin{equation}\label{geomisommap}
K(U_{1,\overline{k}_{1}})=\ilim{L_{1}\in\mathcal{Q}_{1}} K(U_{1,L_{1}})\rightarrow \ilim{L_{1}\in\mathcal{Q}_{1}}K(U_{2,\rho(L_{1})})=K(U_{2,\overline{k}_{2}}).
\end{equation}\par
First, we consider the case that    $r_{1}\geq 3$ and $(g_{1},r_{1})\neq (0,3), (0,4)$. Then, by  Lemma  \ref{reconstructionlemmaabel}(1)(2) and Lemma \ref{Talem},   the multiplicative monoid isomorphism (\ref{geomisommap}) is additive.    Hence we obtain an isomorphism $\tilde{f}^{0}_{\Phi}:\tilde{X}^{0}_{1}\xrightarrow{\sim}\tilde{X}^{0}_{2}$. As $F(\Phi):K(U_{1})\xrightarrow{\sim}K(U_{2})$ is  a restriction of (\ref{geomisommap}),   $F(\Phi)$ is also additive.  Hence we obtain an isomorphism $f_{\Phi}:X_{1}\xrightarrow{\sim}X_{2}$. By construction, $\tilde{f}^{0}_{\Phi}$ and $f_{\Phi}$ satisfy (ii).  By Proposition \ref{bidecoabel}(3) and (ii),  $\tilde{f}^{0}_{\Phi}$ and $f_{\Phi}$ also satisfy (i). \par 
Next, we  consider general $(g_{1},r_{1})$. By Lemma \ref{gandr}, there exists an open subgroup $H_{1}\subset \Pi_{U_{1}}^{(m)}$  containing $\overline{\Pi}_{U_{1}}^{[m-2]}/\overline{\Pi}_{U_{1}}^{[m]}$ such  that   $r(U_{1,H_{1}})\geq 3$ and  that   $(g(U_{1,H_{1}}),r(U_{1,H_{1}}))\neq (0,3), (0,4)$.  We obtain that  $\tilde{F}^{0}(\Phi|_{H_{1}})$ and $F(\Phi|_{H_{1}})$ are additive, that  $\tilde{F}^{0}(\Phi|_{H_{1}})(\Gamma(U_{1,H_1,\overline{k}}, O_{X_{1,H_1,\overline{k}}}))=\Gamma(U_{2,\Phi(H_1),\overline{k}}, O_{X_{2,\Phi(H_1),\overline{k}}})$, and that $F(\Phi|_{H_{1}})(\Gamma(U_{1,H_1}, O_{X_{1,H_1}}))=\Gamma(U_{2,\Phi(H_1)}, O_{X_{2,\Phi(H_1)}})$, where  $\tilde{F}^{0}(\Phi|_{H_{1}})$ and $F(\Phi|_{H_{1}})$ stand for the isomorphisms  of  multiplicative monoids  $K(U_{1,H_{1},\overline{k}})\xrightarrow{\sim} K(U_{2,\Phi(H_{1}),\overline{k}})$ and $K(U_{1,H_{1}})\xrightarrow{\sim} K(U_{2,\Phi(H_{1})})$ induced by  $H_{1}^{(2)}\xrightarrow{\sim} \Phi(H_{1})^{(2)} $, respectively (see Lemma \ref{reconstructionlemmaabel}(1)).  Hence,   by Lemma  \ref{prop1.14},  we obtain  isomorphisms $\tilde{f}^{0}_{\Phi}:\tilde{X}^{0}_{1}\xrightarrow{\sim}\tilde{X}^{0}_{2}$ and  $f_{\Phi}:X_{1}\xrightarrow{\sim}X_{2}$ satisfying the condition (i). By  construction of $F(\Phi)$, $\tilde{f}^{0}_{\Phi}$ and  $f_{\Phi}$ satisfy the condition (ii). Hence the equivalence  in  (\ref{eq123})  follows. (Note that the implication $\Leftarrow$ in (\ref{eq123}) is clear.)\par
Next, we show that $f_{\Phi}$ satisfies the condition (iii).     Let  $\tilde{f}_{\Phi}$ be the image of $f_{\Phi}$ by $\Isom(U_1,U_2)\rightarrow \Isom(\Pi_{U_{1}}^{\text{ab}},\Pi_{U_{2}}^{\text{ab}})$. By (ii),  we obtain that  $\Phi^{\text{ab}}(D_{v})=D_{f^{\text{cl}}_{\Phi}(v)}=\tilde{f}_{\Phi}(D_{v})$ for each $v\in U^{\text{cl}}_{1}$.  We set   $s_{v}:G_{\kappa(v)}\xrightarrow{\sim} D_{v}$. By  Proposition \ref{Gkisomorphic}(2), we get $\Phi^{\text{ab}}(s_{v}(\text{Fr}_{\kappa(v)}))=\tilde{f}_{\Phi}(s_{v}(\text{Fr}_{\kappa(v)}))$, where $\text{Fr}_{\kappa(v)}$ is the Frobenius element of $G_{\kappa(v)}$.  Since we have $\Pi_{U_{1}}^{\text{ab}}=\overline{\langle s_{v}(\text{Fr}_{\kappa(v)})\ |\ v\in U_{1}^{\text{cl}}\rangle}$ by  Chebotarev's density theorem, we obtain that $\tilde{f}_{\Phi}=\Phi^{\text{ab}}$. Thus, $f_{\Phi}$ satisfies the condition (iii). \par 
Finally, we show that $\tilde{f}^{0}_{\Phi}$ satisfies the condition (iv).  For any $L_{1}\in \mathcal{Q}_{i}$, $\tilde{f}^{0}_{\Phi}$ and $\Phi_{L_{1}}$ induces the same  isomorphism $\Pi_{U_{1,L_{1}}}^{\text{ab}}\xrightarrow{\sim}\Pi_{U_{2,L_{2}}}^{\text{ab}}$ by (iii). Since we have that  $\overline{\Pi}^{1}\xrightarrow{\sim}(\plim{L\in\mathcal{Q}_{i}}\Pi_{U_{i,L}})^{\text{ab}}\xrightarrow{\sim}\plim{L\in\mathcal{Q}_{i}}\Pi_{U_{i,L}}^{\text{ab}}$, (iv) follows.
\end{proof}

%%%%%%%%%%%%%%%%%%%%%%%%%%%%%%%%%%%%%%%%%%%%%%%%%%%%%%%%%%%%%%%%%%%%%%%%%%%%%%%%%%%%%%%%%%%%%
%%%%%%%%%%%%%%%%%%%%%%%%%%%%%%%%%%%%%           %%%%%%%%%%%%%%%%%%%%%%%%%%%%%            %%%%%%%%%%%%%%%
%%%%%%%%%%%%%%%%%%%%%%%%%%%%%%%%%%%       %%      %%%%%%%%%%%%%%%%%%%%%%%%%      %%%        %%%%%%%%%%%%%%
%%%%%%%%%%%%%%%%%%%%%%%%%%%%%%%%%%%%%%%%%%      %%%%%%%%%%%%%%%%%%%%%%%%      %%%%%    %%%%%%%%%%%%%%
%%%%%%%%%%%%%%%%%%%%%%%%%%%%%%%%%%%%%%%%%%%     %%%%%%%%%%%%%%%%%%%%%%%%%%%%%%%    %%%%%%%%%%%%%%%
%%%%%%%%%%%%%%%%%%%%%%%%%%%%%%%%%%%%%%%%%%      %%%%%%%%%%%%%%%%%%%%%%%%%%%         %%%%%%%%%%%%%%%%%
%%%%%%%%%%%%%%%%%%%%%%%%%%%%%%%%%%%%%%%%    %%%%%%%%%%%%%%%%%%%%%%%%%%%%%%%%%    %%%%%%%%%%%%%%%%
%%%%%%%%%%%%%%%%%%%%%%%%%%%%%%%%%%%%%%    %%%%%%%%%%%%%%%%%%%%%%%%%%%%%   %%%%%     %%%%%%%%%%%%%%%
%%%%%%%%%%%%%%%%%%%%%%%%%%%%%%%%%%%%%    %%%%%%%%%%%%%%%%%    %%%%%%%%%%%     %%%       %%%%%%%%%%%%%%%%
%%%%%%%%%%%%%%%%%%%%%%%%%%%%%%%%%%%%                      %%%%%%%%     %%%%%%%%%%%%              %%%%%%%%%%%%%%%%
%%%%%%%%%%%%%%%%%%%%%%%%%%%%%%%%%%%%%%%%%%%%%%%%%%%%%%%%%%%%%%%%%%%%%%%%%%%%%%%%%%%%%%%%%%%%%%%

\subsection{The strong bi-anabelian results over finite fields}\label{subsectionfinstrong}
\hspace{\parindent}In this subsection,  we show the strong bi-anabelian  $m$-step solvable Grothendieck conjecture for  affine hyperbolic curves over   finite fields, and  obtain  corollaries.
 
 \begin{lemma}\label{stxlemma1}
Let $X$ and $Y$ be schemes of finite type over $\text{Spec}(\mathbb{Z})$ and assume that $X$ is integral. Let  $f, g$: $X \rightarrow  Y$ be morphisms.  If  $f$ and $g$  coincide set-theoretically on the set of closed points of $X$, then one of the following conditions (a)-(b) holds.
\begin{enumerate}[(a)]
\item   $f = g$
\item $X$ is a scheme over $\mathbb{F}_p$ for some prime $p$, and there exists $a\in \mathbb{Z}$ such that either $a\geq 0$,  $f= g\circ \text{Fr}_{X}^{a}$, or $a<0$, $f\circ\text{Fr}^{-a}_{X}=g$. If, moreover, $f$ is not constant, then $a\in\mathbb{Z}$ is unique. 
\end{enumerate}
\end{lemma}
\begin{proof}
See the proof of   \cite{St2002P} Theorem 1.2.1. We remark that, in  the assertion of \cite{St2002P} Theorem 1.2.1, it is assumed  that $f$ and $g$ coincide as morphisms of topological spaces. However,  in the proof of \cite{St2002P} Theorem 1.2.1,  we only need the fact   that $f$ and $g$  coincide set-theoretically on the set of closed points of $X$ (cf.  \cite{St2002P} Proposition 1.2.4). 
\end{proof}

\begin{lemma}\label{injectiveofPi}
Assume that $U_{1}$ is hyperbolic. Then the natural map
\begin{equation*}
\Isom(\tilde{U}^{m}_{1}/U_1,\tilde{U}^{m}_2/U_2)\rightarrow \Isom(\Pi_{U_{1}}^{(m)},\Pi_{U_{2}}^{(m)}).
\end{equation*}
 is injective. 
\end{lemma}
\begin{proof}
If  $\Isom(\tilde{U}^{m}_{1}/U_{1},\tilde{U}^{m}_{2}/U_{2})=\emptyset$, then the assertion follows.  We may assume that  $\Isom(\tilde{U}^{m}_{1}/U_{1},\tilde{U}^{m}_{2}/U_{2})\not=\emptyset$ and that $(X_{1},E_{1})=(X_{2}.E_{2})$. Write  $U$ (resp. $X$, resp. $E$) instead of $U_{i}$ (resp. $X_{i}$, resp. $E_{i}$). Let $(\alpha_{\{1\}},\alpha)$ be an element of $ \Isom(\tilde{U}^{m}/U,\tilde{U}^{m}/U)$ which is mapped to the identity by the natural map $\rho:\Isom(\tilde{U}^{m}/U,\tilde{U}^{m}/U)\rightarrow \Aut(\Pi_{U}^{(m)})$.  Let $H$ be an open subgroup of $\Pi_{U}^{(m)}$, $U_{H}$ the  \'{e}tale covering over $U$ corresponding  to $H$,  and $\alpha_{H}$ the isomorphism $U_{H}\xrightarrow{\sim} U_{H}$ induced by $\alpha_{\{1\}}$.  Since $\rho(\alpha_{\{1\}},\alpha)$ preserves decomposition groups, we obtain that $\alpha_{\{1\}}(\tilde{v})=\tilde{v}$ for $\tilde{v}\in \tilde{U}^{m}$ by  Proposition \ref{sepprop}.  In particular, we get $\alpha_{H}(v)=v$ for $v\in U_{H}$. By Lemma \ref{stxlemma1}, this implies that there exists   $a_{H}\in \mathbb{Z}_{\geq 0}$ such that  $\alpha_{H}=\text{Fr}_{U_{H}}^{a_{H}}$. Since $\alpha_{H}\in\text{Aut}(U_{H})$, we obtain that $a_{H}=0$. Considering all open subgroups $H$, we obtain that  $(\alpha_{\{1\}},\alpha)$ is the identity.  Hence the assertion follows.
\end{proof}

\begin{definition}\label{m+nletter}
Let $n\in\mathbb{Z}_{\geq 0}$ be an integer satisfying $m\geq n$.  We define  $\Isom^{(m)}(\Pi_{U_{1}}^{(m-n)},\Pi_{U_{2}}^{(m-n)})$ as the image of the  map $\Isom( \Pi_{U_{1}}^{(m)},\Pi_{U_{2}}^{(m)})\rightarrow \Isom(\Pi_{U_{1}}^{(m-n)},\Pi_{U_{2}}^{(m-n)})$ induced by Proposition  \ref{Gkisomorphic}(1).
\end{definition}

The following theorem  is  the   second main result of this section.

\begin{theorem}[Strong bi-anabelian result  over finite fields]\label{finGCstrong}
Assume that $m\geq 3$ and  that  $U_{1}$ is  affine hyperbolic (see Notation of section \ref{sectionfin}).  Let  $n\in\mathbb{Z}_{\geq2}$ be an integer satisfying $m>n$. Then the natural map
\begin{equation*}
\Isom(\tilde{U}^{m-n}_{1}/U_1,\tilde{U}^{m-n}_2/U_2)\rightarrow \Isom^{(m)}(\Pi_{U_{1}}^{(m-n)},\Pi_{U_{2}}^{(m-n)})
\end{equation*}
  is  bijective.
\end{theorem}
\begin{proof}
The injectivity follows from Lemma \ref{injectiveofPi}.  We show the  surjectivity.  First, we construct  a map  $\mathcal{F}: \Isom(\Pi_{U_{1}}^{(m)},\Pi_{U_{2}}^{(m)})\rightarrow \Isom(\tilde{U}^{m-n}_{1}/U_1,\tilde{U}^{m-n}_{2}/U_2)$. Let  $\Phi$ be an element of $\Isom(\Pi_{U_{1}}^{(m)},\Pi_{U_{2}}^{(m)})$.  %(The following construction is also possible  and used in Lemma \ref{atonolemma} below  when    ``$m=2$, $n=1$, $r_{1}\geq 3$, and $(g_{1},r_{1})\neq (0,3), (0,4)$''.)
 Set $\mathcal{Q}_{1}:=\{H\overset{\op}\subset \Pi_{U_{1}}^{(m)} \mid \overline{\Pi}_{U_{1}}^{[m-n]}/\overline{\Pi}_{U_{1}}^{[m]}\subset H$, $r(U_{1, H})\geq 3$ and $(g(U_{1,H}),r(U_{1,H}))\neq (0,3), (0,4)\}$.   For any element  $H\in\mathcal{Q}_{1}$, we write $F(\Phi|_{H}) $ for   the isomorphism  of  multiplicative monoids $K(U_{1,H})\xrightarrow{\sim} K(U_{2,\Phi(H)})$ induced by  $H^{(n)}\xrightarrow{\sim} \Phi(H)^{(n)} $ (Lemma \ref{reconstructionlemmaabel}(1)). The multiplicative monoid isomorphism $F(\Phi|_{H}) $ is a field isomorphism by Theorem \ref{finGCweak}, as $n\geq 2$.  We know that  $\mathcal{Q}_{1}$ is cofinal in the set of all open subgroups of  $\Pi_{U_{1}}^{(m)}$ containing  $\overline{\Pi}_{U_{1}}^{[m-n]}/\overline{\Pi}_{U_{1}}^{[m]}$ by  Lemma \ref{gandr}. Hence   the  field isomorphisms $\{F(\Phi|_{H})\}_{H\in\mathcal{Q}_{1}}$ induce the following field isomorphism  by Lemma \ref{prop1.14}.
\begin{equation}\label{monoidisomuniv}
 \tilde{\mathcal{K}}^{m-n}(U_{1})=\ilim{H\in\mathcal{Q}_{1}} K(U_{1,H}) \rightarrow \ilim{H\in \mathcal{Q}_{1}} K(U_{2,\Phi(H)}) = \tilde{\mathcal{K}}^{m-n}(U_{2}).
\end{equation}
By Theorem \ref{finGCweak},  we obtain  that the  field isomorphism $F(\Phi|_{H})$ induces  a scheme isomorphism  $U_{1,H}\xrightarrow{\sim}U_{2,\Phi(H)}$ for any $H\in\mathcal{Q}_{1}$. Hence the isomorphism (\ref{monoidisomuniv}) induces a scheme isomorphism $\tilde{U}^{m-n}_{1}\xrightarrow{\sim}\tilde{U}^{m-n}_{2}$. We have that   $F(\Phi)$ is a restriction of (\ref{monoidisomuniv}). Hence the scheme isomorphism $\tilde{U}^{m-n}_{1}\xrightarrow{\sim}\tilde{U}^{m-n}_{2}$ induces  $U_{1}\xrightarrow{\sim}U_{2}$. Thus, we obtain the desired  map $\mathcal{F}: \Isom(\Pi_{U_{1}}^{(m)},\Pi_{U_{2}}^{(m)})\rightarrow \Isom(\tilde{U}^{m-n}_{1}/U_1,\tilde{U}^{m-n}_{2}/U_2).$\par
Next, we show that  the diagram 
\begin{equation}\label{commutativediag1}
\vcenter{
\xymatrix@C=80pt{
 && \Isom(\Pi_{U_{1}}^{(m)},\Pi_{U_{2}}^{(m)}) \ar[d]\ar[dll]_{\mathcal{F}}\\
\Isom(\tilde{U}^{m-n}_{1}/U_1,\tilde{U}^{m-n}_{2}/U_2)\ar[rr]_{\Pi^{(m-n)}(\cdot)} &\ & \Isom(\Pi_{U_{1}}^{(m-n)},\Pi_{U_{2}}^{(m-n)}) 
}
}
\end{equation}
is commutative, where the right-hand vertical arrow is induced by using  Proposition \ref{Gkisomorphic}(1).  Let   $\Phi^{m-n}$ be the image of $\Phi$ in $\Isom(\Pi_{U_{1}}^{(m-n)},\Pi_{U_{2}}^{(m-n)})$.  Let $s\in \Sect(G_{k_{1}},\Pi_{U_{1}}^{(m-n)})$ and  $\mathcal{Q}_{s}:=\{H\overset{\text{op}}\subset \Pi_{U_{1}}^{(m-n)}\mid $  $r(U_{1,H})\geq3$,   $(g(U_{1,H}),r(U_{1,H}))\neq (0,3), (0,4)$, and $s(G_{k_{1}})\subset  H $$\}$.   Fix   $H\in \mathcal{Q}_{s}$.  By construction, $\mathcal{F}(\Phi)$ restricts to the isomorphism $U_{1,H}\xrightarrow{\sim} U_{2,\Phi(H)}$ induced by $F(\Phi\mid_{H})$. We obtain that $(\Pi^{(m-n)}\circ\mathcal{F})(\Phi)(H)=\Phi^{m-n}(H)$.  By Lemma \ref{gandr}, for any open subgroup $H'$ of $\Pi_{U_{1}}^{(m-n)}$ containing $s(G_{k_{1}})$,  we can take a characteristic subgroup $\overline{H}''$ of $\overline{\Pi}_{U_{1}}^{m-n}$ that satisfies   $r(U_{1,\overline{H}''})\geq3$,   $(g(U_{1,\overline{H}''}),r(U_{1,\overline{H}''}))\neq (0,3), (0,4)$, and that $\overline{H}''\subset \overline{\Pi}_{U_{1}}^{m-n}\cap H'$. Hence $\mathcal{Q}_{s}$ is cofinal in the set of all open subgroups of $\Pi_{U_{1}}^{(m-n)}$ containing $s(G_{k_{1}})$. This implies that  $s(G_{k_{1}})=\underset{H\in\mathcal{Q}_{s}}\cap H$. Hence we obtain $(\Pi^{(m-n)}\circ\mathcal{F})(\Phi)(s(G_{k_{1}}))=\Phi^{m-n}(s(G_{k_{1}}))$.  Thus, we obtain that $(\Pi^{(m-n)}\circ\mathcal{F})(\Phi)(s(\text{Fr}_{k_{1}}))=\Phi^{m-n}(s(\text{Fr}_{k_{1}}))$ by Proposition \ref{Gkisomorphic}(2). Since $G_{k_{1}}\cong \hat{\mathbb{Z}}$, we have $\Pi_{U_{1}}^{(m-n)}=\overline{\langle s(\text{Fr}_{k_{1}})\ |\ s\in \Sect(G_{k_{1}},\Pi_{U_{1}}^{(m-n)})\rangle}$.   Therefore, we get $(\Pi^{(m-n)}\circ\mathcal{F})(\Phi)=\Phi^{m-n}$ and then the diagram (\ref{commutativediag1}) is commutative. Thus, the surjectivity follows.\par
\end{proof}

\begin{corollary}\label{finGCcorn}
Let the  assumption and the notation be as in   Theorem \ref{finGCstrong}.  Then the subset   $\Isom^{(m)}(\Pi_{U_{1}}^{(m-n)},\Pi_{U_{2}}^{(m-n)})$ of $\Isom(\Pi_{U_{1}}^{(m-n)},\Pi_{U_{2}}^{(m-n)})$  depends  only on $m-n$, not on $m$.
\end{corollary}
\begin{proof}
The assertion follows from Theorem \ref{finGCstrong}.  
\end{proof}

\begin{corollary}\label{finGCcorinn}
Let the  assumption and the notation be as in   Theorem \ref{finGCstrong}. Then  the  natural map
\begin{equation*}
\Isom(U_1,U_2)\rightarrow \Isom^{(m)}(\Pi_{U_{1}}^{(m-n)},\Pi_{U_{2}}^{(m-n)})/\text{Inn}(\Pi_{U_{2}}^{(m-n)})
\end{equation*}
is  bijective, where $\text{Inn}(\Pi^{(m-n)}_{U_{2}})$ stands  for the group of all inner automorphisms of $\Pi^{(m-n)}_{U_{2}}$ and  the action   $\text{Inn}(\Pi_{U_{2}}^{(m-n)})\curvearrowright \Isom^{(m)}(\Pi_{U_{1}}^{(m-n)},\Pi_{U_{2}}^{(m-n)})$ is  induced by taking the composite.
\end{corollary}
\begin{proof}
Let  $p: \Isom(\tilde{U}^{m-n}_{1}/U_{1},\tilde{U}^{m-n}_{2}/U_{2})\rightarrow \Isom(U_{1},U_{2})$ be the natural map.  Any field  isomorphism $K(U_{1})\xrightarrow{\sim}K(U_{2})$ extends to  $\tilde{\mathcal{K}}(U_{1})\xrightarrow{\sim}\tilde{\mathcal{K}}(U_{2})$ (and preserves $K(U_{i,\overline{k_{i}}})$), and hence it extends to  $\tilde{\mathcal{K}}^{m-n}(U_{1})\xrightarrow{\sim}\tilde{\mathcal{K}}^{m-n}(U_{2})$. Thus,  $p$ is surjective.   Consider  the following commutative diagram

\begin{equation}\label{isomequ}
\vcenter{\xymatrix{
\Isom(\tilde{U}^{m-n}_{1}/U_{1},\tilde{U}^{m-n}_{2}/U_{2})\ar@{->>}[d]^{p}\ar[r]^{\Pi^{(m-n)}(\cdot)} & \Isom^{(m)}(\Pi_{U_{1}}^{(m-n)},\Pi_{U_{2}}^{(m-n)})\ar@{->>}[d]\\
\Isom(U_{1},U_{2}) \ar[r]^-{}& \Isom^{(m)}(\Pi_{U_{1}}^{(m-n)},\Pi_{U_{2}}^{(m-n)})/\text{Inn}(\Pi_{U_{2}}^{(m-n)}).
}}\end{equation}
We have that  $p^{-1}p((\tilde{F},F))=\text{Aut}(\tilde{U}^{m-n}_{2}/U_{2})(\tilde{F},F)$ for $(\tilde{F},F)\in \Isom(\tilde{U}^{m-n}_{1}/U_{1},\tilde{U}^{m-n}_{2}/U_{2})$ (see \cite{Ta1997} LEMMA (4.1)(ii)) and $\Pi^{(m-n)}_{U_{2}}\xleftarrow{\sim}\text{Aut}(\tilde{U}^{m-n}_{2}/U_{2})$.   Thus, Theoreo \ref{finGCstrong} implies that   the lower horizontal arrow of (\ref{isomequ}) is bijective.  
\end{proof}

\begin{remark}[The relative version]\label{finrela}
 Assume that $k=k_{1}=k_{2}$. Then we know that   $\Isom^{(m)}(\Pi_{U_{1}}^{(m-n)},\Pi_{U_{2}}^{(m-n)})=\Isom^{(m)}_{G_{k}}(\Pi_{U_{1}}^{(m-n)},\Pi_{U_{2}}^{(m-n)})$ by Proposition \ref{Gkisomorphic}(1)(2). However,  $\text{Isom}_{k}(U_{1},U_{2})\subsetneq\text{Isom}(U_{1},U_{2})$ holds in general. Hence the natural map $\text{Isom}_{k}(U_{1},U_{2})\rightarrow \Isom^{(m)}_{G_{k}}(\Pi_{U_{1}}^{(m-n)},\Pi_{U_{2}}^{(m-n)})/\text{Inn}(\overline{\Pi}_{U_{2}}^{m-n})$ is not bijective in general. For the case that $k$ is a field finitely generated over the prime field, see Theorem \ref{fingeneGCstrong} below.
\end{remark}

%%%%%%%%%%%%%%%%%%%%%%%%%%%%%%%%%%%%%%%%%%%%%%%%%%%%%%%%%%%%%%%%%%%%%%%%%%%%%%%%%%%%%%%%%%%
%%%%%%%%%%%%%%%%%%%%%%%%%%%%%%%%%%%%%           %%%%%%%%%%%%%%%%%%%%%%%%%%%%%%%%%%%%%%%%%%%%%%%%
%%%%%%%%%%%%%%%%%%%%%%%%%%%%%%%%%%%       %%      %%%%%%%%%%%%%%%%%%%%%%%%%%%%%%%%%%%%%%%%%%%%%%%
%%%%%%%%%%%%%%%%%%%%%%%%%%%%%%%%%%%%%%%%%%    %%%%%%%%%%%%%%%%%%%%%%%%%%%%%%%%%%%%%%%%%%%%%%%
%%%%%%%%%%%%%%%%%%%%%%%%%%%%%%%%%%%%%%%%%   %%%%%%%%%%%%%%%%%%%%%%%%%%%%%%%%%%%%%%%%%%%%%%%%%
%%%%%%%%%%%%%%%%%%%%%%%%%%%%%%%%%%%%%%%%   %%%%%%%%%%%%%%%%%%%%%%%%%%%%%%%%%%%%%%%%%%%%%%%%%%
%%%%%%%%%%%%%%%%%%%%%%%%%%%%%%%%%%%%%%%%%    %%%%%%%%%%%%%%%%%%%%%%%%%%%%%%%%%%%%%%%%%%%%%%%
%%%%%%%%%%%%%%%%%%%%%%%%%%%%%%%%%%%    %%%%    %%%%%%%%%%%%%%%%%%%%%%%%%%%%%%%%%%%%%%%%%%%%%%%
%%%%%%%%%%%%%%%%%%%%%%%%%%%%%%%%%%      %%%      %%%%%%%%%%%%%%%%%%%%%%%%%%%%%%%%%%%%%%%%%%%%%%
%%%%%%%%%%%%%%%%%%%%%%%%%%%%%%%%%%%%            %%%%%%%%%%%%%%%%%%%%%%%%%%%%%%%%%%%%%%%%%%%%%%%%
%%%%%%%%%%%%%%%%%%%%%%%%%%%%%%%%%%%%%%%%%%%%%%%%%%%%%%%%%%%%%%%%%%%%%%%%%%%%%%%%%%%%%%%%%%%%

\section{The  $m$-step solvable version of the good reduction criterion for hyperbolic curves}\label{subsectiongood}

%%%%%%%
\hspace{\parindent}In this section,  we show the $m$-step solvable version of the Oda-Tamagawa good reduction criterion for hyperbolic curves over discrete valuation fields and a corollary for hyperbolic curves over the fields of fractions of  henselian regular local rings. 
\\\ \\
{\bf Notation of section \ref{subsectiongood} }  In this section, we use the following notation in addition to Notation (see Introduction).
\begin{itemize}
\item Let $R$ be a discrete valuation ring, $K:=K(R)$ the field of fractions  of $R$, $s\in \text{Spec}(R)$ the  closed point,  and  $\eta\in \text{Spec}(R)$ the generic point. We write  $\kappa(s)$  for the residue field  at $s$ and  $p$ $(\geq 0)$ for the characteristic of $\kappa(s)$.  
\item Let $(X,E)$ be a  smooth curve of type $(g,r)$  over $K$. Set $U:=X-E$.
\item We write  $I \subset G_{K}$ for an inertia group at $s$ (determined up to $G_{K}$-conjugacy). 
\item Fix a prime $\ell$ different from $ p$. 
\end{itemize}

\begin{definition}\label{defstablecurve}
\begin{enumerate}[(1)]
\item Let $S$ be a scheme, $\mathcal{X}$ a scheme over $S$, $\mathcal{E}$ a (possibly empty) closed subscheme of $\mathcal{X}$, and $(g,r)$ a pair of non-negative integers. We  say that  the pair $(\mathcal{X},\mathcal{E})$ is a $semi$-$stable$ (resp. $stable$) $curve$  ($of$ $type$ $(g,r)$) $over$ $S$ if the following conditions (a)-(d) (resp. (a)-(e)) hold.
\begin{enumerate}[(a)]
\item $\mathcal{X}$ is flat, proper, of finite presentation, and of relative dimension one over $S$.
\item For any  geometric point  $\overline{s}$  of  $S$, the geometric fiber   $\mathcal{X}_{\overline{s}}$ at $\overline{s}$ is reduced, connected with at most ordinary double points as singularities, and   satisfies $\text{dim}(H^{1}(\mathcal{X}_{\overline{s}},\mathcal{O}_{\mathcal{X}_{\overline{s}}}))=g$.
\item The composite of $\mathcal{E}\hookrightarrow \mathcal{X}\rightarrow S$ is finite, \'{e}tale,  and of degree $r$. 
\item  For any  geometric point  $\overline{s}$  of  $S$,   $\mathcal{E}_{\overline{s}}$ is contained in the smooth locus of $\mathcal{X}_{\overline{s}}$, where $\mathcal{X}_{\overline{s}}$ and    $\mathcal{E}_{\overline{s}}$ are the generic fibers of $\mathcal{X}$ and $\mathcal{E}$, respectively, at $\overline{s}$.
\item Assume that $2g+r-2>0$. For any irreducible component $\mathcal{T}$ of   $\mathcal{X}_{\overline{s}}$ which is isomorphic to a projective line,   ``the number of  points where  $\mathcal{T}$  meets other components'' plus  ``the number of points of $\mathcal{E}_{\overline{s}}$ on $\mathcal{T}$'' is at least three.
\end{enumerate}
If there is no risk of confusion, we also call  the complement $\mathcal{U}=\mathcal{X}-\mathcal{E}$ a  semi-stable (resp. stable) curve over $S$ (of type $(g,r)$). 
\item We say that a smooth curve (resp. semi-stable, resp. stable) curve $(\mathfrak{X},\mathfrak{E})$ over $\text{Spec}(R)$ is a $ smooth$  (resp. $semi$-$stable$, resp. $stable$) $model$ $of$ $(X,E)$ over $\text{Spec}(R)$ if the  generic fiber $(\mathfrak{X}_{\eta},\mathfrak{E}_{\eta})$ is isomorphic to $(X,E)$ over $K$.  We say that $(X,E)$ has $good$ (resp.  $semi$-$stable$, resp. $stable$) $reduction$ $at$ $s$ if there exists a smooth (resp. semi-stable, resp. stable) model of $(X,E)$ over $\text{Spec}(R)$. 
\end{enumerate}
\end{definition}

 We have the following theorem.

\begin{theorem}[The Oda-Tamagawa good reduction criterion  for hyperbolic curves, see \cite{Oda1990}, \cite{Oda1995}, \cite{Ta1997}]\label{otgoodthm}
 Assume that $(X,E)$ is hyperbolic.  Then the  following conditions  (a)-(c) are equivalent
\begin{enumerate}[(a)]
\item $(X,E)$ has good reduction at $s$.
\item The image of $I $ in $\Out(\overline{\Pi}^{\pro p'}_{U})$ is trivial.
\item The image of $I $ in $\Out(\overline{\Pi}^{\pro\ell}_{U})$ is trivial.
\end{enumerate}
Here, $\overline{\Pi}^{\pro 0'}_{U}$ is defined as $\overline{\Pi}_{U}$.
\end{theorem}

\begin{remark}
The proof of Theorem \ref{otgoodthm}    essentially only  used  the information of $\overline{\Pi}_{U}^{3,\pro\ell}$ (see \cite{Ta1997}  THEOREM (5.3)). In fact,  when $r<2$ (resp. $r\geq 2$), the  $2$-step (resp. $3$-step) solvable version of Theorem \ref{otgoodthm} follows from  \cite{Ta1997} Remark (5.4) and \cite{AMO1992}. (A proof of this fact and a certain  extension will be presented in a forthcoming joint paper by Ippei Nagamachi and the author.)  However, the proof  in  \cite{Ta1997}  has the following problem:

\begin{itemize}
\item  Tamagawa reduced the proof  to the case where $R$ is strictly henselian and then  to the case where $\kappa(s)$ is perfect (i.e., algebraically closed) by using  the claim  ``When  $R$ is strictly henselian,  $X$ has (semi-)stable reduction at $s$ if and only if $J_{X}$ has semi-stable reduction at $s$''.  This claim  is proved in \cite{DM1969} Theorem (2.4) when $\kappa(s)$ is  algebraically closed, but is not proved when $\kappa(s)$ is  separably closed. 
\end{itemize}
Clearly, one possible way to solve  this problem  is to show the claim. In fact, the claim  is already fully  proved in \cite{Na2022} Theorem 3.15, based on the new theory of minimal log regular models. In this section, instead, we take another more elementary way to solve this problem. More precisely, we prove a certain weaker variant of the claim (Lemma \ref{lemmanag}) by discussing the descent for purely inseparable extensions of $\kappa(s)$ and give a complete proof of the $m$-step solvable version of Theorem \ref{otgoodthm} for arbitrary $m\geq 2$.
\end{remark}

Let us consider the $m$-step solvable version of  Theorem \ref{otgoodthm}. \par

\begin{lemma}\label{11.2}
Assume that $R$ is strictly henselian. Then  the following conditions (i)-(ii) are equivalent. 
\begin{enumerate}[(i)]
\item $X$ has semi-stable reduction at $s$  and $E(K^{\text{sep}})=E(K)$. 
\item  $(X,E)$ has  semi-stable reduction at $s$.
 \end{enumerate}
\end{lemma}

\begin{proof}
The implication  (i)$\Leftarrow$(ii)  follows from the fact that $\pi_{1}(\text{Spec}(R))$ is trivial. We consider the implication (i)$\Rightarrow$(ii). Let $\mathfrak{X}$ be a semi-stable model of $X$.  Let $x$ be an element of $E (\subset \mathfrak{X})$. By  the valuative criterion applied to  the diagram
\[
 \xymatrix{
 \kappa(x)\ar[d]\ar[r] & \mathfrak{X}\ar[d]\\
 \text{Spec}(R)\ar[r]&\text{Spec}(R),
 }
 \]
the closed subscheme  $E$  extends to a closed subscheme $\mathfrak{E}$ of $\mathfrak{X}$. Even if  $\mathfrak{E}$ does not satisfy the conditions (c)(d) in Definition \ref{defstablecurve}(1) (in other words, ($\mathfrak{X},\mathfrak{E}$) is not semi-stable), by taking  blowing-ups of  the semi-stable model $\mathfrak{X}$, we can get a semi-stable model of ($X,E)$. Hence  the implication (i)$\Rightarrow$(ii) follows.
\end{proof}

\begin{lemma}\label{11.3}
 Assume that $(X,E)$ is hyperbolic, that $R$ is strictly  henselian, and that $\kappa(s)$ is  perfect.   Then  the following conditions (a)-(c) are equivalent. 
\begin{enumerate}[(a)]
\item The image of  $I $ in $\Aut(\overline{\Pi}_{X}^{1,\pro\ell})$ is finite and $(X,E)$ has  semi-stable reduction at $s$ with a semi-stable model ($\mathcal{X},\mathcal{E}$) such that   the dual graph of the geometric special fiber $\mathfrak{X}_{\overline{s}}$ is a tree (see  \cite{Liu2002} Definition 10.3.17).
\item The image of  $I $ in $\Aut(\overline{\Pi}_{X}^{1,\pro\ell})$ is finite and  $(X,E)$ has semi-stable reduction at $s$.
\item The image of  $I $ in $\Aut(\overline{\Pi}_{X}^{1,\pro\ell})$ is trivial and the image of $I$ in  $\text{Aut}_{\text{set}}(E(K^{\sep}))$ is trivial.
\end{enumerate}
\end{lemma}

\begin{proof} 
First, we show that  (b)$\Leftrightarrow$(c). Note  that  $E(K^{\sep})=E(K)$ if and only if   the image of $I(=G_{K})$ in  $\text{Aut}_{\text{set}}(E(K^{\sep}))$ is trivial. Hence, by Lemma \ref{11.2},   the following follows.
\begin{equation*}\label{eq11}
\text{$(X,E)$ has semi-stable reduction at $s$} \iff \left\{
\begin{array}{lll}
 \text{The image of $I$ in  $\text{Aut}_{\text{set}}(E(K^{\sep}))$ is trivial, and }\\
 \text{$X$ has  semi-stable reduction}\\
\end{array}\right.
\end{equation*}
\noindent In particular, when $E(K^{\text{sep}})\neq \emptyset $ (i.e., $r>0$),  $E(K)\neq \emptyset $ follows from both (b) and (c) by Lemma \ref{wflemma}. Thus,   $X$ has semi-stable reduction at $s$ if and only if $ J_{X}$ has  semi-stable reduction at $s$. (This equivalence follows from \cite{DM1969} Theorem (2.4) when $g\geq 2$. We need the assumption that $\kappa(s)$ is a perfect field here. When $g=1$, this equivalence follows from \cite{DR1973}. When $g=0$, the equivalence is trivial, since $J_{X}$ is trivial.) 
Note that    $ J_{X}$ has semi-stable reduction at $s$ if and only if  the image of  $I $ in $\Aut(\overline{\Pi}_{X}^{1,\pro\ell})$ is unipotent (\cite{SGA7-1} Expo\'{s}e IX Proposition 3.5), and that  a subgroup of $\text{Aut}(\overline{\Pi}^{1,\text{pro-}\ell}_{X})$ is finite and unipotent if and only if  it is  trivial. Hence, by Lemma \ref{wflemma}, (b)$\Leftrightarrow$(c) follows.
 (a)$\Rightarrow$(b) is clear.  Finally, we show that  (b)$\Rightarrow$(a).  If (b) holds, then   $J_{X}$ has good reduction at $s$  by (b)$\Rightarrow(\text{c})$  and  the N\'{e}ron\text{-}Ogg\text{-}Shafarevich good reduction criterion (see \cite{ST1968}). Thus,   the dual graph of  $\mathfrak{X}_{s}$ is a tree by \cite{Liu2002} Remark 10.3.18.  
%Finally, we  show that  (b)$\Rightarrow$(a) in general.  Let $R_{1}$ be  a henselian discrete valuation ring in Lemma \ref{11.2}(2), and  $s_{1}\in\text{Spec}(R_{1})$ the closed point.  Then   the dual graph of  $(\mathfrak{X}_{R_{1}})_{s_{1}}=\mathfrak{X}_{\overline{s}}$ is a tree as   $\kappa(s)$ is   perfect. Because  $\mathfrak{X}_{\overline{s}}\rightarrow \mathfrak{X}_{s}$ does not change the dual graph,   (b)$\Rightarrow$(a) follows.
\end{proof}

The author's original proof  of the $m$-step solvable version of the Oda-Tamagawa good reduction criterion (Theorem \ref{2goodreductiontheorem}(c)$\Rightarrow$(a) below) required the extra assumption that $\kappa(s)$ is perfect. However, by using the following lemma given by  Ippei Nagamachi,  we can also  prove  this $m$-step solvable version  when $\kappa(s)$ is not necessarily perfect. 

\begin{lemma}\label{lemmanag}
Assume that $g\geq 2$.  If $J_{X}$ has good reduction at $s$,  then $X$ has stable reduction at $s$. If, moreover, $X$ has potentially good reduction at $s$, then  $X$ has good reduction at $s$.
\end{lemma}
\begin{proof}
By \cite{Liu2002} Chapter 10 Theorem 4.3, $X$ has potentially stable reduction at $s$. Hence we assume that $X_{K'}$ has stable reduction at $s'\in\text{Spec}(R')^{\text{cl}}$, where  $K'$ is a finite extension of $K$ and  $R'$  is a localization of   the  integral closure $\tilde{R}$ of $R$ in $K'$ at a  maximal ideal of $\tilde{R}$.  Let  $S':=\text{Spec}(R')$, $S'':=S'\times_{S} S'$, $S''':=S'\times_{S}S'\times_{S} S'$,  $\eta':=\text{Spec}(K')$, $\eta'':=\eta'\times_{\eta}\eta'$, and $\eta''':=\eta'\times_{\eta}\eta'\times_{\eta}\eta'$.  Let $t: S'\rightarrow S$ be the natural  morphism, $\text{pr}_{1},\text{pr}_{2}: S''\rightarrow S'$ the first and second  projections, respectively, and $q:=t\circ\text{pr}_{1}=t\circ\text{pr}_{2}$.  Let   $\mathfrak{X}'$ be a stable model of $X_{K'}$ over  $S'$. Then we have the following natural commutative diagram.
\[
\xymatrix{
\text{pr}_{1}^{*}\mathfrak{X}',\text{pr}_{2}^{*}\mathfrak{X}'\ar@<0.5ex>[r] \ar@<-0.5ex>[r]\ar[d]&\mathfrak{X}'\ar[d]&\\
 S''\ar@<0.5ex>[r]^{\text{pr}_{1}} \ar@<-0.5ex>[r]_{\text{pr}_{2}}&S'\ar[r]^{t}&S
&
}
\]
Let $\underline{\text{Isom}}(X_{\eta''},X_{\eta''})\rightarrow \eta''$ (resp. $\underline{\text{Isom}}(\text{pr}_{1}^{*}\mathfrak{X}',\text{pr}_{2}^{*}\mathfrak{X}')\rightarrow S''$) be the Isom-scheme of proper, smooth curves $X_{\eta''}$, $X_{\eta''}$ over $\eta''$  (resp. stable curves $\text{pr}_{1}^{*}\mathfrak{X}',\text{pr}_{2}^{*}\mathfrak{X}'$ over $S''$). Let   $\phi_{\eta''}:\eta''\rightarrow \underline{\text{Isom}}(X_{\eta''},X_{\eta''})$ be the morphism induced by the identity morphism of $X_{\eta''}$.  Let $\mathfrak{J}$ be the N\'{e}ron model of $J_{X}$ over $S$, which is an abelian scheme over $S$. (Note that we use the assumption ``$J_{X}$ has good reduction at $s$'' here.)  Let $\underline{\text{Isom}}(J_{X_{\eta''}},J_{X_{\eta''}})\rightarrow\eta''$ (resp. $\underline{\text{Isom}}(q^{*}\mathfrak{J},q^{*}\mathfrak{J})\rightarrow S''$) be the Isom-scheme of abelian schemes $J_{X_{\eta''}},J_{X_{\eta''}}$ over $\eta''$  (resp. $q^{*}\mathfrak{J}$, $q^{*}\mathfrak{J}$ over $S''$ ). (For the existence and properties of  Isom-schemes, see  \cite{Fa2005} Chapter 5, especially  Theorem 5.23.)   We have that $\underline{\text{Pic}}^{0}(\mathfrak{X}'/S')=t^{*}\mathfrak{J}$ and  $\underline{\text{Pic}}^{0}(\text{pr}_{i}^{*}\mathfrak{X}'/S'')=q^{*}\mathfrak{J}$  for $i=1,2$ (\cite{Bo1990} Chapter 9.5 Theorem 1). Thus, we get the following commutative diagram.
\[
\xymatrix{
\eta''\ar[d]\ar[r]^-{\phi_{\eta''}} & \underline{\text{Isom}}(X_{\eta''},X_{\eta''})\ar[r] \ar[d]& \underline{\text{Isom}}(J_{X_{\eta''}},J_{X_{\eta''}})\ar[d]\ar[r]&\eta''\ar[d]\\
S''& \underline{\text{Isom}}(\text{pr}_{1}^{*}\mathfrak{X}',\text{pr}_{2}^{*}\mathfrak{X}') \ar[r]^{\underline{\text{Pic}}^{0}}& \underline{\text{Isom}}(q^{*}\mathfrak{J},q^{*}\mathfrak{J})\ar[r]&S''
}
\]
Let $S''\rightarrow\underline{\text{Isom}}(q^{*}\mathfrak{J},q^{*}\mathfrak{J}) $ be the morphism  induced by the identity morphism of  $q^{*}\mathfrak{J}$. Set 
\[
T:=S''\times_{ \underline{\text{Isom}}(q^{*}\mathfrak{J},q^{*}\mathfrak{J})}\underline{\text{Isom}}(\text{pr}_{1}^{*}\mathfrak{X}',\text{pr}_{2}^{*}\mathfrak{X}'). 
\]
By \cite{DM1969} Theorem (1.11), $\underline{\text{Isom}}(\text{pr}_{1}^{*}\mathfrak{X}',\text{pr}_{2}^{*}\mathfrak{X}') $ is finite and unramified over $S''$. Hence we get that $T$ is finite and  unramified over $S''$, since  $\underline{\text{Isom}}(q^{*}\mathfrak{J},q^{*}\mathfrak{J})\rightarrow S''$ is separated (see \cite{Fa2005} Chapter 5). Moreover, for any geometric point $\overline{x}$ of $S''$, the map $\text{Isom}_{\kappa(\overline{x})}((\text{pr}_{1}\mathfrak{X}')_{\overline{x}},(\text{pr}_{2}\mathfrak{X}')_{\overline{x}})\rightarrow\text{Aut}_{\kappa(\overline{x})}(q^{*}\mathfrak{J}_{\overline{x}})$ is injective by \cite{DM1969} Theorem (1.13). In particular, we get that the morphism $\underline{\text{Pic}}^{0}$ is radicial  by \cite{Fu2015} Proposition 1.7.1.  Thus,  $T$ is also radicial over $S''$. Since   $T\rightarrow S''$ is finite, unramified, and radicial,   $T\rightarrow S''$ is a closed immersion.  Since $S''\rightarrow S$ is flat and $\eta$ is scheme-theoretically dense in $S$, $\eta''$ is also scheme-theoretically dense in $S''$.  Since $\eta''\rightarrow S''$ factors through $T\hookrightarrow S''$, we have that $T=S''$ and  $\phi_{\eta''}$ uniquely extends to a morphism $\phi_{S''}:S''= T\hookrightarrow \underline{\text{Isom}}(\text{pr}_{1}^{*}\mathfrak{X}',\text{pr}_{2}^{*}\mathfrak{X}')$. Further, $\phi_{S''}$ satisfies the cocycle condition because $\phi_{\eta''}$ satisfies the cocycle condition and $\eta'''\rightarrow S'''$ is scheme-theoretically dense.   By descent theory, $\mathfrak{X}'$ descends to a  (an automatically stable) model of $X$ over $S$.  If $X$ has potentially good reduction at $s$, then this stable model of $X$ over $S$ must be smooth. Therefore, the assertion follows. 
\end{proof}

When $2g+r-2>0$, we write  $\mathcal{M}_{g,r}$ for  the moduli stack  of proper, smooth curves of genus $g$ with $r$ disjoint ordered sections (cf. \cite{Kn1983}). The moduli stack  $\mathcal{M}_{g,r}$ is a  Deligne-Mumford stack  separated over $\text{Spec}(\mathbb{Z})$ by \cite{Kn1983}. (In \cite{DM1969} Definition (4.6), Deligne and Mumford defined ``algebraic stack''. In this paper, we call  ``algebraic stack'' by Deligne and Mumford  ``Deligne-Mumford stack''.)  
We have that the symmetric group $S_{r}$ acts on  $\mathcal{M}_{g,r}$ via the  permutation of  the ordered  sections.  By \cite{MR2005} Theorem 4.1 and Theorem 5.1, there exists a Deligne-Mumford stack    $\mathcal{M}_{g,[r]}:=\mathcal{M}_{g,r}/S_{r}$ separated over $\text{Spec}(\mathbb{Z})$, which turns out to be the moduli stack  of smooth curves of type $(g,r)$. We write $R^{\text{sh}}$ for  a strict henselization of $R$.

\begin{lemma}\label{goodtamagawanag}
Define $\epsilon$ as $0$ (resp. $1$, resp. $3$) when $g\geq 2$ (resp. $g=1$, resp. $g=0$).  Let  $s^{\text{sh}}$ be the closed point of $\text{Spec}(R^{\text{sh}})$.  Let $W$ be a subscheme of $E_{K(R^{\text{sh}})}$ satisfying the  degree of $W$ over $K(R^{\text{sh}})$ is greater than or equal to $\epsilon$.   Assume that $(X_{K(R^{\text{sh}})},E_{K(R^{\text{sh}})})$  has potentially good reduction at $s^{\text{sh}}$ and  $(X_{R^{\text{sh}}},W)$ has good reduction at $s^{\text{sh}}$.  Then $(X,E)$ has a good reduction at $s$.
\end{lemma}

\begin{proof}
Note that $W\subset E(K(R^{\text{sh}}))$. By  assumption, there exists a discrete valuation ring  $R'$ such that $R'$ is  etale over $R$, that  $[K(R'):K(R)]<\infty$, that $(X_{K(R')},E_{K(R')})$  has potentially good reduction at $s'$,  that $W'\subset E(K(R'))$ (hence  $W'_{K(R^{\text{sh}})}=W$), and  that  $(X_{K(R')},W')$ has good reduction at $s'$, where $s'$ stands for the closed point of $\text{Spec}(R')$ and $W'$ is  the image of $W$ in $E_{K(R')}$. We set $S:=\text{Spec}(R)$ and $S':=\text{Spec}(R')$.  Let  $\eta'=\text{Spec}(K(R'))$ be   the generic point of $\text{Spec}(R')$.  Let $(\mathfrak{X}',\mathfrak{W}')$ be a smooth model of $(X_{\eta'},W')$ over $S'$ and $\mathfrak{E}'$ the scheme-theoretic closure of $E_{\eta'}$ in $\mathfrak{X}'$.    Since  $(X_{\eta'},E_{\eta'})$ has  potentially good reduction at $s'$,  there exist an extension $T\rightarrow S'$ of spectra of discrete valuation rings with $[K(T):K(S')]<\infty$  and  a smooth  model   $(\mathcal{X}',\mathcal{E}')$ of $(X_{K(T)},E_{K(T)})$ over $T$. Let  $\mathcal{W}'$ be the scheme-theoretic closure of $W'_{T}$ in $\mathcal{X}'$.   The separatedness of $\mathcal{M}_{g,[\epsilon]}$  implies that a smooth model of $(X_{K(T)},W'_{K(T)})$ over $T$ is unique. Hence we have an isomorphism $(\mathcal{X}',\mathcal{W}')\xrightarrow{\sim}(\mathfrak{X}'_{T}, \mathfrak{W}'_{T})$  over $T$, which  induces an isomorphism  $\mathcal{E}'\xrightarrow{\sim}\mathfrak{E}'_{T}$. In particular, $\mathfrak{E}'$ is finite  \'{e}tale over $S'$. Thus, $(\mathfrak{X}',\mathfrak{E}')$ is a smooth model of $(X_{\eta'},E_{\eta'})$ over $S'$. \par
We set  $S'':=S'\times_{S}S'$, $S''':=S'\times_{S}S'\times_{S}S'$, $\eta'':=\eta'\times_{\eta}\eta'$, and $\eta''':=\eta'\times_{\eta} \eta'\times_{\eta}\eta'$. Let  $\text{pr}_{1}, \text{pr}_{2}: S''\rightarrow S'$ be the first and second projections, respectively.   Let $\underline{\text{Isom}}_{\eta''}((X_{\eta''},E_{\eta''}),(X_{\eta''},E_{\eta''}))\rightarrow \eta''$ (resp. $\underline{\text{Isom}}_{S''}(\text{pr}_{1}^{*}(\mathfrak{X}',\mathfrak{E}'), \text{pr}_{2}^{*}(\mathfrak{X}',\mathfrak{E}'))\rightarrow S''$) be the Isom-scheme of   smooth curves $(X_{\eta''},E_{\eta''})$, $(X_{\eta''},E_{\eta''})$ over $\eta''$  (resp.  $\text{pr}_{1}^{*}(\mathfrak{X}',\mathfrak{E}'),\text{pr}_{2}^{*}(\mathfrak{X}',\mathfrak{E}')$ over $S''$).  Then we have the following  diagram.
\begin{equation*}
\xymatrix{
\eta''\ar[d]\ar[r]^-{\phi_{\eta''}} &\underline{\text{Isom}}_{\eta''}((X_{\eta''},E_{\eta''}),(X_{\eta''},E_{\eta''}))\ar[d]\ar[r]&\eta''\ar[d]\\
S''& \underline{\text{Isom}}_{S''}(\text{pr}_{1}^{*}(\mathfrak{X}',\mathfrak{E}'), \text{pr}_{2}^{*}(\mathfrak{X}',\mathfrak{E}'))\ar[r]&S'',
}
\end{equation*}
where $\phi_{\eta''}$ is the morphism induced by the identity morphism of $X_{\eta''}$. Since $  \underline{\text{Isom}}_{S''}(\text{pr}_{1}^{*}(\mathfrak{X}',\mathfrak{E}'), \text{pr}_{2}^{*}(\mathfrak{X}',\mathfrak{E}'))\rightarrow S'',$ is finite and   $S''$ is nomal, the morphism $\eta''\rightarrow\underline{\text{Isom}}_{\eta''}((X_{\eta''},E_{\eta''}),(X_{\eta''},E_{\eta''}))$ extends to the morphism  $\phi_{S''}: S''\rightarrow\underline{\text{Isom}}_{S''}(\text{pr}_{1}^{*}(\mathfrak{X}',\mathfrak{E}'), \text{pr}_{2}^{*}(\mathfrak{X}',\mathfrak{E}'))$.  Further, $\phi_{S''}$ satisfies the cocycle condition because $\phi_{\eta''}$ satisfies the cocycle condition and $\eta'''\rightarrow S'''$ is scheme-theoretically dense.   By descent theory, $(\mathfrak{X}',\mathfrak{E}')$ descends to a  (an automatically smooth) model of $(X,E)$ over $S$. Thus, the assertion follows.
\end{proof}

\begin{theorem}\label{2goodreductiontheorem}
Assume that $(X,E)$ is hyperbolic, and that  $m\geq2$. Then the following conditions (a)-(c) are equivalent. 
\begin{enumerate}[(a)]
\item\label{gooda} $(X,E)$ has good reduction at $s$.
\item\label{goodc} The image of $I $ in $\Out(\overline{\Pi}_{U}^{m,\pro p'})$ is trivial.
\item\label{goodb} The image of $I $ in $\Out(\overline{\Pi}_{U}^{m,\pro\ell})$ is trivial.
\end{enumerate}
Here, $\overline{\Pi}^{m,\pro 0'}_{U}$ is defined as $\overline{\Pi}^{m}_{U}$.
\end{theorem}

\begin{proof}
The implication (\ref{gooda})$\Rightarrow$(\ref{goodc}) follows from \cite{SGA1} Expo\'{s}e XIII.  The implication  (\ref{goodc})$\Rightarrow$(\ref{goodb}) is clear.  First, we show that  $(\ref{goodb})\Rightarrow (\ref{gooda})$ under the assumption  that $R$ is strictly henselian  and $\kappa(s)$ is perfect.  We have that (c) implies the condition (c) in Lemma \ref{11.3} by Lemma \ref{wflemma}. Hence, by Lemma \ref{11.3}(c)$\Rightarrow$(a), we get  that $(X,E)$ has a semi-stable model $(\mathfrak{X},\mathfrak{E})$ and the   dual graph of the closed fiber of $(\mathfrak{X}, \mathfrak{E})$ is a  tree. Set $\mathfrak{U}:=\mathfrak{X}-\mathfrak{E}$.   We consider the specialization homomorphism $\Pi_{U}^{(m,\pro\ell)}\twoheadrightarrow (\Pi_{\mathfrak{U}}^{(m,\text{pro-}\ell)}\xleftarrow{\sim})\Pi_{\mathfrak{U}_{s}}^{(m,\pro\ell)}$ (see \cite{SGA1} Expo\'{s}e X Corollary 2.3). Let $H$ be an open normal subgroup of $\Pi_{U}^{(m,\pro\ell)}$  that  satisfies: (i)  $(\overline{\Pi}_{U}^{\pro\ell})^{[m-1]}/(\overline{\Pi}_{U}^{\pro\ell})^{[m]}\subset H$, and  (ii) $\text{Ker}( \Pi_{U}^{(m,\pro\ell)}\twoheadrightarrow \Pi_{\mathfrak{U}_{s}}^{(m,\pro\ell)})\subset H$. (Note that,  by (ii), $H$ also satisfies:   (iii) the composite of  $H\hookrightarrow  \Pi_{U}^{(m,\pro\ell)}\twoheadrightarrow I$ is surjective, since $R$ is strictly henselian.) Let   $(X_{H},E_{H})$ be the covering of $(X,E)$ corresponding to $H$, $\mathfrak{U}_{H}$  the covering of $\mathfrak{U}$  corresponding to the image of $H$  in $\Pi^{(m,\text{pro-}\ell)}_{\mathfrak{U}}$, $\mathfrak{X}_{H}$ the nomalization of $\mathfrak{X}$ in the function field of $\mathfrak{U}_{H}$, and $\mathfrak{E}_{H}:=\mathfrak{X}_{H}-\mathfrak{U}_{H}$.  By Abhyankar's lemma, (ii) implies that $(\mathfrak{X}_{H}, \mathfrak{E}_{H})$ is a semi-stable model of  $(X_{H},E_{H})$. First, we claim that the   dual graph of the closed fiber of $(\mathfrak{X}_{H}, \mathfrak{E}_{H})$ is a  tree. Indeed,  we have the following diagram (see (\ref{outeraction})).
\begin{equation}\vcenter{\xymatrix{
1\ar[r] &\overline{\Pi}_{U}^{m,\pro\ell}\ar[r]\ar[d] & \Pi_{U}^{(m,\pro\ell)}\ar[r]\ar[d]& I\ar[d]^{\text{trivial}}\ar[r] &1\\
1\ar[r] & \text{Inn}(\overline{\Pi}_{U}^{m,\pro\ell})\ar[r]&\Aut(\overline{\Pi}_{U}^{m,\pro\ell})\ar[r] &\Out(\overline{\Pi}_{U}^{m,\pro\ell})\ar[r]&1 
}}
\end{equation}
Set $J:=\text{ker}(\Pi_{U}^{(m,\pro\ell)}\rightarrow\Aut(\overline{\Pi}_{U}^{m,\pro\ell}))$.  By (\ref{goodb}), the natural map  $p: J\twoheadrightarrow I$ is surjective.   By definition of $J$, the map  $H\cap J\rightarrow \text{Aut}(\overline{H})\rightarrow \text{Aut}(\overline{H}^{1})$  is trivial.  We obtain that $p(H\cap J)\overset{\text{op}}\subset I$ as $H\cap J\overset{\text{op}}\subset J$. In particular, the image of $I$ in  $\text{Aut}(\overline{H}^{1})$ is finite, where $I\rightarrow \text{Aut}(\overline{H}^{1})$ is induced by (iii). The condition  (i) implies that $\overline{H}^{1}\xleftarrow{\sim}\overline{\Pi}_{U_{H}}^{1}$.  Hence $(X_{H},E_{H})$  satisfies the condition (b) of Lemma \ref{11.3}. Therefore,  by Lemma \ref{11.3} (b) $\Rightarrow$ (a), the claim follows.  Next, we construct an open normal subgroup of $\Pi_{U}^{(m,\pro\ell)}$  that  satisfies (i)-(iii). Let $\{Z_{i}\}_{i=1,\cdots,j}$ be the set of irreducible components of $\mathfrak{X}_{s}$, and set $W_{i}:=Z_{i}-\mathfrak{E}$. Then $W_{i}$ is smooth, and  
\[
\Pi_{\mathfrak{U}_{s}}^{\text{ab},\pro\ell}\cong \prod_{i=1}^{j} \Pi_{W_{i}}^{\text{ab},\pro\ell},
\]
since the dual graph is tree. We can construct  a quotient of $\Pi_{U}^{(m,\pro\ell)}$ which factors through $\Pi_{\mathfrak{U}_{s}}^{\text{ab},\pro\ell}$ and is isomorphic to  $\mathbb{Z}/\ell\mathbb{Z}$ such that $\Pi_{W_{i}}^{\text{ab},\pro\ell}$ is surjectively mapped onto the quotient for each $i=1,\cdots, j$ with $\Pi_{W_{i}}^{\text{ab},\pro\ell}\neq \{1\}$.  We define $H' \subset \Pi_{U}^{(m,\pro\ell)}$ as the kernel of the surjection $\Pi_{U}^{(m,\pro\ell)}\twoheadrightarrow \mathbb{Z}/\ell\mathbb{Z}$.  $H'$ satisfies the above  conditions (i)-(iii) by the construction. Hence  $(\mathfrak{X}_{H'}, \mathfrak{E}_{H'})$ is a semi-stable model of  $(X_{H'},E_{H'})$ and the   dual graph of the closed fiber of $(\mathfrak{X}_{H'}, \mathfrak{E}_{H'})$ is a  tree. Since the  dual graphs of the closed fibers of the semi-stable models of   $(X,E)$ and  $(X_{H'},E_{H'})$ are trees,  $(X,E)$ has a good reduction at $s$ by  the last paragraph of the  proof of \cite{Ta1997} Theorem (5.3) (d)$\Rightarrow$(a).\par 
Finally, we show that $(\ref{goodb})\Rightarrow (\ref{gooda})$  in general.   By \cite{Liu2002} Lemma 10.3.32, there exists a  henselian discrete valuation ring $R_{1}$ containing $R$ such that    a  uniformizer of  $R$ is a uniformizer of $R_{1}$ and   the residue field of $R_{1}$ is $\overline{\kappa(s)}$.  By the discussion so far and this, we may assume that $(X,E)$ has potentially good reduction at $s$.  When $g\geq 2$, $X$ has good reduction at $s$ by  (c), the N\'{e}ron\text{-}Ogg\text{-}Shafarevich criterion, and Lemma \ref{lemmanag}. Thus, when $g\geq 2$,  $(X,E)$ has good reduction at $s$ by Lemma \ref{goodtamagawanag}. Next, we consider the case that  $g\leq 1$.  We define  $\epsilon$ as $3$ (resp. $1$) when $g=0$ (resp. $g=1$).  The hyperbolicity of $U$ and (c) implies that $|E(K^{\text{sep}})|=|E(K^{\text{sh}})|\geq \epsilon$. When $g=0$,  $(X_{K^{\text{sh}}},P_{1},P_{2},P_{2})$ has good reduction at $s'$ for any $P_{1}$, $P_{2}$, $P_{3}\in E_{K^{\text{sh}}}(K^{\text{sh}})$, since $X_{K^{\text{sh}}}$ is isomorphic to $\mathbb{P}_{K^{\text{sh}}}^{1}$.  When $g=1$,  by (c) and the N\'{e}ron\text{-}Ogg\text{-}Shafarevich criterion,  $(X_{K^{\text{sh}}},P)$ has good reduction at $s'$ for any $P\in E_{K^{\text{sh}}}(K^{\text{sh}})$. Thus, when $g\leq 1$,  $(X,E)$ has good reduction at $s$ by Lemma \ref{goodtamagawanag}. Therefore, the assertion follows.
\end{proof}

\begin{lemma}\label{lemma20.3} Let $R^{\dag}$ be a regular local ring, and  $(\mathfrak{X}^{\dag},\mathfrak{E}^{\dag})$ a smooth curve over $\text{Spec}(R^{\dag})$. Set $\mathfrak{U}^{\dag}:=\mathfrak{X}^{\dag}-\mathfrak{E}^{\dag}$.  Let $\rho: \mathfrak{U}^{\dag}\rightarrow \text{Spec}(R^{\dag})$ be the structure morphism, and  $v\in\mathfrak{U}^{\dag}$. Then $v$ is of codimension one in $\mathfrak{U}^{\dag}$ if and only if $v$ satisfies  one of the following conditions (i)-(ii).
\begin{enumerate}[(i)]
\item $\rho(v)$ is of  codimension one in $\text{Spec}(R^{\dag})$ and $v$ is the generic point of $\mathfrak{U}^{\dag}_{\rho(v)}$.
\item $\rho(v)$ is the generic point of $\text{Spec}(R^{\dag})$ and $\kappa(v)/\kappa(\rho(v))$ is finite. 
\end{enumerate}
\end{lemma}
\begin{proof}
Since $\rho$ is flat, the going-down theorem holds for $\text{Spec}(O_{\mathfrak{U}^{\dag},v})\rightarrow \text{Spec}(O_{\text{Spec}(R^{\dag}),\rho(v)})$. In particular, we get $\text{codim}(\rho(v))\leq \text{codim}(v)$.   Hence $\text{codim}(\rho(v))=0$ or $1$. Since $R^{\dag}$ is a regular local ring, $R^{\dag}$ is universally catenary. Thus, we get the  dimension formula.
\[
\text{codim}(v)=\text{codim}(\rho(v))+\text{tr.deg}_{K(\text{Spec}(R^{\dag}))}(K(\mathfrak{U}^{\dag}))-\text{tr.deg}_{\kappa(\rho(v))}(\kappa(v))
\]
We have $\text{tr.deg}_{K(\text{Spec}(R^{\dag}))}(K(\mathfrak{U}^{\dag}))=1$.   When $\text{codim}(\rho(v))=0$ (resp.  $\text{codim}(\rho(v))=1$), we get $\text{tr.deg}_{\kappa(\rho(v))}(\kappa(v))=0$ (resp. $\text{tr.deg}_{\kappa(\rho(v))}(\kappa(v))=1$) by the  dimension formula. Thus, the assertion follows.
\end{proof}

\begin{corollary}\label{goodcor}
Assume that $m\geq 3$. Let $n\in\mathbb{Z}_{\geq 2}$ be an integer satisfying $m>n$. Let $R^{\dag}$ be a henselian regular local ring,  $K^{\dag}:=K(R^{\dag})$,  $s^{\dag}\in \text{Spec}(R^{\dag})$ the  closed point,    $\eta^{\dag}\in \text{Spec}(R^{\dag})$ the generic point, and  $p^{\dag}$ $(\geq 0)$  the characteristic of $\kappa(s^{\dag})$.   Let $(X^{\dag},E^{\dag})$ be a  hyperbolic curve of type $(g^{\dag},r^{\dag})$  over $K^{\dag}$. Set $U^{\dag}:=X^{\dag}-E^{\dag}$.   Let $(\mathfrak{X}^{\dag},\mathfrak{E}^{\dag})$ be a smooth curve over $\text{Spec}(R^{\dag})$  such that the  generic fiber $(\mathfrak{X}_{\eta^{\dag}}^{\dag},\mathfrak{E}^{\dag}_{\eta^{\dag}})$ is isomorphic to $(X^{\dag},E^{\dag})$ over $K^{\dag}$. Set  $\mathfrak{U}^{\dag}:=\mathfrak{X}^{\dag}-\mathfrak{E}^{\dag}$.   Let $H$ be an open normal subgroup of $\Pi_{U^{\dag}}^{(m)}$ containing $\overline{\Pi}_{U^{\dag}}^{[m-n]}/\overline{\Pi}_{U^{\dag}}^{[m]}$. Let $I^{\dag}\subset G_{K^{\dag}}$ be the inertia group at $s^{\dag}$ and $\ell^{\dag}$ a prime different from $p^{\dag}$.   Then the following conditions (a)-(c) are equivalent.
\begin{enumerate}[(a)]
\item $H$ contains the kernel of the specialization homomorphism $\Pi_{U^{\dag}}^{(m)}\twoheadrightarrow \Pi_{\mathfrak{U}^{\dag}_{s^{\dag}}}^{(m)}$.
\item 
\begin{enumerate}[(i)]
\item The image of $H$ in $G_{K^{\dag}}$ contains $I^{\dag} $.
\item The image of $I^{\dag} $ in $\Out(\overline{H}^{n,\pro (p^{\dag})'})$ is trivial.
\end{enumerate}
\item
\begin{enumerate}[(i)]
\item The image of $H$ in $G_{K^{\dag}}$ contains $I^{\dag} $.
\item The image of $I^{\dag} $ in $\Out(\overline{H}^{n,\pro \ell^{\dag}})$ is trivial.
\end{enumerate}
\end{enumerate}
Here, $\overline{H}^{n,\pro 0'}$ is defined as $\overline{H}^{n}$. In particular, we obtain that $\Pi_{\mathfrak{U}^{\dag}_{s}}^{(m-n)}=\plim{H}\Pi_{U^{\dag}}^{(m)}/H$, where $H$ runs over  all open normal subgroups of $\Pi_{U^{\dag}}^{(m)}$  satisfying $\overline{\Pi}_{U^{\dag}}^{[m-n]}/\overline{\Pi}_{U^{\dag}}^{[m]}\subset H$ and (b) (or equivalently (c)).
\end{corollary}
\begin{proof}
First, we show the assertion when $\text{dim}(R^{\dag})=1$.  Let $R^{\dag\text{sh}}$ be the strictly henselization of $R^{\dag}$. Since $\Pi^{(m)}_{U^{\dag}_{K(R^{\dag\text{sh}})}}$ coincides with the inverse image of $I^{\dag}$ by $\Pi^{(m)}_{U^{\dag}}\rightarrow G_{K^{\dag}}$, we get $\Pi^{(m)}_{U^{\dag}_{K(R^{\dag\text{sh}})}}\twoheadrightarrow \Pi^{(m)}_{\mathfrak{U}^{\dag}_{\overline{s^{\dag}}}}(= \overline{\Pi}^{m}_{\mathfrak{U}^{\dag}_{s^{\dag}}})$ and  $\text{Ker}(\Pi_{U^{\dag}}^{(m)}\twoheadrightarrow \Pi_{\mathfrak{U}^{\dag}_{s^{\dag}}}^{(m)})=\text{Ker}(\Pi^{(m)}_{U^{\dag}_{K(R^{\dag\text{sh}})}}\twoheadrightarrow \Pi^{(m)}_{\mathfrak{U}^{\dag}_{\overline{s}^{\dag}}})$.  Thus, we may assume that $R^{\dag}$ is strictly henselian. By \cite{Ta1997} Lemma (5.5), we have that
\[
(\text{a})\Leftrightarrow \text{``The coefficient field of $(X^{\dag}_{H},E^{\dag}_{H})$ is $K^{\dag}$ and $(X^{\dag}_{H},E^{\dag}_{H})$ has good reduction at $s^{\dag}$"}
\] 
 Thus, (a)$\Leftrightarrow$(b)$\Leftrightarrow$(c)  follows from Theorem \ref{2goodreductiontheorem}. \par
Next, we consider the general case. Let $\rho: \mathfrak{U}^{\dag}\rightarrow \text{Spec}(R^{\dag})$ be the structure morphism.  By  the purity of Zariski-Nagata (\cite{SGA1} Expo\'{s}e X num\'{e}ro 3), the condition (a) holds if and only if  $H$ contains the kernel of the specialization homomorphism $\Pi_{U^{\dag}_{K^{\dag}_{\rho(v)}}}^{(m)}\twoheadrightarrow \Pi_{\mathfrak{U}^{\dag}_{\kappa(\rho(v))}}^{(m)}$ for any $v\in\mathfrak{U}^{\dag}$ satisfying   (i) in Lemma \ref{lemma20.3}, where $K^{\dag}_{\rho(v)}$ stands for the field of fractions of the completion of the localization of  $R^{\dag}$ at $\rho(v)$. Moreover, by  the purity of Zariski-Nagata,   the condition (b) (resp. (c))  holds if and only if the image of $H\cap \Pi_{U^{\dag}_{K^{\dag}_{\rho(v)}}}^{(m)}$ in $G_{K^{\dag}_{\rho(v)}}$ contains $I_{\rho(v), G_{K^{\dag}_{\rho(v)}}}$ and the image of $I_{\rho(v), G_{K^{\dag}_{\rho(v)}}}$ in $\text{Out}(\overline{H}^{n\text{,pro-}(p^{\dag})'})$ (resp. $\text{Out}(\overline{H}^{n\text{,pro-}\ell^{\dag}})$) is trivial  for any $v\in\mathfrak{U}^{\dag}$ satisfying   (i) in Lemma \ref{lemma20.3}.  
Hence, by the case that $\text{dim}(R^{\dag})=1$, (a)$\Leftrightarrow$ (b) (resp.  (a)$\Leftrightarrow$ (c)) follows.  The second assertion follows from the first assertion.
\end{proof}

%%%%%%%%%%%%%%%%%%%%%%%%%%%%%%%%%%%%%%%%%%%%%%%%%%%%%%%%%%%%%%%%%%%%%%%%%%%%%%%%%%%%%%%%%
%%%%%%%%%%%%%%%%%%%%%%%%%%%%%%%%%%%%%     %%%%%%%%%%%%%%%%%%%%%%%%%%%%%%%%%%%%%%%%%%%%%%%%
%%%%%%%%%%%%%%%%%%%%%%%%%%%%%%%%%%%   %%  %%%%%%%%%%%%%%%%%%%%%%%%%%%%%%%%%%%%%%%%%%%%%%%%%
%%%%%%%%%%%%%%%%%%%%%%%%%%%%%%%%%     %%%  %%%%%%%%%%%%%%%%%%%%%%%%%%%%%%%%%%%%%%%%%%%%%%%%%%
%%%%%%%%%%%%%%%%%%%%%%%%%%%%%%%%     %%%%  %%%%%%%%%%%%%%%%%%%%%%%%%%%%%%%%%%%%%%%%%%%%%%%%%%
%%%%%%%%%%%%%%%%%%%%%%%%%%%%%%%     %%%%%  %%%%%%%%%%%%%%%%%%%%%%%%%%%%%%%%%%%%%%%%%%%%%%%%%%
%%%%%%%%%%%%%%%%%%%%%%%%%%%%%%%                        %%%%%%%%%%%%%%%%%%%%%%%%%%%%%%%%%%%%%%%%%%%%%
%%%%%%%%%%%%%%%%%%%%%%%%%%%%%%%%%%%%%%%   %%%%%%%%%%%%%%%%%%%%%%%%%%%%%%%%%%%%%%%%%%%%%%
%%%%%%%%%%%%%%%%%%%%%%%%%%%%%%%%%%%%%%%  %%%%%%%%%%%%%%%%%%%%%%%%%%%%%%%%%%%%%%%%%%%%%%%%%%
%%%%%%%%%%%%%%%%%%%%%%%%%%%%%%%%%%%%%%%%%%%%%%%%%%%%%%%%%%%%%%%%%%%%%%%%%%%%%%%%%%%%%%%%%

\section{The case of  finitely generated fields}\label{sectionfinitelygene}

\hspace{\parindent}In this section,  we show  the (weak bi-anabelian  and strong bi-anabelian) $m$-step solvable Grothendieck conjecture(s) for  affine hyperbolic curves over  a  field finitely generated over the prime field (Theorem \ref{fingeneGCweak} and Theorem \ref{fingeneGCstrong}).   In subsection \ref{subsection4.1}, we define the localization of the category of geometrically reduced schemes over $k$ with respect to relative Frobenius morphisms when $p>0$. In subsections \ref{subsectionfingeneweak}, \ref{subsectionfingenestrong}, we show the main results of this section.   \\\ \\
{\bf Notation of section \ref{sectionfinitelygene} } In this section, we use  the following notation in addition to Notation (see Introduction).  
 \begin{itemize}
\item  Let  $k$ be a field of characteristic $p$ $(\geq 0)$.
\item   For $i=1,$ $2$, let $(X_{i},E_{i})$ (resp. $(X,E)$) be a smooth  curve of type $(g_{i},r_{i})$ (resp. $(g,r)$) over $k$  and set  $U_{i}:=X_{i}-E_{i}$ (resp. $U:=X-E$).
\end{itemize}

%%%%%%%%%%%%%%%%%%%%%%%%%%%%%%%%%%%%%%%%%%%%%%%%%%%%%%%%%%%%%%%%%%%%%%%%%%%%%%%%%%%%%%%%%
%%%%%%%%%%%%%%%%%%%%%%%%%%%%%%%%%%%%%     %%%%%%%%%%%%%%%%%%%%%%%%%%%%%%%%%%%%%%%%%%%%%%%%
%%%%%%%%%%%%%%%%%%%%%%%%%%%%%%%%%%%   %%  %%%%%%%%%%%%%%%%%%%%%%%%%%%%%%%%%%%%%%%%%%%%%%%%%
%%%%%%%%%%%%%%%%%%%%%%%%%%%%%%%%%     %%%  %%%%%%%%%%%%%%%%%%%%%%%%%%%%%%%%%%%%%%%%%%%%%%%%%%
%%%%%%%%%%%%%%%%%%%%%%%%%%%%%%%%     %%%%  %%%%%%%%%%%%%%%%%%%%%%%%%%%%%%%%%%%%%%%%%%%%%%%%%%
%%%%%%%%%%%%%%%%%%%%%%%%%%%%%%%     %%%%%  %%%%%%%%%%%%%%%%%%%%%%%%%%%%%%%%%%%%%%%%%%%%%%%%%%
%%%%%%%%%%%%%%%%%%%%%%%%%%%%%%%                        %%%%%%%%%%%%%%%%%%%%%%%%%%%%%%%%%%%%%%%%%%%%%
%%%%%%%%%%%%%%%%%%%%%%%%%%%%%%%%%%%%%%%   %%%%%%%%%%%%%%%%%%%%%%%%%%%%%%%%%%%%%%%%%%%%%%
%%%%%%%%%%%%%%%%%%%%%%%%%%%%%%%%%%%%%%%  %%%%%%%%%%%%%%%%%%%%%%%%%%%%%%%%%%%%%%%%%%%%%%%%%%
%%%%%%%%%%%%%%%%%%%%%%%%%%%%%%%%%%%%%%%%%%%%%%%%%%%%%%%%%%%%%%%%%%%%%%%%%%%%%%%%%%%%%%%%%

\subsection{The category $\text{Sch}_{k}^{\text{geo.red.}}$}\label{subsection4.1}
%When we consider the Grothendieck conjecture,  it is important to consider universal homeomorphisms between curves. 
%The reason is that  a universal homeomorphism is not necessarily isomorphic, but the image of  a universal homeomorphism by $\Pi^{(m)}(\cdot)$ is an isomorphism between the \'{e}tale  fundamental groups.
\hspace{\parindent} In this subsection, we define and investigate the localization of the category  of geometrically reduced schemes over $k$ with respect to relative Frobenius morphisms. This generalizes  the contents of \cite{St2002P} Appendix B. In the rest of this  subsection,  \textbf{we assume that $p>0$.} We write Sch (resp. $\text{Sch}_{k}$, resp.  $\text{Sch}^{\text{red.}}$) for the category of schemes (resp. $k$-schemes, resp. reduced schemes). We define  $\text{Sch}_{k}^{\text{geo.red.}}$ as the full subcategory of $\text{Sch}_{k}$ consisting of   all geometrically  reduced schemes over $k$.  

\begin{lemma}\label{epimono}
\begin{enumerate}[(1)]
\item  Let $Z$ be a reduced scheme over $\mathbb{F}_{p}$. Then $\text{Fr}_{Z}$ is  an epimorphism in $\text{Sch}$ and  a monomorphism in $\text{Sch}^{\text{red}}$. 
\item Let $Z$ be a geometrically reduced scheme over $k$. Then $\text{Fr}_{Z/k}$ is an epimorphism in $\text{Sch}$ and  a monomorphism in $\text{Sch}^{\text{red}}$. In particular, $\text{Fr}_{Z/k}$  is  an epimorphism and   a monomorphism in $\text{Sch}_{k}^{\text{geo.red.}}$.
\end{enumerate}
\end{lemma}
\begin{proof}
(1)  Since $Z$ is reduced, the  $p$-th power endomorphism $\text{Fr}_{Z}^{\#}:O_{Z}\rightarrow (\text{Fr}_{Z})_{*}O_{Z}$ is  clearly injective. Moreover, by definition,   $\text{Fr}_{Z}$ is surjective. Hence $\text{Fr}_{Z}$ is an epimorphism in $\text{Sch}$. Next, we show that $\text{Fr}_{Z}$ is  a monomorphism in $\text{Sch}^{\text{red}}$.  Let $Z'$ be a reduced scheme  and $f,g\in \text{Hom}_{\text{Sch}^{\text{red}}}(Z',Z)$ with $\text{Fr}_{Z}\circ f=\text{Fr}_{Z}\circ g$. Since  $f\circ \text{Fr}_{Z'}=\text{Fr}_{Z}\circ f=\text{Fr}_{Z}\circ g=g\circ \text{Fr}_{Z'}$ and $\text{Fr}_{Z'}$ is epimorphism in $\text{Sch}$, we get $f=g$.\\
(2) Since $\text{Fr}_{Z}$ is a monomorphism in $\text{Sch}^{\text{red}}$ by (1), $\text{Fr}_{Z/k}$ is also   a monomorphism in $\text{Sch}^{\text{red}}$.  Next, we show that $\text{Fr}_{Z/k}$ is  an epimorphism in Sch. Since absolute Frobenius morphisms $\text{Fr}_{Z}$ and $\text{Fr}_{\text{Spec}(k)}$ are universally homeomorphisms, $\text{Fr}_{Z/k}$ is surjective.  Thus,  it is sufficient  to show that  $\text{Fr}_{Z/k}^{\#}:O_{Z(1)}\rightarrow (\text{Fr}_{Z/k})_{*}O_{Z}$ is injective. By the standard limit argument,   we may assume that $Z$ is the spectrum of   a  geometrically reduced, finitely generated  $k$-algebra $A$.  Then the  injectivity follows from \cite{Du1995} Theorem 3(a)$\Rightarrow $(d), since  $Z\rightarrow  \text{Spec}(k)$ is flat. The second assertion follows from the first assertion.
\end{proof}

  We write  \textbf{Fr}  for the class consisting of  all isomorphism, all relative Frobenius morphisms of  geometrically reduced schemes over $k$, and their composites. We  define $\text{Sch}_{k,\textbf{Fr}^{-1}}^{\text{geo.red.}}$ as the category  obtained by localizing  $\text{Sch}_{k}^{\text{geo.red.}}$ with respect to \textbf{Fr} and  write $\mathcal{Q}_{k}: \text{Sch}_{k}^{\text{geo.red.}}\rightarrow \text{Sch}_{k,\textbf{Fr}^{-1}}^{\text{geo.red.}}$ for the localization functor.   For any  objects $Z_{1}$, $Z_{2}$ in $\text{Sch}_{k}^{\text{geo.red.}}$,   we write $\text{Hom}_{k}(\mathcal{Q}_{k}(Z_{1}),\mathcal{Q}_{k}(Z_{2})):=\text{Hom}_{\text{Sch}_{k,\textbf{Fr}^{-1}}^{\text{geo.red.}}}(\mathcal{Q}_{k}(Z_{1}),\mathcal{Q}_{k}(Z_{2}))$.  

\begin{remark}\label{18.1.1}
\begin{enumerate}[(i)]
\item  Let  $Z_{1}$, $Z_{2}$ be elements in $\text{Sch}_{k}^{\text{geo.red.}}$, and $n_{1}$, $n_{2}$ non-negative integers. Then we have the natural  map $\text{Hom}_{k}(Z_{1}(n_{1}),Z_{2}(n_{2}))\rightarrow \text{Hom}_{k}(\mathcal{Q}_{k}(Z_{1}),\mathcal{Q}_{k}(Z_{2}))\ f\mapsto\mathcal{Q}_{k}(\text{Fr}^{n_{2}}_{Z_{2}/k})^{-1}\circ  \mathcal{Q}_{k}(f)\circ \mathcal{Q}_{k}(\text{Fr}^{n_{1}}_{Z_{1}/k})$. By Lemma \ref{epimono}(2), \textbf{Fr} forms a right multiplicative system, see \cite{KS2006} Definition 7.1.5. In particular, by \cite{KS2006} Theorem 7.1.16, we obtain that the natural map 
\begin{equation*}\label{eq11.3.4}
\ilim{n}\text{Hom}_{k}(Z_{1},Z_{2}(n))\rightarrow \text{Hom}_{k}(\mathcal{Q}_{k}(Z_{1}),\mathcal{Q}_{k}(Z_{2}))
\end{equation*}
is bijective, where $n$ runs over all non-negative integers and transfer morphisms are defined as the  left composite of the  relative Frobenius morphisms. In particular, the functor $\mathcal{Q}_{k}$ is faithful by Lemma \ref{epimono}(2).   
%In the rest of this section, we simply write $Z$ and  $f$ instead of $\mathcal{Q}(Z)$ and  $\mathcal{Q}(f)$ for any  object $Z$ and any  morphism $f$ in  $\text{Sch}_{k}^{\text{geo.red.}}$, respectively. 

\item Let $L$ be a  separable algebraic extension  of $k$.  Then  the forgetful (faithful) functor $\text{Sch}_{L}\rightarrow \text{Sch}_{k}$ induces a faithful functor  $\tilde{u}_{L/k}: \text{Sch}_{L}^{\text{geo.red.}}\rightarrow \text{Sch}_{k}^{\text{geo.red.}}$.   We claim that $\tilde{u}_{L/k}$ induces a  faithful functor $u_{L/k}:\text{Sch}_{L,\textbf{Fr}^{-1}}^{\text{geo.red.}}\rightarrow \text{Sch}_{k.\textbf{Fr}^{-1}}^{\text{geo.red.}}$. Indeed,  to show that $\tilde{u}_{L/k}$ induces $u_{L/k}$, it is sufficient  to show that $\text{Fr}_{Y/L}$ is identified with $\text{Fr}_{Y/k}$ for any $Y\in\text{Sch}_{L}^{\text{geo.red.}}$.  Note that $k^{\frac{1}{p}}$ and $ L$ are linearly disjoint over $k$, since $k^{\frac{1}{p}}/k$ is a purely inseparable extension and $L/k$ is a separable extension.  Further,  we have that $Lk^{\frac{1}{p}}=L^{\frac{1}{p}}$, since  $L^{\frac{1}{p}}/Lk^{\frac{1}{p}}/L$ is a purely inseparable extension and $L^{\frac{1}{p}}/Lk^{\frac{1}{p}}/k^{\frac{1}{p}}$ is a separable extension.  Hence the homomorphism  $\phi: L\otimes_{k}k^{\frac{1}{p}} \rightarrow L^{\frac{1}{p}}$ is an isomorphism. Thus,   $\text{Fr}_{L/k}$ is an isomorphism. This implies that $\text{Fr}_{Y/L}$ is identified with $\text{Fr}_{Y/k}$.  The faithfulness of $u_{L/k}$ follows from that  $\ilim{n}\text{Hom}_{L}(Y_{1},Y_{2}(n))\rightarrow\ilim{n}\text{Hom}_{k}(Y_{1},Y_{2}(n))$ is injective for any $Y_{1}$, $Y_{2}\in\text{Sch}_{L}^{\text{geo.red.}}$.
\end{enumerate}
\end{remark}

For any   separable algebraic extension $L$ of $k$,  any objects $Z_{1}$, $Z_{2}$  in $\text{Sch}_{k}^{\text{geo.red.}}$,  any objects $Y_{1}$, $Y_{2}$  in $\text{Sch}_{L}^{\text{geo.red.}}$, and  any   morphism  $s_{1}: u_{L/k}\circ\mathcal{Q}_{L}(Y_{1})\rightarrow \mathcal{Q}_{k}(Z_{1})$, $s_{2}: u_{L/k}\circ\mathcal{Q}_{L}(Y_{2})\rightarrow \mathcal{Q}_{k}(Z_{2})$  in $\text{Sch}^{\text{geo.red.}}_{k,\textbf{Fr}^{-1}}$,   we define  $\text{Isom}_{L/k}(\mathcal{Q}_{L}(Y_{1})/\mathcal{Q}_{k}(Z_{1}),\mathcal{Q}_{L}(Y_{2})/\mathcal{Q}_{k}(Z_{2}))$ as the set
\[
\left\{(f_{Y},f_{Z})\in \Isom_{L}(\mathcal{Q}_{L}(Y_{1}),\mathcal{Q}_{L}(Y_{2}))\times  \Isom_{k}(\mathcal{Q}_{k}(Z_{1}),\mathcal{Q}_{k}(Z_{2}))\middle|
s_{2}\circ u_{L/k}(f_{Y})=f_{Z}\circ s_{1}\text{ in }\text{Sch}_{k,\textbf{Fr}^{-1}}^{\text{geo.red.}}.
\right\}.
\]
When $Y_{1}=Y_{2}$ and $Z_{1}=Z_{2}$, we define  $\text{Aut}_{L/k}(\mathcal{Q}_{L}(Y_{1})/\mathcal{Q}_{k}(Z_{1})):=\text{Isom}_{L/k}(\mathcal{Q}_{L}(Y_{1})/\mathcal{Q}_{k}(Z_{1}),\mathcal{Q}_{L}(Y_{2})/\mathcal{Q}_{k}(Z_{2}))$. Next, we investigate isomorphisms in $\text{Sch}_{k,\textbf{Fr}^{-1}}^{\text{geo.red.}}$.   

\begin{lemma}\label{example4.7}
Assume that $U_{1}$ is   hyperbolic and  $U_{1,\overline{k}}$ does not  descend to a curve over  $\overline{\mathbb{F}}_{p}$  (``non-isotrivial'' in the sense of \cite{St2002P}). 
\begin{enumerate}[(1)]
\item   There exists an integer  $\delta_{U_{1},U_{2}}\in\mathbb{Z}$ such that the map
\[
\ilim{n}\text{Isom}_{k}(U_{1}(n),U_{2}(n+\delta_{U_{1},U_{2}}))\rightarrow \text{Isom}_{k}(\mathcal{Q}_{k}(U_{1}),\mathcal{Q}_{k}(U_{2}))
\]
is bijective, where $n$ runs over all integers satisfying  $n\geq 0$ and $n+\delta_{U_{1},U_{2}}\geq0$ and the transfer maps are defined as  relative Frobenius twists $f\mapsto f(a)$ ($a\in\mathbb{Z}_{\geq 0}$). If, moreover, $\text{Isom}_{k}(\mathcal{Q}_{k}(U_{1}),\mathcal{Q}_{k}(U_{2}))\neq \emptyset$, then $\delta_{U_{1},U_{2}}$ is unique.
\item Let $L$ be a finite separable extension of $k$. Let     $s_{i}: V_{i}\twoheadrightarrow U_{i}$  be a connected finite \'{e}tale covering which is tame outside of $U_{i}$ . Assume that  the coefficient field of $V_{i}$ coincide with  $L$.   Then $V_{1,\overline{L}}$ does not  descend to a curve over  $\overline{\mathbb{F}}_{p}$. 
\item Let the  assumption and the notation be as in (2). Assume that  $\text{Isom}_{L/k}(\mathcal{Q}_{L}(V_{1})/\mathcal{Q}_{k}(U_{1}),\mathcal{Q}_{L}(V_{2})/\mathcal{Q}_{k}(U_{2}))\neq \emptyset$, then the natural  map  
\begin{equation*}\label{4.12q}
\ilim{n}\text{Isom}_{L/k}(V_{1}(n)/U_{1}(n),V_{2}(n+\delta_{U_{1},U_{2}})/U_{2}(n+\delta_{U_{1},U_{2}}))\rightarrow \text{Isom}_{L/k}(\mathcal{Q}_{L}(V_{1})/\mathcal{Q}_{k}(U_{1}),\mathcal{Q}_{L}(V_{2})/\mathcal{Q}_{k}(U_{2}))
\end{equation*}
is bijective, where $n$ runs over all integers satisfying  $n\geq 0$ and $n+\delta_{U_{1},U_{2}}\geq0$ and the transfer maps are defined as  relative Frobenius twists $(f_{V},f_{U})\mapsto (f_{V}(a),f_{U}(a))$ ($a\in\mathbb{Z}_{\geq 0}$).   In particular,  $\delta_{U_{1},U_{2}}=\delta_{V_{1},V_{2}}$ holds.
\end{enumerate}
\end{lemma}

\begin{proof}
(1) When  $\text{Isom}_{k}(\mathcal{Q}_{k}(U_{1}),\mathcal{Q}_{k}(U_{2}))=\emptyset$, we have that $\text{Isom}_{k}(U_{1}(a),U_{2}(b))= \emptyset$ for any $a$, $b\in\mathbb{Z}_{\geq 0}$, and hence the assertion is clear for any $\delta_{U_{1},U_{2}}$. We assume that $\text{Isom}_{k}(\mathcal{Q}_{k}(U_{1}),\mathcal{Q}_{k}(U_{2}))\neq \emptyset$.  The injectivity follows from Remark \ref{18.1.1}(i). Next, we show the surjectivity. Let $f$ be an element of $\text{Isom}_{k}(\mathcal{Q}_{k}(U_{1}),\mathcal{Q}_{k}(U_{2}))$. Then we can choose   $n_{2}\in\mathbb{Z}_{\geq 1}$ and $\rho_{1}: U_{1}\rightarrow U_{2}(n_{2})$  in $\text{Sch}_{k}^{\text{geo.red.}}$ as a  representative element of $f$ by Remark \ref{18.1.1}(i). Since $f$ is an  isomorphism in $\text{Sch}_{k,{\bf Fr}^{-1}}^{\text{geo.red.}}$, there exist  $N\in\mathbb{Z}_{\geq 0}$ and $\rho_{2}: U_{2}\rightarrow U_{1}(N)$ such that $\text{Fr}_{U_{1}/k}^{N+n_{2}}=\rho_{2}(n_{2})\circ \rho_{1}$.   The equality  implies that  $\rho_{1}$ is finite and $K(U_{1})/K(U_{2}(n_{2}))$ is a purely inseparable extension. Thus, there exists $n_{1}\in\mathbb{Z}_{\geq 0}$ such that $U_{1}(n_{1})\xrightarrow{\sim} U_{2}(n_{2})$ (\cite{Liu2002} Proposition 4.21) and that the isomorphism $U_{1}(n_{1})\xrightarrow{\sim} U_{2}(n_{2})$ represents $f$ (in the sense of Remark \ref{18.1.1}(i)).   Set $\delta_{U_{1},U_{2}}:=n_{2}-n_{1}$. If $U_{1}(n'_{1})\xrightarrow{\sim}U_{2}(n'_{2})$  in $\text{Sch}_{k}^{\text{geo.red.}}$ for some $n'_{1}$, $n'_{2}\in\mathbb{Z}_{\geq 0}$,  then we get $U_{1}(n_{1}+n'_{2})\xrightarrow{\sim}U_{1}(n_{1}'+n_{2})$  in $\text{Sch}_{k}^{\text{geo.red.}}$. By   \cite{St2002P} Corollary B.2.4, we obtain that $n_{1}+n'_{2}=n'_{1}+n_{2}$. In other words,  $n'_{2}-n'_{1}=n_{2}-n_{1}=\delta_{U_{1},U_{2}}$.  Hence the assertion follows.\\
(2) We have that the natural morphism $V_{1,\overline{L}}\rightarrow U_{1,\overline{k}}$ is dominant. Hence, by \cite{Ta2002} Lemma (1.32), the assumption ``$U_{1,\overline{k}}$ does not  descend to a curve over  $\overline{\mathbb{F}}_{p}$'' implies that  $V_{1,\overline{L}}$ does not  descend to a curve over  $\overline{\mathbb{F}}_{p}$.  Thus, the assertion follows.\\
(3) The injectivity follows from Remark \ref{18.1.1}(i). Next, we show the surjectivity. Let $(f_{V},f_{U}) $ be an element of $\text{Isom}_{L/k}(\mathcal{Q}_{L}(V_{1})/\mathcal{Q}_{k}(U_{1}),\mathcal{Q}_{L}(V_{2})/\mathcal{Q}_{k}(U_{2}))$.   By (1), Remark \ref{18.1.1}(i), and the equality $\mathcal{Q}_{k}(s_{2})\circ u_{L/k}(f_{V})=f_{U}\circ \mathcal{Q}_{k}(s_{1})$,  there exist $M$, $N$, $\alpha\in\mathbb{Z}_{\geq 0}$,  $\delta_{U_{1},U_{2}}$, $\delta_{V_{1},V_{2}}\in\mathbb{Z}$, $\phi_{U}: U_{1}(N)\xrightarrow{\sim}U_{2}(N+\delta_{U_{1},U_{2}})$,  $\phi_{V}: V_{1}(N)\xrightarrow{\sim}V_{2}(N+\delta_{V_{1},V_{2}})$ such that the diagram 
\begin{equation*}
\xymatrix@C=40pt@R=18pt{
V_{1}(N)\ar[rr]^-{\sim}_-{\phi_{V}}\ar@{->>}[d]_{s_{1}(N)} &&V_{2}(N+\delta_{V_{1},V_{2}})\ar@{->>}[d]^{s_{2}(N+\delta_{V_{1},V_{2}})}\\
U_{1}(N)\ar[d]^{\text{Fr}_{U_{1}(N)/k}^{\alpha}}&& U_{2}(N+\delta_{V_{1},V_{2}})\ar[d]^-{\text{Fr}_{U_{2}(N+\delta_{V_{1},V_{2}})/k}^{M-(N+\delta_{V_{1},V_{2}})}}\\
U_{1}(N+\alpha)\ar[r]^-{\sim}_-{\phi_{U}(\alpha)} &U_{2}(N+\alpha+\delta_{U_{1},U_{2}})\ar[r]^-{\text{Fr}_{U_{2}(N+\alpha+\delta_{U_{1},U_{2}})/k}^{M-(N+\alpha+\delta_{U_{1},U_{2}})}}&U_{2}(M)
}
\end{equation*}
is commutative in $\text{Sch}_{k}^{\text{geo.red.}}$.  Since the inseparable degree of  the composite of  the maps $V_{1}(N)\xrightarrow{\sim}V_{2}(N+\delta_{V_{1},V_{2}})\xrightarrow{s_{2}(N+\delta_{V_{1},V_{2}})} U_{2}(N+\delta_{V_{1},V_{2}})\xrightarrow{\text{Fr}}U_{2}(M)$  coincides with  the  inseparable degree of  the composite of  the maps $V_{1}(N)\xrightarrow{s_{1}(N)}U_{1}(N)\xrightarrow{\text{Fr}}U_{1}(N+\alpha)\xrightarrow{\sim}U_{2}(N+\alpha+\delta_{U_{1},U_{2}})\xrightarrow{\text{Fr}} U_{2}(M)$, we have  that 
\[
\alpha+(M-N-\alpha-\delta_{U_{1},U_{2}})=\text{log}_{p}([K(V_{1}(N)):K(U_{2}(M))]_{i})=M-N-\delta_{V_{1},V_{2}}.
\]
Thus, we obtain that  $\delta_{U_{1},U_{2}}=\delta_{V_{1},V_{2}}$. Set $n:=N+\alpha$. Then, by Lemma \ref{epimono}, we conclude  the  diagram
\begin{equation*}
\xymatrix{
V_{1}(n)\ar[r]^-{\phi_{V}(\alpha)}\ar@{->>}[d]_{s_{1}(n)} &V_{2}(n+\delta_{U_{1},U_{2}})\ar@{->>}[d]^{s_{2}(n+\delta_{U_{1},U_{2}})}\\
U_{1}(n)\ar[r]^-{\phi_{U}(\alpha)} &U_{2}(n+\delta_{U_{1},U_{2}}).
}
\end{equation*}
commutes. Thus, the assertion follows. 
\end{proof}

%We have that the map $\text{Isom}_{k^{\sep}/k}(\tilde{U}_{1}^{m}/U_{1},\tilde{U}_{2}^{m}/U_{2})\rightarrow \text{Isom}_{G_{k}}(\Pi^{(m)}_{U_{1}},\Pi^{(m)}_{U_{2}})$ factors  through  the set \\
%$\text{Isom}_{k^{\sep}/k}(\mathcal{Q}(\tilde{U}_{1}^{m})/\mathcal{Q}(U_{1}),\mathcal{Q}(\tilde{U}_{2}^{m})/\mathcal{Q}(U_{2}))$.

%%%%%%%%%%%%%%%%%%%%%%%%%%%%%%%%%%%%%%%%%%%%%%%%%%%%%%%%%%%%%%%%%%%%%%%%%%%%%%%%%%%%%%%%%
%%%%%%%%%%%%%%%%%%%%%%%%%%%%%%%%%%%%%     %%%%%%%%%%%%%%%%%%%%%%%%%%%%%%%%%%%%%%%%%%%%%%%%
%%%%%%%%%%%%%%%%%%%%%%%%%%%%%%%%%%%   %%  %%%%%%%%%%%%%%%%%%%%%%%%%%%%%%%%%%%%%%%%%%%%%%%%%
%%%%%%%%%%%%%%%%%%%%%%%%%%%%%%%%%     %%%  %%%%%%%%%%%%%%%%%%%%%%%%%%%%%%%%%%%%%%%%%%%%%%%%%%
%%%%%%%%%%%%%%%%%%%%%%%%%%%%%%%%     %%%%  %%%%%%%%%%%%%%%%%%%%%%%%%%%%%%%%%%%%%%%%%%%%%%%%%%
%%%%%%%%%%%%%%%%%%%%%%%%%%%%%%%     %%%%%  %%%%%%%%%%%%%%%%%%%%%%%%%%%%%%%%%%%%%%%%%%%%%%%%%%
%%%%%%%%%%%%%%%%%%%%%%%%%%%%%%%                        %%%%%%%%%%%%%%%%%%%%%%%%%%%%%%%%%%%%%%%%%%%%%
%%%%%%%%%%%%%%%%%%%%%%%%%%%%%%%%%%%%%%%   %%%%%%%%%%%%%%%%%%%%%%%%%%%%%%%%%%%%%%%%%%%%%%
%%%%%%%%%%%%%%%%%%%%%%%%%%%%%%%%%%%%%%%  %%%%%%%%%%%%%%%%%%%%%%%%%%%%%%%%%%%%%%%%%%%%%%%%%%
%%%%%%%%%%%%%%%%%%%%%%%%%%%%%%%%%%%%%%%%%%%%%%%%%%%%%%%%%%%%%%%%%%%%%%%%%%%%%%%%%%%%%%%%%

\subsection{The weak  bi-anabelian results over finitely generated  fields}\label{subsectionfingeneweak}
\hspace{\parindent}In this subsection, we show the weak bi-anabelian  $m$-step solvable  Grothendieck conjecture for affine hyperbolic curves over a field finitely generated over the prime field.  In subsection \ref{subsection4.1}, we define the category  $\text{Sch}_{k,\textbf{Fr}^{-1}}^{\text{geo.red.}}$ when $p>0$. To consider the case that $p=0$ and $p>0$ at the same time, we define the following definition.

%\begin{definition}  
%We define $\mathfrak{S}_{k}$ as the category $\text{Sch}_{k}^{\text{geo.red.}}$ (resp.  $\text{Sch}_{k,\textbf{Fr}^{-1}}^{\text{geo.red.}}$)  when $p=0$ (resp, $p>0$).  Let $L$ be an extension of $k$.  Let    $Y_{i}$ be  an object in $\mathfrak{S}_{L}$, $Z_{i}$   an object in $\mathfrak{S}_{k}$,   and    $Y_{i}\rightarrow Z_{i}$  a morphism in $\mathfrak{S}_{k}$ for $i=1,2$.   We write $\text{Hom}_{\mathfrak{S}_{k}}(Z_{1},Z_{2})$ for the set  $\text{Hom}_{k}(Z_{1},Z_{2})$ (resp.   $\text{Hom}_{k}(\mathcal{Q}_{k}(Z_{1}),\mathcal{Q}_{k}(Z_{2}))$) and write  $\text{Isom}_{L/k,\mathfrak{S}}(Y_{1}/Z_{1},Y_{2}/Z_{2})$ for the set  $\text{Isom}_{L/k}(Y_{1}/Z_{1},Y_{2}/Z_{2})$ (resp.   $\text{Isom}_{L/k}(\mathcal{Q}_{L}(Y_{1})/\mathcal{Q}_{k}(Z_{1}),\mathcal{Q}_{L}(Y_{2})/\mathcal{Q}_{k}(Z_{2}))$). When $Y_{1}=Y_{2}$ and $Z_{1}=Z_{2}$, we define  $\text{Aut}_{L/k,\mathfrak{S}}(Y_{1}/Z_{1}):=\text{Isom}_{L/k,\mathfrak{S}}(Y_{1}/Z_{1},Y_{2}/Z_{2})$.
% \end{definition}

\begin{definition}  
We define $\mathfrak{S}_{k}$ as the category $\text{Sch}_{k}^{\text{geo.red.}}$ (resp.  $\text{Sch}_{k,\textbf{Fr}^{-1}}^{\text{geo.red.}}$)  when $p=0$ (resp, $p>0$).  Let $L$ be an extension of $k$.  Let    $Y_{i}$ be  an object in $\mathfrak{S}_{L}$, $Z_{i}$   an object in $\mathfrak{S}_{k}$,   and    $Y_{i}\rightarrow Z_{i}$  a morphism in $\mathfrak{S}_{k}$ for $i=1,2$.   We  write  $\text{Isom}_{\mathfrak{S}_{L}/\mathfrak{S}_{k}}(Y_{1}/Z_{1},Y_{2}/Z_{2})$ for the set  $\text{Isom}_{L/k}(Y_{1}/Z_{1},Y_{2}/Z_{2})$ (resp.   $\text{Isom}_{L/k}(\mathcal{Q}_{L}(Y_{1})/\mathcal{Q}_{k}(Z_{1}),\mathcal{Q}_{L}(Y_{2})/\mathcal{Q}_{k}(Z_{2}))$). When $Y_{1}=Y_{2}$ and $Z_{1}=Z_{2}$, we define  $\text{Aut}_{\mathfrak{S}_{L}/\mathfrak{S}_{k}}(Y_{1}/Z_{1}):=\text{Isom}_{\mathfrak{S}_{L}/\mathfrak{S}_{k}}(Y_{1}/Z_{1},Y_{2}/Z_{2})$.
\end{definition}
 
  \begin{remark}\label{fndisoms}
If  a morphism $\phi:U_{1}\xrightarrow[k]{}  U_{2}$ is a universal homeomorphism (e.g., $p>0$, $n\in\mathbb{Z}_{\geq 1}$, $U_{2}=U_{1}(n)$, and $\phi=\text{Fr}^{n}_{U_{1}/k}$), then the homomorphism $\Pi_{U_{1}}^{(m)}\xrightarrow[G_{k}]{}\Pi_{U_{2}}^{(m)}$ induced by $\phi$ (up to inner automorphism of $\overline{\Pi}_{U_{2}}^{m})$ is an isomorphism. Hence, by Remark \ref{18.1.1}(1),  we obtain a natural  map $\text{Isom}_{\mathfrak{S}_{k}}(U_{1},U_{2})\rightarrow\text{Isom}_{G_{k}}(\Pi_{U_{1}}^{(m)},\Pi_{U_{2}}^{(m)})/\text{Inn}(\overline{\Pi}_{U_{2}}^{m})$, and  a natural  map $\text{Isom}_{\mathfrak{S}_{k^{\sep}}/\mathfrak{S}_{k}}(\tilde{U}_{1}^{m}/U_{1},\tilde{U}_{2}^{m}/U_{2})\rightarrow \text{Isom}_{G_{k}}(\Pi_{U_{1}}^{(m)},\Pi_{U_{2}}^{(m)})$.
 \end{remark}

\begin{definition}
 Let $S$ be a scheme. Let  $N\in \mathbb{Z}$ be a positive integer that is invertible on $S$.  Let $(\mathcal{X},\mathcal{E})$ be a smooth curve of type ($g,r$) over $S$ and    $p:\mathcal{X}\rightarrow S$ the structure morphism. Set $\mathcal{U}:=\mathcal{X}-\mathcal{E}$.  We call an  isomorphism $\theta:R^{1}p_{*}\mathbb{Z}/N\mathbb{Z}\xrightarrow{\sim}(\mathbb{Z}/N\mathbb{Z})^{2g}$ of \'{e}tale sheaves on $S$ $a$ $level$ $N$ $structure$ on $\mathcal{X}/S$.  (If there is no risk of confusion, we also call it $a$  $level$ $N$ $structure$ on $\mathcal{U}/S$.)
\end{definition}

\begin{remark}\label{rem4.2}\begin{enumerate}[(i)]
\item  (cf. \cite{St2002P} section 7.2.2) Let $f: S'\rightarrow S$ be a morphism. Let  $p':\mathcal{X}'\rightarrow S'$ be the base change  of the proper, smooth curve $p:\mathcal{X}\rightarrow S$ by $f$. By the proper base change theorem for  \'{e}tale cohomology, we obtain a canonical isomorphism $f^{*}R^{1}p_{*}\mathbb{Z}/N\mathbb{Z}\xrightarrow{\sim} R^{1}p'_{*}\mathbb{Z}/N\mathbb{Z}$. Thus, a level $N$ structure on  $\mathcal{X}/S$ induces a level $N$ structure on  $\mathcal{X}'/S'$.  For any  point $s\in S$, we write $\theta_{s}$ for the level $N$ structure on $\mathcal{X}_{s}/\kappa(s)$ induced by a level $N$ structure $\theta$ on $\mathcal{X}/S$.   
\item Let  $\mathcal{U}\rightarrow S$ be a smooth curve. Then there exists a finite,  \'{e}tale covering    $S'\rightarrow S$ such that   the base change $\mathcal{U}' \rightarrow S'$ of $\mathcal{U}\rightarrow S$ by $S'\rightarrow S$ has a level $N$ structure. 
\end{enumerate}
\end{remark}

We write $\mathcal{M}_{g,r}[N]$ for the moduli stack of  proper, smooth curves of genus $g$ equipped with $r$ disjoint ordered sections and  a level $N$ structure over $\text{Spec}(\mathbb{Z}[\frac{1}{N}])$.  We know that the moduli stack $\mathcal{M}_{g,r}$ $(=\mathcal{M}_{g,r}[1])$ is not always a scheme. 

\begin{lemma}\label{moduli4}
Assume that $2g+r-2>0.$ Let $N\geq 3$. Then   $\mathcal{M}_{g,r}[N]$ is a separated  scheme of finite type over $\text{Spec}(\mathbb{Z}[\frac{1}{N}])$.
\end{lemma}
\begin{proof}
 $\mathcal{M}_{g,r+1}\rightarrow \mathcal{M}_{g,r}$ is relatively representable and $\mathcal{M}_{g,r+1}[N]\cong \mathcal{M}_{g,r+1}\times_{\mathcal{M}_{g,r}} \mathcal{M}_{g,r}[N]$. Hence  we may assume either $(g,r)=(0,3),(1,1)$ or ``$g\geq 2$ and $r=0$".  When $(g,r)=(0,3)$, the assertion is clear because $\mathcal{M}_{0,3}[N]$ is isomorphic to $\text{Spec}(\mathbb{Z}[\frac{1}{N}])$.  We have that $\mathcal{M}_{1,1}[N]$ is  a separated  scheme of finite type over $\text{Spec}(\mathbb{Z}[\frac{1}{N}])$ by \cite{KatzMazur} Theorem 3.7.1. When  $g\geq 2$,   $\mathcal{M}_{g,0}[N]$ is   a separated  scheme of finite type over $\text{Spec}(\mathbb{Z}[\frac{1}{N}])$  by  \cite{Se1960} Th\'{e}or\`{e}me (or \cite{DM1969} (5.14)).
\end{proof}

\begin{lemma}\label{injelemma4-2}
Assume that $k$ is finitely generated over the prime field and that $U_{1}$ is hyperbolic. Assume that $U_{1,\overline{k}}$ does not  descend to a curve over  $\overline{\mathbb{F}}_{p}$  when $p>0$.   Let $N\in\mathbb{Z}_{\geq 3}$ with $p\nmid N$. Let $\overline{\Pi}_{U_i}^{1}/N$ be  the maximal exponent $N$ quotient of $\overline{\Pi}^{1}_{U_{i}}$.  Then the natural map
\begin{equation}\label{map11.3}
\Isom_{\mathfrak{S}_{k}}(U_1,U_2)\rightarrow \Isom_{G_{k}}(\overline{\Pi}_{U_{1}}^{1}/N,\overline{\Pi}_{U_{2}}^{1}/N),
\end{equation}
is injective, where the map is induced by using Remark \ref{fndisoms}. In particular, the map $\Isom_{\mathfrak{S}_{k}}(U_{1},U_{2}) \rightarrow \Isom_{G_{k}}(\Pi_{U_{1}}^{(m)},\Pi_{U_{2}}^{(m)})/\text{Inn}(\overline{\Pi}_{U_{2}}^{m})$  is also   injective. 
\end{lemma}
\begin{proof}
 If  $\Isom_{k}(U_{1},U_{2})=\emptyset$ , then the assertions are clear.  Hence we may assume that  $ (X_{1},E_{1})=(X_{2},E_{2})$. We write $X$, $E$, $U$, $g$, $r$ instead of $X_{i}$, $E_{i}$, $U_{i}$, $g_{i}$, $r_{i}$, respectively.    First, we show that the natural map $\rho: \text{Aut}_{k}(U)\rightarrow \text{Aut}_{G_{k}}(\overline{\Pi}_{U}^{1}/N)$ is injective.   By Lemma \ref{wflemma}, we get 
\begin{equation*}\label{eq19.1}
0\rightarrow \mathbb{Z}/N(1)\rightarrow \mathbb{Z}/N[E(k^{\sep})]\bigotimes_{\mathbb{Z}/N} \mathbb{Z}/N(1)\xrightarrow{} \overline{\Pi}_{U}^{1}/N\rightarrow J_{X}[N]\rightarrow 0. \ \ \ (r>0)
\end{equation*}
\begin{equation*}\label{eq19.1-2}
 \overline{\Pi}_{U}^{1}/N\xrightarrow{\sim} J_{X}[N]\ \ \ (r=0)
\end{equation*}
Let $f\in \text{Ker}(\Aut_{k}(U)\rightarrow \Aut_{G_{k}}(\overline{\Pi}_{U}^{1}/N))$ and $f^{*}$  the automorphism of $J_{X}$   induced by $f$.  When $r>0$, the isomorphism $\mathbb{Z}/N[E(k^{\sep})]\bigotimes_{\mathbb{Z}/N} \mathbb{Z}/N(1)/(\mathbb{Z}/N(1)) \xrightarrow{\sim}\mathbb{Z}/N[E(k^{\sep})]\bigotimes_{\mathbb{Z}/N}  \mathbb{Z}/N(1)/(\mathbb{Z}/N(1))$ induced by $f$ is  trivial. Hence  the bijection $E(k^{\text{sep}})\xrightarrow{\sim} E(k^\text{sep})$ induced by $f$ is  trivial. Thus, when   $g=0$, we get $f=\text{id}$, since  $|E(k^{\text{sep}})|\geq 3$.   Next, we  consider the case that $g\geq 1$.  We have that  $J_{X}[N]\xrightarrow{\sim}J_{X}[N]$  induced by $f$ is trivial.  $f^{*}$  has  finite order  by  the  hyperbolicity of $(X,E)$.   Thus, we get $f^{*}=\text{id}$ by  \cite{Se1960} Th\'{e}or\`{e}me. Therefore, we get $f=\text{id}$.  Hence the natural map $\rho: \Isom_{k}(U_1,U_2)\rightarrow \Isom_{G_{k}}(\overline{\Pi}_{U_{1}}^{1}/N,\overline{\Pi}_{U_{2}}^{1}/N)$ is injective.  When $p=0$, the first assertion follows.  When $p>0$, the first  assertion follows from the  injectivety of $\rho$ and   Lemma \ref{example4.7}(1).   Observe that we have the maps $\text{Aut}_{k}(U) \rightarrow \text{Aut}_{G_{k}}(\Pi_{U}^{(m)})/\text{Inn}(\overline{\Pi}_{U}^{m})\rightarrow \text{Aut}_{G_{k}}(\Pi_{U}^{(1)}) \rightarrow  \text{Aut}_{G_{k}}(\overline{\Pi}_{U}^{1}/N).$ Hence the second assertion follows from the first assertion.
\end{proof}

\begin{lemma}\label{dzetaslem}
Let $t$ be a finite field of characteristic $p$ and  $V$ an integral scheme of finite type over $t$ satisfying $\text{dim}(V)>0$. For any point $v$, set $d_{v}:=[\kappa(v):\mathbb{F}_{p}]$. Then $\underset{v\in V^{\text{cl}}}\cap d_{v}\hat{\mathbb{Z}}=\{0\}$ holds. 
\end{lemma}
\begin{proof}
 By replacing $V$ with a suitable open subscheme  if necessary, we may assume that $V$ is affine. By the Noether normalization lemma (\cite{Liu2002} Lemma 2.1.9), there exists a finite surjective morphism $V\rightarrow \mathbb{A}_{t}^{\text{dim}(V)}$ over $t$. Hence we may assume that $V= \mathbb{A}_{t}^{m}$ for some $m\in\mathbb{Z}_{> 0}$. For any $n\in\mathbb{Z}_{>0}$, we have that $\mathbb{F}_{p^{n}}-\underset{0<a<n}\cup\mathbb{F}_{p^{a}}$ is not empty. Thus, $\underset{v\in (\mathbb{A}_{t}^{m})^{\text{cl}}}\cap d_{v}\hat{\mathbb{Z}}=\underset{n\in\mathbb{Z}_{>0}}\cap n\hat{\mathbb{Z}}=\{0\}$ follows.
\end{proof}

We write  $\rho^{n}_{C}$ for the natural isomorphism $C(n)=C\underset{t,\text{Fr}^{n}_{t}}\times t\rightarrow C$ (not necessary over $t$) for any smooth curve $C$ over any finite field $t$ and any non-negative integer $n\in\mathbb{Z}_{\geq 0}$.  Let us prove the following lemma which is  important  in the proof of  the weak and strong  bi-anabelian  $m$-step solvable  Grothendieck conjectures.

\begin{lemma}\label{stxlemma2fin}
Assume that $k$ is finitely generated over the prime field and that   $U_{1}$ is affine hyperbolic. Assume that $U_{1,\overline{k}}$ does not  descend to a curve over  $\overline{\mathbb{F}}_{p}$ when $p>0$.   Assume that  $m$ satisfies 
\begin{equation*}
\begin{cases}
\ \ m\geq  4 & \ \ (\text{if }r_{1}\geq  3 \text{ and }  (g_{1},r_{1})\neq (0,3), (0,4))\\
\ \ m\geq 5 & \ \ (\text{if }r_{1}<3 \text{ or }  (g_{1},r_{1})= (0,3), (0,4)).
\end{cases}
\end{equation*}
Let  $\Phi: \Pi^{(m)}_{U_{1}}\xrightarrow[G_{k}]{\sim} \Pi^{(m)}_{U_{2}}$ be a $G_{k}$-isomorphism.   Let  $S$ be an integral regular scheme  of finite type over $\text{Spec}(\mathbb{Z})$ with function field $k$ and $\eta$ the generic point of $S$.  Let $N\in\mathbb{Z}_{\geq3}$ be  an integer  which is invertible on $S$. Let $(\mathcal{X}_{i},\mathcal{E}_{i})$ be a smooth curve of type $(g_{i},r_{i})$ over $S$ with generic fiber $(X_{i},E_{i})$ and $\mathcal{U}_{i}:=\mathcal{X}_{i}-\mathcal{E}_{i}$ for $i=1$, $2$.  Then, when $p=0$ (resp. $p>0$),   there exists (resp. exist) a unique isomorphism $f_{\Phi}^{S}: \mathcal{U}_{1}\underset{S}\rightarrow \mathcal{U}_{2}$ (resp. a unique pair  $n_{1}$, $n_{2}\in \mathbb{Z}_{\geq 0}$ with $n_{1} n_{2}=0$ and a unique isomorphism $f_{\Phi}^{S}: \mathcal{U}_{1}(n_{1})\underset{S}\rightarrow \mathcal{U}_{2}(n_{2})$) such that  the following condition ($\dag$) is satisfied for every  $s\in S^{\text{cl}}$. 
\begin{itemize}
 \item[(\dag)]  Let    $f_{\Phi,s}^{S}: \mathcal{U}_{1,s}\xrightarrow[\kappa(s)]{\sim} \mathcal{U}_{2,s}$  (resp. $f_{\Phi,s}^{S}: \mathcal{U}_{1,s}(n_{1})\xrightarrow[\kappa(s)]{\sim} \mathcal{U}_{2,s}(n_{2})$) be the isomorphism  induced by   $f_{\Phi}^{S}$.   Let  $\Phi_{s}$ be the  image of $\Phi$ by the map  $ \Isom_{G_{k}}(\Pi_{U_{1}}^{(m)},\Pi_{U_{2}}^{(m)})\rightarrow \text{Isom}_{G_{\kappa(s)}}(\Pi^{(m-2)}_{\mathcal{U}_{1,s}},\Pi^{(m-2)}_{\mathcal{U}_{2,s}})$ induced by   Corollary \ref{goodcor}. Let $f_{{\Phi}_{s}}$ be  the image of $\Phi_{s}$ by the map  $\text{Isom}_{G_{\kappa(s)}}(\Pi^{(m-2)}_{\mathcal{U}_{1,s}},\Pi^{(m-2)}_{\mathcal{U}_{2,s}})\rightarrow\Isom(\mathcal{U}_{1,s},\mathcal{U}_{2,s})$  induced by  Theorem \ref{finGCweak}.  Then $f_{\Phi,s}^{S}=f_{\Phi_{s}}$ (resp. $f_{\Phi,s}^{S}=(\rho^{n_{2}}_{\mathcal{U}_{2,s}})^{-1} \circ f_{{\Phi}_{s}}\circ \rho^{n_{1}}_{\mathcal{U}_{1,s}}$) holds. 
\end{itemize}
\end{lemma}

\begin{proof}
By Proposition \ref{severalinvrecoprop}, we obtain that $g_{1}=g_{2}$ and $r_{1}=r_{2}$. In particular, $U_{2}$ is also affine hyperbolic.  We write $g$, $r$, instead of $g_{i}$, $r_{i}$, respectively. 

First, we show the uniqueness of $f_{\Phi}^{S}$ (resp. ($n_{1}, n_{2},f_{\Phi}^{S}$)). Assume that there exists $\tilde{f}_{\Phi}^{S}$ (resp. ($\tilde{n}_{1},\tilde{n}_{2},\tilde{f}_{\Phi}^{S}$)) that satisfies the condition (\dag). Let $s$ be a closed point of  $S$. Then  (\dag) implies that $f_{\Phi,s}^{S}=f_{{\Phi}_{s}}=\tilde{f}_{\Phi,s}^{S}$ when $p=0$.  When $p>0$, we have that $n_{1}-n_{2}=\delta_{U_{1},U_{2}}=\tilde{n}_{1}-\tilde{n}_{2}$ by \cite{St2002P} Corollary B.2.4. Hence $(n_{1},n_{2})=(\tilde{n}_{1},\tilde{n}_{2})$ follows, since $n_{1}n_{2}=\tilde{n}_{1}\tilde{n}_{2}=0$ and $n_{1}$, $n_{2}$, $\tilde{n}_{1}$, $\tilde{n}_{2}$ are non-negative. Thus, (\dag) implies that $f_{\Phi,s}^{S}=(\rho^{n_{2}}_{\mathcal{U}_{2,s}})^{-1} \circ f_{{\Phi}_{s}}\circ \rho^{n_{1}}_{\mathcal{U}_{1,s}}=\tilde{f}_{\Phi,s}^{S}$.  Since  any  closed point $x$ of $\mathcal{U}_{1}$ is contained in some fiber $\mathcal{U}_{1,s'}$ ($s'\in S^{\text{cl}}$),  $f_{\Phi}^{S}=\tilde{f}_{\Phi}^{S}$ follows by Lemma  \ref{stxlemma1}.  (Note that, when $p>0$, the integer $a$ in Lemma  \ref{stxlemma1} is zero in this case, since $f_{\Phi}^{S}$ and $\tilde{f}_{\Phi}^{S}$ are $S$-morphisms.)\par 
  Next, we construct $f_{\Phi}^{S}$ under the two extra assumptions:  ``(i): $\mathcal{E}_{1}$ is a disjoint union  of ordered sections  $(\psi_{1,j}:S\rightarrow \mathcal{E}_{1})_{1\leq j\leq r}$ over $S$'' and ``(ii): there exists a level $N$ structure $\theta_{1}:R^{1}p_{1*}\mathbb{Z}/N\mathbb{Z}\xrightarrow{\sim}(\mathbb{Z}/N\mathbb{Z})^{2g_{1}}$ on $\mathcal{U}_{1}/S$'', where $p_{i}:\mathcal{X}_{i}\twoheadrightarrow S$ stands for  the structure morphism.  By Proposition \ref{inertiareco}(2) (with $(h,m')=(1,0)$ (resp. $(2,0)$) for $m=4$ (resp. $m\geq 5$)), $\Phi$ induces a unique  $G_{k}$-equivariant bijection $\tilde{E}_{1}^{0}\xrightarrow{\sim} \tilde{E}_{2}^{0}$ satisfying that  the diagram (\ref{innercommu}) is commutative.   Hence   $\mathcal{E}_{2}$ is also a  disjoint union  of sections  $(\psi_{2,j}:S\rightarrow \mathcal{E}_{2})_{1\leq j\leq r}$ over $S$ with  the order   induced by  $\Phi$  and  the order of $(\psi_{1,j}:S\rightarrow \mathcal{E}_{1})_{1\leq j\leq r}$. Moreover,  we obtain a level $N$ structure $\theta_{2}$ on $\mathcal{U}_{2}/S$  from  $\Phi$ and $\theta_{1}$.  Let $\zeta_{1},\zeta_{2}: S\rightarrow \mathcal{M}_{g,r}[N]$ be the canonical morphisms  classifying $(\mathcal{X}_{1},(\psi_{1,j}:S\rightarrow \mathcal{E}_{1})_{1\leq j\leq r}, \theta_{1})$, $(\mathcal{X}_{2},(\psi_{2,j}:S\rightarrow \mathcal{E}_{2})_{1\leq j\leq r}, \theta_{2})$, respectively.  
Since $\kappa(s)$ is finite, there exist positive integers $n_{1,s}$, $n_{2,s}$ such that the composite $f_{s}'$  of the morphisms $\mathcal{U}_{1}(n_{1,s})\xrightarrow{\rho^{n_{1,s}}_{\mathcal{U}_{1,s}}}\mathcal{U}_{1,s}\xrightarrow[f_{{\Phi}_{s}}]{\sim}\mathcal{U}_{2,s}\xrightarrow{(\rho^{n_{2,s}}_{\mathcal{U}_{2,s}})^{-1}} \mathcal{U}_{2,s}(n_{2,s})$  is a $\kappa(s)$-isomorphism. 
 By Proposition \ref{inertiareco}(2) (with $(h,m')=(1,0)$ (resp. $(2,0)$) for $m=4$ (resp. $m\geq 5$)), $\Phi_{s}$   induces a unique  $G_{\kappa(s)}$-equivariant bijection $\tilde{\mathcal{E}}_{1,s}^{0}\xrightarrow{\sim} \tilde{\mathcal{E}}_{2,s}^{0}$ satisfying  the diagram (\ref{innercommu}) is commutative. By Theorem \ref{finGCweak}(iv),  $\Phi_{s}$ and $f_{\Phi_{s}}$ induce the same bijection  $\tilde{\mathcal{E}}_{1,s}^{0}\xrightarrow{\sim} \tilde{\mathcal{E}}_{2,s}^{0}$.  Thus,  $f'_{s}$ preserves  the orders of $(\psi_{1,j})_{1\leq j\leq r}$ and $(\psi_{2,j})_{1\leq j\leq r}$.   The level $N$ structures $\theta_{1}$, $\theta_{2}$ on $\mathcal{U}_{1}/S$, $\mathcal{U}_{2}/S$ induce level $N$ structures $\theta_{1,s}$, $\theta_{2,s}$ on $\mathcal{U}_{1,s}/\kappa(s)$, $\mathcal{U}_{2,s}/\kappa(s)$, respectively. By Theorem \ref{finGCweak}(iv), $\Phi_{s}$ and $f_{\Phi_{s}}$ induce the same isomorphism $\overline{\Pi}_{\mathcal{X}_{1,s}}^{1}/N\xrightarrow{\sim}\overline{\Pi}_{\mathcal{X}_{2,s}}^{1}/N$.  Hence  $f_{s}'$ preserves the  level $N$ structures $\theta_{1,s}(n_{1,s})$, $\theta_{2,s}(n_{2,s})$ on $\mathcal{U}_{1,s}(n_{1,s})/\kappa(s)$, $\mathcal{U}_{2,s}(n_{2,s})/\kappa(s)$ induced by $\theta_{1,s}$, $\theta_{2,s}$, respectively. Thus, $ (\zeta_{1}|_{s})\circ\text{Fr}_{\kappa(s)}^{n_{1,s}}= (\zeta_{2}|_{s})\circ\text{Fr}_{\kappa(s)}^{n_{2,s}}$ follows. In particular, $\zeta_{1}(s)=\zeta_{2}(s)$ follows  for any $s\in S^{\text{cl}}$.  Therefore, by Lemma \ref{stxlemma1} and Lemma \ref{moduli4}, $\zeta_{1}=\zeta_{2}$ (resp. there exists a unique pair  $n_{1}$, $n_{2}\in \mathbb{Z}_{\geq 0}$ with $n_{1} n_{2}=0$  such that $\zeta_{1}\circ \text{Fr}_{S}^{n_{1}}= \zeta_{2}\circ \text{Fr}_{S}^{n_{2}}$) follows when $p=0$ (resp. $p>0$). Hence we get a unique isomorphism $(\mathcal{X}_{1},(\psi_{1,j})_{1\leq j\leq r},\theta_{1})\rightarrow  (\mathcal{X}_{2},(\psi_{2,j})_{1\leq j\leq r},\theta_{2}) $ (resp. $(\mathcal{X}_{1},(\psi_{1,j})_{1\leq j\leq r},\theta_{1})(n_{1})\rightarrow  (\mathcal{X}_{2},(\psi_{2,j})_{1\leq j\leq r},\theta_{2})(n_{2})$)  over $S$, which induces an isomorphism $f_{\Phi}^{S}: \mathcal{U}_{1}\rightarrow \mathcal{U}_{2}$ (resp. $f_{\Phi}^{S}: \mathcal{U}_{1}(n_{1})\rightarrow \mathcal{U}_{2}(n_{2})$) over $S$. \par
Next, we show that the isomorphism $f_{\Phi}^{S}$ satisfies (\dag) for every  $s\in S^{\text{cl}}$.  
First, we assume that  $p>0$.   
Let $s$ be an element of    $S^{\text{cl}}$ and $\eta$ the generic point of $S$. 
By Theorem  \ref{finGCweak}(i),   $\Phi_{s}$  induces an isomorphism  $\tilde{f}^{0}_{\Phi_{s}}:\mathcal{U}_{1,s}^{0}\xrightarrow{\sim}\mathcal{U}_{2,s}^{0}$.  
By Lemma \ref{wflemma}, we obtain that $(\overset{2g+r-1}\wedge\overline{\Pi}_{\mathcal{U}_{i,s}}^{1,\pro p' })^{\otimes 2}=\hat{\mathbb{Z}}^{\pro p'}(2(g+r-1))$. 
We write $\beta_{s}$, $\beta_{\eta}$ for the elements of $\text{Aut}(\hat{\mathbb{Z}}^{\pro p'}(2(g+r-1)))=(\hat{\mathbb{Z}}^{\pro p'})^{\times}$ induced by $\tilde{f}^{0}_{\Phi_{s}}$, $\Phi$, respectively. 
Since  $\tilde{f}_{\Phi_{s}}^{0}$ and $\Phi_{s}$ induces the same isomorphism  $\overline{\Pi}^{1,\text{pro-}p'}_{\mathcal{U}_{1,s}}\xrightarrow{\sim}\overline{\Pi}^{1,\text{pro-}p'}_{\mathcal{U}_{2,s}}$ by Theorem  \ref{finGCweak}(iv),  we obtain that $\beta_{s}=\beta_{\eta}$.   We write  $\alpha_{s}$ for an element of $\hat{\mathbb{Z}}$ such that the element of  $G_{\mathbb{F}_{p}}=\text{Aut}(\overline{\kappa(s)})$ induced by $\tilde{f}^{0}_{\Phi_{s}}$ is $\text{Fr}^{\alpha_{s}}_{\mathbb{F}_{p}}$.  Then $(p^{\alpha_{s}})^{2(g+r-1)}=\beta_{s}=\beta_{\eta}=\beta_{t}=(p^{\alpha_{t}})^{2(g+r-1)}$ follows for any $t\in S^{\text{cl}}$.     
Since  the  homomorphism $\hat{\mathbb{Z}}\rightarrow (\hat{\mathbb{Z}}^{\text{pro-}p'})^{\times}$, $\gamma\mapsto p^{\gamma}$ is  injective, we get that $ 2\alpha_{s}(g+r-1)=2\alpha_{t}(g+r-1)$. Since the map $\hat{\mathbb{Z}}\rightarrow \hat{\mathbb{Z}}$ of multiplication by $n$ $(n\in \mathbb{Z}_{\geq 1})$ are injective, we obtain that $\alpha_{s}=\alpha_{t}$. Hence $\alpha_{s}$ (in other words, the isomorphism $\overline{\kappa(s)}\xrightarrow{\sim} \overline{\kappa(s)}$ induced by $\tilde{f}_{\Phi_{s}}^{0}$) does not depend on $s$. 
 We write $\alpha$ instead of $\alpha_{s}$. 
Set $d_{\zeta_{1}(s)}:=[\kappa(\zeta_{1}(s)):\mathbb{F}_{p}](=[\kappa(\zeta_{2}(s)):\mathbb{F}_{p}])$.      
Since  $\zeta_{1}\circ \text{Fr}_{S}^{n_{1}}= \zeta_{2}\circ \text{Fr}_{S}^{n_{2}}$ and $ (\zeta_{1}|_{s})\circ\text{Fr}_{\kappa(s)}^{n_{1,s}}= (\zeta_{2}|_{s})\circ\text{Fr}_{\kappa(s)}^{n_{2,s}}$, we obtain that  $ (\zeta_{1}|_{s})\circ\text{Fr}_{\kappa(s)}^{n_{1}+n_{2,s}}=(\zeta_{2}|_{s})\circ\text{Fr}_{\kappa(s)}^{n_{2}+n_{2,s}}= (\zeta_{1}|_{s})\circ\text{Fr}_{\kappa(s)}^{n_{2}+n_{1,s}}$.  
Hence $ n_{2}-n_{1}\equiv n_{2,s}-n_{1,s}\equiv\alpha$ (mod $d_{\zeta_{1}(s)}$) follows. 
By a theorem of Chevalley, $\zeta_{1}(S)(=\zeta_{2}(S))$ is constructible in $\mathcal{M}_{g,r}[N]$, hence contains a non-empty open subset $T$ of $\overline{\zeta_{1}(S)}$. 
As $S$ is irreducible, so is $T$, and we regard $T$ as a reduced subscheme of $\mathcal{M}_{g,r}[N]$. 
(Note that $\text{dim}(T)>0$, since $U_{1}$  (hence, a fortiori, $(X_{1},(\psi_{1,j,\eta})_{1\leq j\leq r},\theta_{\eta})$) does not descend to a curve over $\overline{\mathbb{F}}_{p}$.) 
Now, applying Lemma \ref{dzetaslem} to this $T$, we obtain that $n_{2}-n_{1}=\alpha$, since $\alpha$ does not depend on $s$. 
 By definition of $\alpha(=\alpha_{s})$, we have  that $n_{2,s}-n_{1,s}\equiv \alpha$ (mod $[\kappa(s):\mathbb{F}_{p}])$. 
Hence $n_{2,s}-n_{1,s}\equiv \alpha\equiv n_{2}-n_{1}$ (mod $[\kappa(s):\mathbb{F}_{p}]$) follows.
Thus, $(\rho^{n_{2}}_{\mathcal{U}_{2,s}})^{-1} \circ f_{{\Phi}_{s}}\circ \rho^{n_{1}}_{\mathcal{U}_{1,s}}$ is a $\kappa(s)$-isomorphism. 
Since  $\mathcal{M}_{g,r}[N]$ is fine by Lemma \ref{moduli4},   there is at most  one element of  the set $\Isom_{\kappa(s)}((\mathcal{X}_{1,s},(\psi_{1,j,s}:\text{Spec}(\kappa(s))\rightarrow \mathcal{E}_{1,s})_{1\leq j\leq r}, \theta_{1,s})(n_{1}),(\mathcal{X}_{2,s},(\psi_{2,j,s}:\text{Spec}(\kappa(s))\rightarrow \mathcal{E}_{2,s})_{1\leq j\leq r}, \theta_{2,s})(n_{2}))$. 
This implies  that $f^{S}_{\Phi,s}= (\rho^{n_{2}}_{\mathcal{U}_{2,s}})^{-1} \circ f_{{\Phi}_{s}}\circ \rho^{n_{1}}_{\mathcal{U}_{1,s}}$. Hence  $f_{\Phi}^{S}$ satisfies (\dag). 
When $p=0$, we can prove that $f_{\Phi}^{S}$ satisfies (\dag) for every  $s\in S^{\text{cl}}$ in a similar way to the case that $p>0$ and  $n_{1}=n_{2}=0$.  
More precisely, let $s\in S^{\text{cl}}$. Take $t\in S^{\text{cl}}$ such that $p_{s}\neq p_{t}$, where $p_{s}:=\text{ch}(\kappa(s))$, $p_{t}:=\text{ch}(\kappa(t))$. Define $\alpha_{s}\in\hat{\mathbb{Z}}$, $\beta_{s}\in(\hat{\mathbb{Z}}^{\text{pro-}p'_{s}})^{\times}$, $\alpha_{t}\in\hat{\mathbb{Z}}$, $\beta_{t}\in(\hat{\mathbb{Z}}^{\text{pro-}p'_{t}})^{\times}$ as in the case that $p>0$. Then we obtain that $(p_{s}^{\alpha_{s}})^{2(g+r-\epsilon)}=\beta_{s}=\beta_{t}=(p_{t}^{\alpha_{t}})^{2(g+r-\epsilon)}$ in $(\hat{\mathbb{Z}}^{\text{pro-}p'_{s},\text{pro-}p'_{t}})^{\times}$, from which $\alpha_{s}(=\alpha_{t})=0$. The rest of the proof for $p>0$ works with $\alpha=0$. (See also the proof of \cite{Ta1997} Claim (6.8).)\par 

Finally, we construct $f_{\Phi}^{S}$ in general. There exists a connected  finite Galois covering $S'$ of $S$  such that  $(\mathcal{X}'_{1}, \mathcal{E}'_{1}):=(\mathcal{X}_{1},\mathcal{E}_{1})\times_{S}S'$ satisfies  the assumptions (i)(ii) above, where  $\mathcal{U}'_{1}:=\mathcal{X}'_{1},-\mathcal{E}'_{1}$.  Let $(\psi'_{1,j}:S'\rightarrow \mathcal{E}'_{1})_{1\leq j\leq r}$ be the  disjoint union  of ordered sections and  $\theta'_{1}$ the level $N$ structure on $\mathcal{U}_{1}'/S'$. Set $(\mathcal{X}'_{2}, \mathcal{E}'_{2}):=(\mathcal{X}_{2},\mathcal{E}_{2})\times_{S}S'$, $ \mathcal{U}'_{2}:=\mathcal{U}_{2}\times_{S}S'$,  $L:=K(S')$, and $\Phi_{L}:=\Phi\mid_{\Pi_{U_{1,L}}^{(m)}}$.  By the arguments in the case that we assume (i)(ii),   $\mathcal{E}'_{2}$ is also a  disjoint union  of sections  $(\psi'_{2,j}:S'\rightarrow \mathcal{E}'_{2})_{1\leq j\leq r}$ over $S'$  and there exists   a level $N$ structure  $\theta'_{2}$ on $\mathcal{U}'_{2}/S'$  such that     $\Phi_{L}$ induces  a unique isomorphism $(\mathcal{X}'_{1},(\psi'_{1,j})_{1\leq j\leq r},\theta'_{1})\rightarrow  (\mathcal{X}'_{2},(\psi'_{2,j})_{1\leq j\leq r},\theta'_{2}) $ (resp. $(\mathcal{X}'_{1},(\psi'_{1,j})_{1\leq j\leq r},\theta'_{1})(n_{1})\rightarrow  (\mathcal{X}'_{2},(\psi'_{2,j})_{1\leq j\leq r},\theta'_{2})(n_{2})$ for some $n_{1}, n_{2}\in\mathbb{Z}_{\geq 0}$ satisfying $n_{1}n_{2}=0$)  over $S'$, which induces an isomorphism $f_{\Phi_{L}}^{S'}: \mathcal{U}'_{1}\rightarrow \mathcal{U}'_{2}$ (resp. $f_{\Phi_{L}}^{S'}: \mathcal{U}'_{1}(n_{1})\rightarrow \mathcal{U}'_{2}(n_{2})$) over $S'$.  
 Let  $\rho$  be an element of $\text{Aut}(S'/S)\ (\xrightarrow{\sim}\text{Gal}(L/k))$.  Since  $(\mathcal{X}'_{1}, \mathcal{E}'_{1})$ satisfies  the assumption (i), the images of $\Phi_{L}$ and $\rho^{-1}\circ\Phi_{L}\circ \rho$ in $\Isom_{G_{L}}(\overline{\Pi}_{X_{1,L}}^{1}/N, \overline{\Pi}_{X_{2,L}}^{1}/N)$  are the same.  Hence $\rho^{-1}\circ f_{\Phi_{L}}^{S'}\circ \rho$ also preserves the level $N$ structures $\theta'_{1}$, $\theta'_{2}$ (resp. $\theta'_{1}(n_{1})$, $\theta'_{2}(n_2)$).  Since $\mathcal{E}'_{i}$ is a  disjoint union  of sections, we obtain that the action $G_{L}\curvearrowright E_{i}(k^{\text{sep}})$ is  trivial. 
Hence   $\rho^{-1}\circ f_{\Phi_{L}}^{S'}\circ \rho$ also preserves the orders of $(\psi'_{1,j})_{1\leq j\leq r}$, $(\psi'_{2,j})_{1\leq j\leq r}$ (resp.  $(\psi'_{1,j}(n_{1}))_{1\leq j\leq r}$, $(\psi'_{2,j}(n_{2}))_{1\leq j\leq r}$).  
Since  $\mathcal{M}_{g,r}[N]$ is fine,  $\rho^{-1}\circ f_{\Phi_{L}}^{S'}\circ \rho= f_{\Phi_{L}}^{S'}$ follows.  Considering all $\rho\in \text{Aut}(S'/S)$, we get an isomorphism  $f_{\Phi}^{S}: \mathcal{U}_{1}\xrightarrow[S]{\sim}\mathcal{U}_{2}$ (resp. $f_{\Phi}^{S}: \mathcal{U}_{1}(n_{1})\rightarrow \mathcal{U}_{2}(n_{2})$) by Galois descent.  This isomorphism $f_{\Phi}^{S}$ satisfies the condition (\dag), since $f_{\Phi_{L}}^{S'}$ satisfies the condition (\dag) and, for any $s'\in S'^{\text{cl}}$,   $f_{{\Phi}_{s}}\circ a_{1}=a_{2}\circ f_{\Phi_{L,s'}}$  follows  by Lemma \ref{prop1.14}, where $s$ stands for the image of $s'$ in $S$, $f_{\Phi_{L,s'}}$ stands for the image of  $\Phi_{L}$ by the composite of the maps  $ \Isom_{G_{L}}(\Pi_{U_{1,L}}^{(m)},\Pi_{U_{2,L}}^{(m)})\rightarrow \text{Isom}_{G_{\kappa(s')}}(\Pi^{(m-2)}_{\mathcal{U}'_{1,s'}},\Pi^{(m-2)}_{\mathcal{U}'_{2,s'}})\rightarrow\Isom(\mathcal{U}'_{1,s'},\mathcal{U}'_{2,s'})$, and $a_{i}:\mathcal{U}'_{i,s'}\rightarrow \mathcal{U}_{i,s}$ (resp. $a_{i}:\mathcal{U}'_{i,s'}(n_{i})\rightarrow \mathcal{U}_{i,s}(n_{i})$) stands for the natural morphism. Thus, $f_{\Phi}^{S}$ is the desired ismomorphism.
 \end{proof}

\begin{theorem}[Relative weak bi-anabelian result over finitely generated fields]\label{fingeneGCweak}
Assume that $k$ is finitely generated over the prime field, and that  $U_{1}$ is affine hyperbolic (see Notation of section \ref{sectionfinitelygene}). Assume that $U_{1,\overline{k}}$ does not  descend to a curve over  $\overline{\mathbb{F}}_{p}$  when $p>0$.   Assume that  $m$ satisfies
\begin{equation*}
\begin{cases}
\ \ m\geq  4 & \ \ (\text{if }r_{1}\geq  3 \text{ and }  (g_{1},r_{1})\neq (0,3), (0,4))\\
\ \ m\geq 5 & \ \ (\text{if }r_{1}<3 \text{ or }  (g_{1},r_{1})= (0,3), (0,4)).
\end{cases}
\end{equation*}
 Then  the following holds.
\begin{equation*}
\Pi^{(m)}_{U_{1}}\xrightarrow[G_{k}]{\sim} \Pi^{(m)}_{U_{2}}\Longleftrightarrow U_{1}\xrightarrow{\sim}U_{2} \text{ in }\mathfrak{S}_{k}
\end{equation*}
\end{theorem}
\begin{proof}
The implication $\Leftarrow$ is clear by Remark \ref{fndisoms}. We show the implication $\Rightarrow$.  Since we can take   a (sufficiently small)  integral regular scheme $S$ of finite type over $\text{Spec}(\mathbb{Z})$ with function field $k$ such  that there exists an affine hyperbolic curve of type $(g_{i},r_{i})$ over $S$ whose generic fiber is isomorphic to (and identified with) $(X_{i},E_{i})$ for $i=1,2 $. Hence the assertion follows from Lemma \ref{stxlemma2fin}.
\end{proof}

%%%%%%%%%%%%%%%%%%%%%%%%%%%%%%%%%%%%%%%%%%%%%%%%%%%%%%%%%%%%%%%%%%%%%%%%%%%%%%%%%%%%%%%%%
%%%%%%%%%%%%%%%%%%%%%%%%%%%%%%%%%%%%%     %%%%%%%%%%%%%%%%%%%%%%%%%%%%%%%%%%%%%%%%%%%%%%%%
%%%%%%%%%%%%%%%%%%%%%%%%%%%%%%%%%%%   %%  %%%%%%%%%%%%%%%%%%%%%%%%%%%%%%%%%%%%%%%%%%%%%%%%%
%%%%%%%%%%%%%%%%%%%%%%%%%%%%%%%%%     %%%  %%%%%%%%%%%%%%%%%%%%%%%%%%%%%%%%%%%%%%%%%%%%%%%%%%
%%%%%%%%%%%%%%%%%%%%%%%%%%%%%%%%     %%%%  %%%%%%%%%%%%%%%%%%%%%%%%%%%%%%%%%%%%%%%%%%%%%%%%%%
%%%%%%%%%%%%%%%%%%%%%%%%%%%%%%%     %%%%%  %%%%%%%%%%%%%%%%%%%%%%%%%%%%%%%%%%%%%%%%%%%%%%%%%%
%%%%%%%%%%%%%%%%%%%%%%%%%%%%%%%                        %%%%%%%%%%%%%%%%%%%%%%%%%%%%%%%%%%%%%%%%%%%%%
%%%%%%%%%%%%%%%%%%%%%%%%%%%%%%%%%%%%%%%   %%%%%%%%%%%%%%%%%%%%%%%%%%%%%%%%%%%%%%%%%%%%%%
%%%%%%%%%%%%%%%%%%%%%%%%%%%%%%%%%%%%%%%  %%%%%%%%%%%%%%%%%%%%%%%%%%%%%%%%%%%%%%%%%%%%%%%%%%
%%%%%%%%%%%%%%%%%%%%%%%%%%%%%%%%%%%%%%%%%%%%%%%%%%%%%%%%%%%%%%%%%%%%%%%%%%%%%%%%%%%%%%%%%

\subsection{The strong  bi-anabelian results over finitely generated  fields} \label{subsectionfingenestrong}
\hspace{\parindent}In this subsection, we show the strong bi-anabelian  $m$-step solvable  Grothendieck conjecture for affine hyperbolic curves over a field finitely generated over the prime field.\par

\begin{lemma}\label{4.9lem}
Assume that $U_{1}$ is hyperbolic. Assume that $U_{1,\overline{k}}$ does not  descend to a curve over  $\overline{\mathbb{F}}_{p}$  when $p>0$.   
\begin{enumerate}[(1)]
\item  The natural  map  $u: \Isom_{\mathfrak{S}_{k^{\text{sep}}}/\mathfrak{S}_{k}}(\tilde{U}^{m}_{1}/U_{1},\tilde{U}^{m}_{2}/U_{2})\rightarrow \Isom_{\mathfrak{S}_{k}}(U_{1},U_{2})$  is surjective. 
Further, for $(\tilde{t},t)\in \Isom_{\mathfrak{S}_{k^{\sep}}/\mathfrak{S}_{k}}(\tilde{U}^{m}_{1}/U_{1},\tilde{U}^{m}_{2}/U_{2})$, the equality $u^{-1}u((\tilde{t},t))=\text{Aut}_{\mathfrak{S}_{k^{\sep}}/\mathfrak{S}_{k^{\sep}},\text{id}}(\tilde{U}_{2}^{m}/U_{2,k^{\text{sep}}})\cdot\tilde{t}\ (=\text{Aut}_{\mathfrak{S}_{k^{\sep}}/\mathfrak{S}_{k},\text{id}}(\tilde{U}_{2}^{m}/U_{2})\cdot\tilde{t})$ holds, where 
$\text{Aut}_{\mathfrak{S}_{k^{\sep}}/\mathfrak{S}_{k^{\sep}},\text{id}}(\tilde{U}_{2}^{m}/U_{2,k^{\text{sep}}}) $ and $\text{Aut}_{\mathfrak{S}_{k^{\sep}}/\mathfrak{S}_{k},\text{id}}(\tilde{U}_{2}^{m}/U_{2})$ stand for the kernel of the natural maps $\text{Aut}_{\mathfrak{S}_{k^{\sep}}/\mathfrak{S}_{k^{\sep}}}(\tilde{U}_{2}^{m}/U_{2,k^{\text{sep}}})\rightarrow \text{Aut}_{\mathfrak{S}_{k^{\sep}}}(U_{2,k^{\text{sep}}})$ and  $\text{Aut}_{\mathfrak{S}_{k^{\sep}}/\mathfrak{S}_{k}}(\tilde{U}_{2}^{m}/U_{2})$ $\text{Aut}_{\mathfrak{S}_{k}}(U_{2})$, respectively. 
\item  Let $n\in \mathbb{Z}_{\geq 0}$ be an integer satisfying $m> n$.  Then the image of the natural map 
$\Isom_{\mathfrak{S}_{k^{\text{sep}}}/\mathfrak{S}_{k}}(\tilde{U}^{m-n}_{1}/U_{1},\tilde{U}^{m-n}_{2}/U_{2})$\\%はみ出し防止のため改行
$\rightarrow \Isom_{G_{k}}(\Pi_{U_{1}}^{(m-n)},\Pi_{U_{2}}^{(m-n)})$  (defined in Remark \ref{fndisoms})  is contained in $\Isom^{(m)}_{G_{k}}(\Pi_{U_{1}}^{(m-n)},\Pi_{U_{2}}^{(m-n)})$. 

\item   Let $n\in \mathbb{Z}_{\geq 0}$ be an integer satisfying $m> n$. Consider the following commutative  diagram.
\begin{equation}\label{isomequf}
\vcenter{\xymatrix@R=15pt{
\Isom_{\mathfrak{S}_{k^{\text{sep}}}/\mathfrak{S}_{k}}(\tilde{U}^{m-n}_{1}/U_{1},\tilde{U}^{m-n}_{2}/U_{2})\ar@{->>}[d]^{u}\ar[r] & \Isom^{(m)}_{G_{k}}(\Pi_{U_{1}}^{(m-n)},\Pi_{U_{2}}^{(m-n)})\ar@{->>}[d]\\
\Isom_{\mathfrak{S}_{k}}(U_{1},U_{2}) \ar[r]^-{}& \Isom_{G_{k}}^{(m)}(\Pi_{U_{1}}^{(m-n)},\Pi_{U_{2}}^{(m-n)})/\text{Inn}(\overline{\Pi}_{U_{2}}^{m-n}).
}}\end{equation}
Here, $\text{Inn}(\overline{\Pi}^{m}_{U_{i}})$ is the group of inner automorphisms of  $\Pi_{U_{i}}^{(m)}$ induced by elements of $\overline{\Pi}^{m}_{U_{i}}$.  Then  the upper horizontal map  of (\ref{isomequf}) is injective (resp. surjective) if and only if the lower  horizontal map of (\ref{isomequf}) is injective (resp. surjective).  (Remark that the  lower horizontal map is injective by Lemma \ref{injelemma4-2}.)
\end{enumerate}
\end{lemma}

\begin{proof}
(1) When $p=0$, the assertion follows from   \cite{Ta1997} Lemma (4.1)(ii). We assume  that $p>0$.  First, we show the surjectivity of $u$. Let $t$ be an element of $\Isom_{k}(\mathcal{Q}_{k}(U_{1}),\mathcal{Q}_{k}(U_{2}))$. Then, by Lemma \ref{example4.7}(1), there exist $n_{1}$, $n_{2}\in \mathbb{Z}_{\geq 0}$ and  $T\in$$\Isom_{k}(U_{1}(n_{1}),U_{2}(n_{2}))$ such that $\mathcal{Q}_{k}(T)=t$.   Since the natural map $\Isom_{k^{\sep}/k}(\widetilde{U_{1}(n_{1})}^{m}/U_{1}(n_{1}),\widetilde{U_{2}(n_{2})}^{m}/U_{2}(n_{2}))$ $\rightarrow\Isom_{k}(U_{1}(n_{1}),U_{2}(n_{2}))$ is clearly surjective, we obtain an element  $(\tilde{T},T)\in$$\Isom_{k^{\sep}/k}(\widetilde{U_{1}(n_{1})}^{m}/U_{1}(n_{1}),\widetilde{U_{2}(n_{2})}^{m}/U_{2}(n_{2}))$.  We know that $\widetilde{U_{i}(n_{i})}^{m}=\tilde{U}_{i}^{m}(n_{i})$. Thus, $(\mathcal{Q}_{k^{\text{sep}}}(\tilde{T}),\mathcal{Q}_{k}(T))\in\text{Isom}_{k^{\text{sep}}/k}(\mathcal{Q}_{k^{\text{sep}}}(\tilde{U}^{m}_{1})/\mathcal{Q}_{k}(U_{1}), \mathcal{Q}_{k^{\text{sep}}}(\tilde{U}^{m}_{2})/\mathcal{Q}_{k}(U_{2})))$ satisfies $u((\mathcal{Q}_{k^{\text{sep}}}(\tilde{T}),\mathcal{Q}_{k}(T))=t$. The second assertion clearly follows from the definition.  \\
(2)  In a similar way to (1), we can prove that the natural map $\Isom_{\mathfrak{S}_{k^{\sep}}/\mathfrak{S}_{k}}(\tilde{U}^{m}_{1}/U_1,\tilde{U}^{m}_{2}/U_2)\rightarrow$\\$ \Isom_{\mathfrak{S}_{k^{\sep}}/\mathfrak{S}_{k}}(\tilde{U}^{m-n}_{1}/U_1,\tilde{U}^{m-n}_{2}/U_2)$ is surjective. This implies that the image of the natural map \\%はみ出し防止のため改行
 $\Isom_{\mathfrak{S}_{k^{\sep}}/\mathfrak{S}_{k}}(\tilde{U}^{m-n}_{1}/U_{1},\tilde{U}^{m-n}_{2}/U_{2})\rightarrow \Isom_{G_{k}}(\Pi_{U_{1}}^{(m-n)},\Pi_{U_{2}}^{(m-n)})$   is contained in $\Isom^{(m)}_{G_{k}}(\Pi_{U_{1}}^{(m-n)},\Pi_{U_{2}}^{(m-n)})$. \\
(3) We may assume that $\Isom_{\mathfrak{S}_{k}}(U_{1},U_{2})\neq \emptyset$.  Let $t$ be an element of $\Isom_{\mathfrak{S}_{k}}(U_{1},U_{2})$. By (1), we have an element  $(\tilde{t},t)\in u^{-1}(t)$  and  the equality $u^{-1}(t)=\text{Aut}_{\mathfrak{S}_{k^{\sep}}/\mathfrak{S}_{k^{\sep}},\text{id}}(\tilde{U}_{2}^{m-n}/U_{2,k^{\text{sep}}})\cdot\tilde{t}$. When $p>0$, since $\Isom_{\mathfrak{S}_{k}}(U_{1},U_{2})\neq \emptyset$, $U_{2,\overline{k}}$ also does not  descend to a curve over  $\overline{\mathbb{F}}_{p}$. Hence  we obtain that   $\overline{\Pi}_{U_{2}}^{m-n}\xleftarrow{\sim}\ilim{a\geq 0}\text{Aut}_{k^{\text{sep}}/k^{\text{sep}}}(\tilde{U}^{m-n}_{2}(a)/U_{2,k^{\text{sep}}}(a))\xrightarrow{\sim}\text{Aut}_{\mathfrak{S}_{k^{\sep}}/\mathfrak{S}_{k^{\sep}},\text{id}}(\tilde{U}_{2}^{m-n}/U_{2,k^{\text{sep}}})$ by Lemma \ref{example4.7}(3).  (Note that  $\delta_{U_{2},U_{2}}=0$.) When $p=0$, we have that $\overline{\Pi}_{U_{2}}^{m-n}\xleftarrow{\sim}\text{Aut}_{\mathfrak{S}_{k^{\sep}}/\mathfrak{S}_{k^{\sep}},\text{id}}(\tilde{U}_{2}^{m-n}/U_{2,k^{\text{sep}}})$.  The  assertion follows from  the isomorphism  $\overline{\Pi}_{U_{2}}^{m-n}\xleftarrow{\sim}\text{Aut}_{\mathfrak{S}_{k^{\sep}}/\mathfrak{S}_{k^{\sep}},\text{id}}(\tilde{U}_{2}^{m-n}/U_{2,k^{\text{sep}}})$ and Proposition \ref{center}(1).
\end{proof}

\begin{lemma}\label{lemmaaaaa}
Assume that $k$ is finitely generated over the prime field,  and  that $U_{1}$ is affine hyperbolic. Assume that  $U_{1,\overline{k}}$ does not  descend to a curve over  $\overline{\mathbb{F}}_{p}$ when $p>0$.   Let $n\in\mathbb{Z}_{\geq 0}$ be an integer satisfying  $m\geq  n$.  Let $H_1$, $H'_{1}$ be open subgroups of  $\Pi^{(m)}_{U_{1}}$ that  satisfy   $\overline{\Pi}_{U_{1}}^{[m-n]}/\overline{\Pi}_{U_{1}}^{[m]} \subset H'_{1}\subset H_{1}$.  We assume that $(n,g(U_{H_{1}}),r(U_{H_{1}}))$ and $(n,g(U_{H'_{1}}),r(U_{H'_{1}}))$ satisfy the assumption for  $(m,g_{1},r_{1})$ in  Theorem \ref{fingeneGCweak}.  Let   $\Phi: \Pi^{(m)}_{U_{1}}\xrightarrow{\sim}\Pi^{(m)}_{U_{2}}$ be an isomorphism, $H_{2}:=\Phi(H_{1})$, and  $H_{2}':=\Phi(H_{1}')$. Then   the  following diagram is commutative in $\mathfrak{S}_{k}$.
\begin{equation}\label{oooo}\vcenter{
\xymatrix@C=60pt@R=15pt{
U_{1,H_{1}} \ar[r]^-{\phi} & U_{2,H_{2}}\\
U_{1,H'_{1}} \ar[r]^{\phi'}\ar[u]  & U_{2,H'_{2}}\ar[u]. 
}
}
\end{equation} 
Here, $\phi$ (resp. $\phi'$) stands for the isomorphism in $\mathfrak{S}_{k}$  induced by the isomorphism $H_{1}^{(n)}\xrightarrow{\sim}H_{2}^{(n)}$ (resp. $H'^{(n)}_{1}\xrightarrow{\sim}H'^{(n)}_{2}$) that is induced by  $\Phi\mid_{H_{1}}$ (resp. $\Phi\mid_{H'_{1}}$) by using Lemma \ref{stxlemma2fin}. 
\end{lemma}

\begin{proof}
By Proposition \ref{severalinvrecoprop}, $U_{2}$ is  also affine hyperbolic.   Let $S$ be an   integral regular scheme of finite type over $\text{Spec}(\mathbb{Z})$ with function field $k$. By replacing $S$ with a suitable open subscheme  if necessary, we may assume that, for $i=1,2 $,  there exists a smooth  curve $(\mathcal{X}_{i},\mathcal{E}_{i})$ of type $(g_{i},r_{i})$ over $S$ whose generic fiber is isomorphic to (and identified with) $(X_{i},E_{i})$.  Set $\mathcal{U}_{i}:=\mathcal{X}_{i}-\mathcal{E}_{i}$.  Let $L$ (resp. $L'$) be a finite extension of $k$ corresponding to the open subgroup $\text{Image}(H_{i}\rightarrow G_{k})$ (resp.  $\text{Image}(H'_{i}\rightarrow G_{k})$)  of $G_{k}$ and $T^{\ast}$ (resp. $T'^{\ast}$) the regular locus of the normalization of $S$ in $L$ (resp. $T^{\ast}$ in $L'$).  Let  $\mathcal{X}^{\ast}_{i,H_{i}}$, $\mathcal{U}^{\ast}_{i,H_{i}}$ (resp.  $\mathcal{X}^{\ast}_{i,H'_{i}}$, $\mathcal{U}^{\ast}_{i,H'_{i}}$) are the normalizations of $\mathcal{X}_{i}$, $\mathcal{U}_{i}$  in $K(U_{i,H_{i}})$ (resp. $K(U_{i,H'_{i}})$), respectively.  There exists an open subscheme $T\subset T^{\ast}$ such that the curves over $T$  induced by  the restrictions of $\mathcal{X}^{\ast}_{i,H_{i}}$, $\mathcal{U}^{\ast}_{i,H_{i}}$ over $T$  is smooth.  We write $\mathcal{X}_{i,H_{i}}$, $\mathcal{U}_{i,H_{i}}$ for the curves over $T$ induced by the restrictions of $\mathcal{X}^{\ast}_{i,H_{i}}$, $\mathcal{U}^{\ast}_{i,H_{i}}$  over $T$. Moreover, there exists  an open subscheme $T'\subset T'^{\ast}$ such that $T'^{\ast}\rightarrow T^{\ast}$ induces a morphism $T'\rightarrow T$ and that  the curves over $T'$  induced by  the restrictions of $\mathcal{X}^{\ast}_{i,H'_{i}}$, $\mathcal{U}^{\ast}_{i,H'_{i}}$ over $T'$  is smooth.  We write $\mathcal{X}_{i,H'_{i}}$, $\mathcal{U}_{i,H'_{i}}$ for  the curves over $T'$ induced by the restrictions of $\mathcal{X}^{\ast}_{i,H'_{i}}$, $\mathcal{U}^{\ast}_{i,H'_{i}}$  over $T'$.  Set $\mathcal{E}_{i,H_{i}}:=\mathcal{X}_{i,H_{i}}-\mathcal{U}_{i,H_{i}}$ (resp. $\mathcal{E}_{i,H'_{i}}:=\mathcal{X}_{i,H'_{i}}-\mathcal{U}_{i,H'_{i}}$).  Then,  by Lemma \ref{stxlemma2fin},  there exist isomorphisms $F:\mathcal{U}_{1,H_{1}}\xrightarrow[T]{\sim}  \mathcal{U}_{2,H_{2}}$, $F':\mathcal{U}_{1,H'_{1}}\xrightarrow[T']{\sim}  \mathcal{U}_{2,H'_{2}}$ (resp. non-negative integers $n_{1}$, $n_{1}'$, $n_{2}$, $n_{2}'$ with $n_{1}n_{2}=n_{1}'n_{2}'=0$ and  isomorphisms $F:\mathcal{U}_{1,H_{1}}(n_{1})\xrightarrow[T]{\sim}  \mathcal{U}_{2,H_{2}}(n_{2})$, $F':\mathcal{U}_{1,H'_{1}}(n'_{1})\xrightarrow[T']{\sim}  \mathcal{U}_{2,H'_{2}}(n'_{2})$) when $p=0$ (resp. $p>0$).  First, we assume that $p>0$. By symmetry, we may assume that $n_{1}'=0$ and set $n':=n'_{2}$.  Let $s'$ be a closed pint of $T'$ and $s$ the image of $s'$ by $T'\rightarrow T$.  Let $p_{i}:\mathcal{U}_{i,H'_{i}}\rightarrow \mathcal{U}_{i,H_{i}}\underset{T}\times T'$ be  the morphism induced by  $U_{i,H'_{i}}\rightarrow U_{i,H_{i}}$ By taking the  fiber at $s'$, we obtain the following diagram.
\begin{equation}\label{eq4.15.1}
\vcenter{\xymatrix@R=15pt@C=50pt{
\mathcal{U}_{1,H_{1},s}(n_{1})\ar@/^18pt/[rrr]^{F_{s}}\ar[r]_{\rho_{\mathcal{U}_{1,H_{1},s}}^{n_{1}}}&\mathcal{U}_{1,H_{1},s}\ar[r]^{\sim}\ar@{}[rdd]|{(A)}&\mathcal{U}_{2,H_{2},s}&\mathcal{U}_{2,H_{2},s}(n_{2})\ar[l]^{\rho_{\mathcal{U}_{2,H_{2},s}}^{n_{2}}}&\mathcal{U}_{2,H_{2},s}(N)\ar[l]^{\rho_{\mathcal{U}_{2,H_{2},s}(n_{2})}^{N-n_{2}}}\\
\mathcal{U}_{1,H_{1},s'}(n_{1})\ar[u]&\mathcal{U}_{1,H_{1},s'}\ar[u]&\mathcal{U}_{2,H_{2},s'}\ar[u]&&\mathcal{U}_{2,H_{2},s'}(N)\ar[u]\\
\mathcal{U}_{1,H'_{1},s'}(n_{1})\ar[u]^{p_{1,s'}(n_{1})}\ar[r]^{\rho_{\mathcal{U}_{1,H'_{1},s'}}^{n_{1}}}&\mathcal{U}_{1,H'_{1},s'}\ar[u]^{p_{1,s'}}\ar@/_18pt/[rr]_{F'_{s'}}\ar[r]^{\sim}&\mathcal{U}_{2,H'_{2},s'}\ar[u]_{p_{2,s'}}&\mathcal{U}_{2,H'_{2},s'}(n')\ar[l]_{\rho_{\mathcal{U}_{2,H'_{2},s'}}^{n'}}&\mathcal{U}_{2,H'_{2},s'}(N)\ar[l]_{\rho_{\mathcal{U}_{2,H'_{2},s'}(n')}^{N-n'}}\ar[u]_{p_{2,s'}(N)}\\
}
}
\end{equation}
Here, $N$ is an integer satisfying  $N\geq \text{max}\{n_{2}, n'\}$ and  the upper vertical arrows are natural projections.   By Lemma \ref{prop1.14} and the condition (\dag) in Lemma \ref{stxlemma2fin}, the quadrangle (A) in  (\ref{eq4.15.1}) is commutative.   Hence all morphisms $\mathcal{U}_{2,H'_{2},s'}(N)\rightarrow \mathcal{U}_{1,H_{1},s}(n_{1})$ appearing in  (\ref{eq4.15.1}) induce the same element of $\text{Hom}(\kappa(s),\kappa(s'))$.  In particular, we obtain that $N-n'-n_{1}\equiv N-n_{2}$ (mod $[\kappa(s):\mathbb{F}_{p}]$).  By considering infinitely many  closed points in $T'$, $n'+n_{1}=n_{2}$ follows.  Hence $(n_{1},n_{2})=(0,n')$ holds. We have that  any  closed point  of $\mathcal{U}_{1,H'_{1}}$ is contained in some fiber $\mathcal{U}_{1,H'_{1},s'}$ ($s'\in T'^{\text{cl}}$). Hence, by using  the commutativity of (\ref{eq4.15.1}) and  Lemma \ref{stxlemma1}, the following diagram is commutative.
\[
\xymatrix@C=60pt@R=15pt{
\mathcal{U}_{1,H_{1}} \ar[r]^-{F} & \mathcal{U}_{2,H_{2}}(n')\\
\mathcal{U}_{1,H'_{1}} \ar[r]^{F'}\ar[u]  & \mathcal{U}_{2,H'_{2}}(n')\ar[u].
}
\]
(Note that the integer $a$ in Lemma  \ref{stxlemma1} is zero in this case, since all morphisms are $S$-morphisms.)  Thus, (\ref{oooo}) is commutative  in $\mathfrak{S}_{k}$. When $p=0$, we can prove the assertion in a similar way  to the case that $p>0$ and $n_{1}=n_{2}=n'_{1}=n'_{2}=0$.
\end{proof}
\begin{definition}\label{m+nletter2}
Let $n\in\mathbb{Z}_{\geq 0}$ be an integer satisfying $m\geq n$.   We define  $\Isom_{G_{k}}^{(m)}(\Pi_{U_{1}}^{(m-n)},\Pi_{U_{2}}^{(m-n)})$ as the image of the  map $\Isom_{G_{k}}( \Pi_{U_{1}}^{(m)},\Pi_{U_{2}}^{(m)})\rightarrow \Isom_{G_{k}}(\Pi_{U_{1}}^{(m-n)},\Pi_{U_{2}}^{(m-n)})$.
\end{definition}

\begin{theorem}[Relative strong bi-anabelian result over finitely generated fields]\label{fingeneGCstrong}
Assume that $m\geq 5$, that  $k$ is finitely generated over the prime field, and that   $U_{1}$ is   affine  hyperbolic (see Notation of section \ref{sectionfinitelygene}). Assume that $U_{1,\overline{k}}$ does not  descend to a curve over  $\overline{\mathbb{F}}_{p}$ when $p>0$.  Let $n\in\mathbb{Z}_{\geq 4}$ be an integer satisfying $m> n$. Then  the map
\begin{equation*}
\Isom_{\mathfrak{S}_{k^{\sep}}/\mathfrak{S}_{k}}(\tilde{U}^{m-n}_{1}/U_1,\tilde{U}^{m-n}_{2}/U_2)\rightarrow \Isom_{G_{k}}^{(m)}(\Pi^{(m-n)}_{U_{1}},\Pi^{(m-n)}_{U_2})
\end{equation*}
 (defined in Lemma \ref{4.9lem}(2))  is bijective.
\end{theorem}
\begin{proof}
The injectivity follows from Lemma \ref{injelemma4-2} and  Lemma \ref{4.9lem}(3).  We show the  surjectivity.   We may assume that $\Isom_{G_{k}}(\Pi^{(m)}_{U_{1}},\Pi^{(m)}_{U_2})\neq \emptyset$.  Let   $\Phi$ be an element of $ \Isom(\Pi_{U_{1}}^{(m)},\Pi_{U_{2}}^{(m)})$.   Set $\mathcal{Q}_{1}:=\{H\overset{\op}\subset \Pi_{U_{1}}^{(m)} \mid \overline{\Pi}_{U_{1}}^{[m-n]}/\overline{\Pi}_{U_{1}}^{[m]}\subset H, r(U_{1,H})\geq 3\text{ and  } (g(U_{1,H}),r(U_{1,H}))\neq (0,3), (0,4)\}$.   Let   $H_{1}$ be an element of $\mathcal{Q}_{1}$ and  set $H_{2}:=\Phi(H_{1})$.   Let  $H'_{1}$ be an element of $\mathcal{Q}_{1}$ satisfying  $H'_{1}\subset H_{1}$ and set $H'_{2}:=\Phi(H_{1}')$. Then we  obtain the following  commutative diagram  in $\mathfrak{S}_{k}$ by Lemma \ref{lemmaaaaa}.
\begin{equation}\label{18.5.3}\vcenter{
\xymatrix@C=60pt{
U_{1,H_{1}} \ar[r]^-{\phi} & U_{2,H_{2}}\\
U_{1,H'_{1}} \ar[r]^{\phi'}\ar[u]  & U_{2,H'_{2}}\ar[u] 
}
}
\end{equation} 
Here, $\phi$ (resp. $\phi'$) stands for the isomorphism induced by the isomorphism $H_{1}^{(n)}\xrightarrow{\sim}H_{2}^{(n)}$ (resp. $H'^{(n)}_{1}\xrightarrow{\sim}H'^{(n)}_{2}$) that is induced by  $\Phi|_{H_{1}}$ (resp. $\Phi|_{H'_{1}}$) by using   Lemma \ref{stxlemma2fin}. Since $\mathcal{Q}_{i}$ is cofinal in the set of all open subgroups of $\Pi_{U_{1}}^{(m)}$ by Lemma \ref{gandr}, we obtain an isomorphism $\tilde{\mathcal{F}}(\Phi)\in \Isom_{\mathfrak{S}_{k^{\sep}}}(\tilde{U}^{m-n}_{1},\tilde{U}^{m-n}_{2})$.  The assumption ``$m\geq 5$'' implies  that $\Phi$ induces an isomorphism $F(\Phi)\in \Isom_{\mathfrak{S}_{k}}(U_1,U_2)$ by Lemma \ref{gandr}. By Lemma \ref{lemmaaaaa}, we have that $(\tilde{\mathcal{F}}(\Phi), F(\Phi))\in \Isom_{\mathfrak{S}_{k^{\sep}}/\mathfrak{S}_{k}}(\tilde{U}^{m-n}_{1}/U_1,\tilde{U}^{m-n}_{2}/U_2)$. Thus, we obtain a map
 $\mathcal{F}: \Isom_{G_{k}}(\Pi_{U_{1}}^{(m)},\Pi_{U_{2}}^{(m)})\rightarrow \Isom_{\mathfrak{S}_{k^{\sep}}/\mathfrak{S}_{k}}(\tilde{U}^{m-n}_{1}/U_1,\tilde{U}^{m-n}_{2}/U_2)$, and it suffices to show the commutativity of the following diagram.
\begin{equation*}\label{commutativediag3}
\vcenter{
\xymatrix@C=80pt{
 &\Isom_{G_{k}}(\Pi_{U_{1}}^{(m)},\Pi_{U_{2}}^{(m)})\ar[d]\ar[dl]_{\mathcal{F}}\\
\Isom_{\mathfrak{S}_{k^{\sep}}/\mathfrak{S}_{k}}(\tilde{U}^{m-n}_{1}/U_1,\tilde{U}^{m-n}_{2}/U_2)\ar[r]_{\Pi^{(m-n)}(\cdot)} &\Isom_{G_{k}}(\Pi^{(m-n)}_{U_{1}},\Pi^{(m-n)}_{U_2})
}
}
\end{equation*}
 Let   $\Phi^{m-n}$ be the image of $\Phi$ in $\Isom_{G_{k}}(\Pi_{U_{1}}^{(m-n)},\Pi_{U_{2}}^{(m-n)})$.   Let  $G\overset{\text{op}}\subset G_{k}$, $s\in \Sect(G,\Pi_{U_{1}}^{(m-n)})$  and $\mathcal{Q}_{s}:=\{H\overset{\text{op}}\subset \Pi_{U_{1}}^{(m-n)}\mid$ $r(U_{1,H})\geq3$,   $ (g(U_{1,H}),r(U_{1,H}))\neq (0,3), (0,4)$, and $s(G)\subset  H $$\}$.     Fix   $H\in \mathcal{Q}_{s}$.    Since $\mathcal{F}(\Phi)$ maps $U_{1,H}$ to $U_{2,\Phi^{m-n}(H)}$ by  construction of $\mathcal{F}$,  we obtain that $(\Pi^{(m-n)}\circ\mathcal{F})(\Phi)(H)=\Phi^{m-n}(H)$.   By Lemma \ref{gandr}, for any open subgroup $H'$ of $\Pi_{U_{1}}^{(m-n)}$ containing $s(G)$,  we can take an open  characteristic subgroup $\overline{H}''$ of $\overline{\Pi}_{U_{1}}^{m-n}$ that satisfies   $r(U_{1,\overline{H}''})\geq3$,   $(g(U_{1,\overline{H}''}),r(U_{1,\overline{H}''}))\neq (0,3), (0,4)$, and that $\overline{H}''\subset \overline{\Pi}_{U_{1}}^{m-n}\cap H'$. Hence $\mathcal{Q}_{s}$ is cofinal in the set of all open subgroups of $\Pi_{U_{1}}^{(m-n)}$ containing $s(G)$. This implies that  $s(G)=\underset{H\in\mathcal{Q}_{s}}\cap H$. Hence we obtain $(\Pi^{(m-n)}\circ\mathcal{F})(\Phi)(s(G))=\Phi^{m-n}(s(G))$.   Since the isomorphisms $(\Pi^{(m-n)}\circ\mathcal{F})(\Phi)$ and $\Phi^{m-n}$ are   $G_{k}$-isomorphisms, we obtain $(\Pi^{(m-n)}\circ\mathcal{F})(\Phi)(x)=\Phi^{m-n}(x)$ for $x\in s(G)$.   Note that we have $\Pi_{U_{1}}^{(m-n)}=\overline{\langle s(G)\ |\ G\overset{\text{op}}\subset G_{k},\ s\in \Sect(G,\Pi_{U_{1}}^{(m-n)})\rangle}$, since $k$ is a Hilbertian field (\cite{Fa2008} Proposition 13.4.1).   Therefore, we get $(\Pi^{(m-n)}\circ\mathcal{F})(\Phi)=\Phi^{m-n}$, as desired.

\end{proof}

\begin{corollary}
Let the  assumption and the notation be as in   Theorem \ref{fingeneGCstrong}. Then  the subset $\Isom^{(m)}_{G_{k}}(\Pi_{U_{1}}^{(m-n)},\Pi_{U_{2}}^{(m-n)})$ of $\Isom_{G_{k}}(\Pi_{U_{1}}^{(m-n)},\Pi_{U_{2}}^{(m-n)})$ depends only on $m-n$, not $m$.
\end{corollary}
\begin{proof}
The assertion follows from  Theorem \ref{fingeneGCstrong}.
\end{proof}

\begin{corollary}\label{fingeneGCcorinn2}
Let the  assumption and the notation be as in  Theorem \ref{fingeneGCstrong}. Then the natural map
\begin{equation*}
\Isom_{\mathfrak{S}_{k}}(U_1,U_2)\rightarrow \Isom^{(m)}_{G_{k}}(\Pi_{U_{1}}^{(m-n)},\Pi_{U_{2}}^{(m-n)})/\text{Inn}(\overline{\Pi}_{U_{2}}^{m-n})
\end{equation*}
  is  bijective.
\end{corollary}

\begin{proof}
The assertion follows from  Lemma \ref{4.9lem}(3) and Theorem \ref{fingeneGCstrong}.
\end{proof}

%%%%%%%%%%%%

%\bibliographystyle{plain}
%\bibliography{bibtex}

%%%%%%%%%%%%%%%%%%%%%%%%%%%%%%%%%%%%%%%%%%%%%%%%%%%%%%%%%%%%%%%%%%%%%%%%%

\end{document}